\documentclass[a4paper,11pt,reqno]{amsart}
\usepackage{latexsym, ifthen, xspace}
\usepackage{amsmath, amssymb, amsthm}
\usepackage{enumerate, calc, bbm}
\usepackage[dvipsnames]{xcolor}
\usepackage{paralist}
\usepackage[colorlinks=true, pdfstartview=FitV, linkcolor=blue,
citecolor=blue, urlcolor=blue, pagebackref=false]{hyperref}
\usepackage{microtype}
\microtypesetup{protrusion=true,expansion=true}

\usepackage[text={33pc,605pt},centering]{geometry}    %11pt
\usepackage{orcidlink}

\usepackage{graphicx}

\usepackage{mathtools}
\usepackage{tikz}
\usepackage{mathrsfs}

\usepackage{comment} 

\usetikzlibrary{decorations.markings}
\usetikzlibrary{shapes.geometric}
\tikzstyle{vertex}=[circle, draw, inner sep=0pt, minimum size=6pt]

\newtheorem{theorem}{Theorem}[section]
\newtheorem{lemma}[theorem]{Lemma}
\newtheorem{proposition}[theorem]{Proposition}
\newtheorem{assumption}[theorem]{Assumption}
\newtheorem{condition}[theorem]{Condition}
\newtheorem{corollary}[theorem]{Corollary}

\theoremstyle{definition}
\newtheorem{definition}[theorem]{Definition}
\newtheorem{example}[theorem]{Example}

\theoremstyle{remark}
\newtheorem{remark}[theorem]{Remark}

\numberwithin{equation}{section}

\DeclareMathAlphabet{\mathsl}{OT1}{cmss}{m}{sl}
\SetMathAlphabet{\mathsl}{bold}{OT1}{cmss}{bx}{sl}

\newcommand{\de}{\ensuremath{\delta}}

\newcommand{\si}{\ensuremath{\sigma}}

\newcommand{\om}{\ensuremath{\omega}}

\newcommand{\ve}{\ensuremath{\varepsilon}}

\newcommand{\vp}{\ensuremath{\varphi}}

\newcommand{\Ga}{\ensuremath{\Gamma}}

\newcommand{\Si}{\ensuremath{\Sigma}}

\newcommand{\Om}{\ensuremath{\Omega}}
\newcommand{\cB}{\ensuremath{\mathcal B}}
\newcommand{\cC}{\ensuremath{\mathcal C}}
\newcommand{\cD}{\ensuremath{\mathcal D}}

\newcommand{\cF}{\ensuremath{\mathcal F}}

\newcommand{\cI}{\ensuremath{\mathcal I}}

\newcommand{\cL}{\ensuremath{\mathcal L}}

\newcommand{\cP}{\ensuremath{\mathcal P}}
\newcommand{\cQ}{\ensuremath{\mathcal Q}}

\newcommand{\bbB}{\ensuremath{\mathbb B}}

\newcommand{\bbE}{\ensuremath{\mathbb E}}

\newcommand{\bbM}{\ensuremath{\mathbb M}}
\newcommand{\bbN}{\ensuremath{\mathbb N}}

\newcommand{\bbP}{\ensuremath{\mathbb P}}
\newcommand{\bbQ}{\ensuremath{\mathbb Q}}
\newcommand{\bbR}{\ensuremath{\mathbb R}}

\newcommand{\bbV}{\ensuremath{\mathbb V}}
\newcommand{\bbW}{\ensuremath{\mathbb W}}
\newcommand{\bbX}{\ensuremath{\mathbb X}}
\newcommand{\bbY}{\ensuremath{\mathbb Y}}
\newcommand{\bbZ}{\ensuremath{\mathbb Z}}

\DeclareMathAlphabet{\bbmsl}{U}{bbm}{m}{sl}

\newcommand{\N}{\mathbb{N}}
\newcommand{\R}{\mathbb{R}}
\newcommand{\Z}{\mathbb{Z}}
\newcommand{\E}{\mathbb{E}}
\newcommand{\p}{\mathbb{P}}
\newcommand{\X}{\mathbb{X}}

\newcommand{\md}{\ensuremath{\mathrm{d}}}
\newcommand{\mD}{\ensuremath{\mathrm{D}}}

\newcommand{\Norm}[2]{%
  \ensuremath{%
    \mathchoice{\big\lVert #1 \big\rVert}
    {\lVert #1 \rVert}
    {\lVert #1 \rVert}
    {\lVert #1 \rVert}_{\raisebox{-.0ex}{$\scriptstyle #2$}}
  }
}

\DeclareMathOperator{\prob}{\mathbb{P}} %law random environment
\DeclareMathOperator{\PP}{\mathbf{P}} %law random environment conditioned on origin in infinite cluster

\DeclareMathOperator{\mean}{\mathbb{E}}
\DeclareMathOperator{\EE}{\mathbf{E}}

\newcommand{\ldef}{\ensuremath{\mathrel{\mathop:}=}}
\newcommand{\rdef}{\ensuremath{=\mathrel{\mathop:}}}

\newcommand{\simom}[2]{#1 \overset{  \omega  }\sim  #2}
\newcommand{\simomxy}{\simom{x}{y}}

\newcommand{\indicator}{\mathbbm{1}}

\newcommand{\indizero}{\indicator_{\{0 \in \mathcal{C}_\infty(\omega)\}}}

\newcommand{\indizeroo}{\indicator_{\{0 \in \mathcal{C}_\infty(\omega)\}}}

\newtheorem{prop}[theorem]{Proposition}
\newtheorem{coro}[theorem]{Corollary}
\allowdisplaybreaks[2]
\expandafter\def\expandafter\UrlBreaks\expandafter{\UrlBreaks\do\+\do\&\do\=\do\?\do\_\do\#\do\%\do\-\do\:}

\begin{document}
\emergencystretch=1em

\title[Rough walks in random conductances]{Invariance Principles for Rough Walks in Random Conductances}

% Remove any unused author tags.

% author one information
\author{Johannes B\"aumler}
\address{University of California, Los Angeles}
\curraddr{520 Portola Plaza,
  Los Angeles, CA, USA}
\email{jbaeumler@math.ucla.edu}
\thanks{}

% author two information
\author{Noam Berger}
\address{Technische Universit\"at M\"unchen}
\curraddr{Boltzmannstraße 3, 85748 Garching}
\email{noam.berger@tum.de}
\thanks{}

% author three information
\author{Tal Orenshtein}
\address{Universit\`a degli Studi di Milano-Bicocca}
\curraddr{Via Roberto Cozzi 55,
  20125 Milano, Italy}
\email{tal.orenshtein@unimib.it}
\thanks{}

% author four information
\author[M.\ Slowik]{Martin Slowik \orcidlink{0000-0001-5373-5754}}
\address{University of Mannheim}
\curraddr{Mathematical Institute, B6, 26, 68159 Mannheim}
\email{slowik@math.uni-mannheim.de}
\thanks{}



\date{\today}

\dedicatory{}

\begin{abstract}
We establish annealed and quenched invariance principles for random walks in random conductances lifted to the $p$-variation rough path topology, allowing for degenerate environments and long-range jumps. Our proof is based on a unified structural strategy where pathwise convergence is viewed as a natural upgrade of the classical theory. This approach decouples the martingale lift from terms involving the integrals with respect to the corrector and the quadratic covariations. In the quenched regime, we show that the existence of a stationary potential for the corrector with $2+\epsilon$ moments is sufficient to ensure the vanishing of the corrector in $p$-variation for any $p > 2$. This input, combined with our structural framework, provides a direct and modular pathway to rough path convergence. We further provide a transfer lemma to construct this potential from spatial moment bounds. While presently verified in the literature primarily for nearest-neighbor settings, our formulation isolates the exact analytic input required for pathwise convergence in more general environments.  
  \end{abstract}

\maketitle

%\tableofcontents

\section{Introduction}

The purpose of this work is to identify and develop a minimal and structural probabilistic mechanism that yields rough path invariance principles for random walks in random conductances. We treat both the annealed and quenched regimes and allow for degenerate environments and long-range jumps. Our results show that rough convergence is governed entirely by the martingale–corrector decomposition and that the second-order structure of the limit is encoded by the energy of the corrector.

Beyond the extension of classical invariance principles to the rough path topology, the primary aim of this work is to present a unifying strategy where pathwise convergence is viewed as a natural upgrade of the classical theory. In the annealed setting, reversibility furnishes a forward–backward representation that enables the control of the corrector lift. In the quenched setting, this structure is replaced by the existence of a stationary potential for the corrector having $2+\epsilon$ moments. In both regimes, we employ a decoupling principle that isolates the martingale lift from the corrector-dependent integral terms, bypassing the treatment of joint convergence.

This modularity allows us to identify that rough convergence is essentially governed by the same fundamental decomposition that appears in the classical theory, provided one employs the appropriate $p$-variation estimates. A key ingredient of the argument is the control of second-level quantities via martingale-transform inequalities, providing bounds intrinsic to rough path analysis. This yields precisely the variation control required for rough tightness while simplifying the identification of the limiting area anomaly.

The requirement of a stationary potential with $2+\epsilon$ moments in the quenched setting represents a technical refinement that aligns the rough path upgrade with the minimal integrability of the classical regime. Recent breakthroughs in quantitative homogenization by Gloria, Neukamm and Otto \cite{GloriaOtto2011}, Mourrat \cite{Mourrat2012}, Dario \cite{dario2018optimal}, and Andres and Neukamm \cite{AndresNeukamm2019} have established that the corrector often possesses all moments in dimensions $d \ge 3$; however, we demonstrate that apriori, such high-order regularity is not a prerequisite for pathwise scaling limits.   

By utilizing a refined ergodic theorem for the supremum norm and leveraging $p$-variation estimates of Lépingle–BDG type, we prove that $2+\epsilon$ moments suffice to control the second-level quantities. This suggests that the "all-moments" regularity often found in the literature, while available in specific settings, is not the fundamental threshold for rough tightness. Moreover, we show that these settings suffice to guarantee the limiting Brownian motion has a non-degenerate covariance matrix. 

To extend the applicability of the quenched result, we provide a transfer lemma: we show that sufficiently strong spatial moment bounds on the corrector, together with mild volume regularity, imply the existence of the required stationary potential. This formulation identifies the analytic input required for pathwise convergence.

Invariance principles in rough path topology naturally arise in the study of stochastic differential equations by approximating the Brownian motion driving the equation by another process. As was already observed in the 1965 work of Wong and Zakai, when considering the weak convergence in the Skorohod topology of the process alone \cite{eugene1965relation} the limit depends on the type of approximation. However, when considering a two-level object: the process and its iterated integrals (or alternatively its antisymmetric part, which is sometimes called its Lévy area) in an appropriate $p$-variation rough path space, continuity is obtained. In fact, it turns out that in the limit of the approximation the second level may yield an additional drift term, which is known as the area correction / anomaly.

There are in recent years a growing number of results in various fields aiming at proving invariance principles in the rough path topology, characterizing the area anomaly and when it is non-vanishing.
For fast-slow dynamical systems and deterministic approximations, Kelly \cite{kelly2016rough} and Kelly and Melbourne \cite{kelly2016smooth} proved invariance principles for approximations of SDEs, while Chevyrev, Friz, Korepanov, and Melbourne \cite{chevyrev2020superdiffusive} established superdiffusive rough limits.
In the context of stochastic averaging, Hairer and Li \cite{hairer2022generating} generated diffusions from systems driven by fractional Brownian motion, Friz and Kifer \cite{kifer2024almost} proved almost sure diffusion approximations, and Gottwald and Melbourne \cite{gottwald2024time} clarified the Lévy area correction for time-reversible dynamical systems.
The work by Engel, Friz, and Orenshtein \cite{EFO24} review various stationary examples of such invariance principles and present the way the different frameworks can be characterized in terms of a Kubo-Green formula for the covariance matrix and the area correction.
Further generalizations of the rough path framework to new stochastic processes include extensions to multi-dimensional Volterra processes (Gehringer, Li and Sieber \cite{gehringer2022functional}), a rough functional Breuer-Major theorem for Gaussian sequences (Elad Altman, Klose and Perkowski \cite{altman2026rough}), and foundational $p$-variation estimates for martingale transforms (Zorin-Kranich \cite{zorin2020weighted} and Friz and Zorin-Kranich \cite{friz2020rough}).
These tools have been increasingly applied to singular SPDEs, including periodic homogenization for the generalized Parabolic Anderson Model (Chen, Fehrman and Xu \cite{chen2026periodic}) and Langevin dynamics on fluctuating Helfrich surfaces (Djurdjevac, Kremp and Perkowski, \cite{djurdjevac2025rough}; see also the PhD theses of Kremp, 2022, and Gehringer, 2022).

Foundational discrete invariance principles in the rough path topology began with the rough Donsker theorem by Breuillard, Friz, and Huesmann \cite{breuillard2009random}, which was subsequently generalized to random walks and Lévy processes on Lie groups (Chevyrev \cite{chevyrev2018random}), general càdlàg semimartingales \cite{ChevyrevFriz2019}, and supported by characteristic function techniques for geometric rough paths (Chevyrev and Lyons \cite{chevyrev2016characteristic}).

The explicit deterministic area anomaly was subsequently isolated in discrete geometric settings by Lopusanschi and Simon for Markov chains on periodic graphs \cite{LS18} and hidden Markov walks \cite{lopusanschi2020area} (see also \cite{lopusanschi2018ballistic}), as well as by Ishiwata, Kawabi and Namba for non-symmetric walks on nilpotent covering graphs \cite{ishiwata2020central,ishiwata2021central}.

For random walks in random environment in the ballistic regime, the annealed result was proven using a delayed regenerative structure \cite{orenshtein2021rough,lopusanschi2018ballistic}.

% Finally, the annealed result in the random conductance model was proven by Deuschel, Orenshtein and Perkowski \cite{deuschel2021additive} for nearest-neighbor random walks in i.i.d.\ uniformly elliptic conductances as an application of a Kipnis-Varadhan type invariance principle for additive functionals of Markov processes in the rough path topology.

Finally, the annealed result in the random conductance model was established by Deuschel, Orenshtein and Perkowski \cite{deuschel2021additive} for nearest-neighbor random walks in i.i.d.\ uniformly elliptic conductances. That work treated the model as an application of a general Kipnis-Varadhan type invariance principle for additive functionals in the rough path topology. In this context, the present work can be viewed as a substantial advancement of that program. By moving beyond the theory of additive functionals to a dedicated structural framework, we are able to provide the first quenched result and remove the constraints of uniform ellipticity and nearest-neighbor jumps, providing a comprehensive treatment of Gaussian limits for the reversible regime.

\subsection{The model}
Consider an undirected graph with vertex set $\mathbb{Z}^{d}$, $d \geq 1$, and edge set $E = \{\{x,y\} : x,y \in \mathbb{Z}^{d},\; x \ne y\}$. In particular, each vertex of the graph has infinite degree. For $\omega \in \Omega = \left[0,\infty \right)^E$ and $e \in E$, we call $\omega(e)$ the \emph{conductance} of the edge $e$. Let $(\Omega, \cF) = ([0, \infty)^E, \cB([0,\infty)^{E})$ be the measure space which is equipped  with the Borel-$\si$-algebra. 

For any given $\omega \in \Omega$, the random walk $X = (X_{t})_{t \geq 0}$ is defined to be a continuous-time Markov chain on $\mathbb{Z}^{d}$ with generator, $L^{\omega}$, acting on bounded functions $f\colon \mathbb{Z}^{d} \to \bbR$ by 
\begin{align}\label{eq:process_generator}
  (L^{\om} f)(x)
  \;=\;
  \sum_{y \in \mathbb{Z}^{d}} \om(\{x,y\})\, \big( f(y) - f(x) \big)
  \;=\;
  \sum_{\substack{y \in \mathbb{Z}^{d}\\ \simomxy }} \om(\{x,y\})\, \big( f(y) - f(x) \big),
\end{align}
where $\simomxy$ denotes $\om(\{x,y\})>0$. 
We denote by $P_x^{\om}$ the law of the process $X$ on $\cD([0, \infty), \bbR^d)$ -- the space of $\bbR^d$-valued c\`adl\`ag functions on $[0, \infty)$ -- starting at time $0$ in $x$.  The expectation corresponding to $P_x^{\om}$ will be denoted by $E_x^{\om}$.  Notice that $X$ is \emph{reversible} with respect to the counting measure on $\mathbb{Z}^{d}$.

For $\omega \in \Omega$ and $x\in \mathbb{Z}^{d}$
we denote by $\mathcal{C}_{x}(\omega)$ the $\omega$-cluster of $x$, that is the connected-component of $x$ in the graph with vertex set $\mathbb{Z}^{d}$ whose edges are the strictly positive conductances of $\om$. More explicitly, $\mathcal{C}_{x}(\omega)$ contains $x$ together with all vertices $y\in \mathbb{Z}^{d}$ so that there is a path from $x$ to $y$ using only edges satisfying $\omega(e)>0$. We say that $x$ has an infinite $\omega$-cluster if the set $\mathcal{C}_{x}(\omega)$ is infinite. We write $\tau_x : \Omega \to \Omega$ for the translation of the environment by $x$, i.e. $(\tau_x\omega)(\left\{u,v\right\}) = \omega\left(\left\{u+x,v+x\right\}\right)$.
We also set
\begin{align}\label{weighteddegree}
  \mu^{\om}(x) \;\ldef\; \sum\nolimits_{y \in \mathbb{Z}^{d}} \om(\{x,y\}) \in [0,+\infty].
\end{align}
Henceforth, let $\prob$ be a probability measure on $(\Omega, \cF)$, and we write $\mean$ to denote the expectation with respect to $\prob$. For the rest of the paper we assume the following.  

\begin{assumption}\label{ass:general}
  The following conditions hold for $\prob$.
  \begin{enumerate}[(i)]
   \item There exists a unique infinite cluster $\prob$-a.s., which is denoted by $\mathcal{C}_{\infty}= \mathcal{C}_{\infty} \left( \omega \right)$. 
  \item $\prob$ is stationary and ergodic with respect to translations $\{\tau_x : x \in \mathbb{Z}^{d}\}$ of $\,\mathbb{Z}^{d}$.
  \item $\E\!\big[\mu^{\om}(0)\big] < \infty\,$ and $\,\mean\!\big[\sum_{x \in \mathbb{Z}^{d}} \om(\{0,x\})\, |x|^2 \big] < \infty$. 
      \end{enumerate}
\end{assumption}
 We are interested in the random walk on the infinite cluster. To do so, first avoid heavy notation by assuming without loss of generality that there exists a unique infinite cluster \emph{for every} $\omega\in\Om$. Then we set 
 \[
 \Omega_0 = \left\{\omega \in \Omega :  0 \in \mathcal{C}_\infty(\omega)\right\} \subset \Omega
 \]
 to be the set of all environments whose infinite cluster contains the origin.

 Conditions (i) and (ii) of Assumption \ref{ass:general} imply that $\prob\left( \Omega_0\right) = \prob\left( 0 \in \mathcal{C}_\infty \right) > 0 $. Hence, we can define the probability measure $\mathbf{P}: \mathcal{B}(\Omega) \to \left[0,1\right]$ by
\begin{equation*}
  \mathbf{P}(A) = \prob(A\;|\;\Omega_0).
\end{equation*}
We write $\mathbf{E}$ for the expectation under $\mathbf{P}$. 
Notice that by construction, for $\mathbf{P}$-a.e.\ $\omega$
the random walk starting at the origin, that is the continuous time Markov chain with the generator defined in \eqref{eq:process_generator}
and the initial condition $P_{0}^{\omega}(X_0 = 0)=1$, will stay on $\mathcal{C}_\infty$ at all times $P_{0}^{\omega}$-almost surely. 
\subsection{Harmonic embedding and the corrector}\label{subsec:harm}
To fix notation relevant in the main results we will now briefly recall the abstract construction of the corrector and the harmonic coordinates. 
\begin{definition}\label{def:cocycle_property}
  A measurable function, also called a random field, $\Psi\!: \Om \times \mathbb{Z}^{d} \to \bbR$ satisfies the cocycle property if for $\prob$-a.e. $\om$, it holds that
  \begin{align*}
    \Psi(\tau_x\om, y-x)
    \;=\;
    \Psi(\om, y) - \Psi(\om, x),
    \qquad \text{for } x,y\in\mathbb{Z}^d \text{ with } y\in\mathcal{C}_{x}(\omega).
  \end{align*}
  We write $L^2_\mathrm{cov}(\prob)$ to denote the set of all functions $\Psi\!: \Om \times \mathbb{Z}^{d} \to \bbR$ such that $\Psi$ satisfies the cocycle property and
  \begin{align*}
    \Norm{\Psi}{L_\mathrm{cov}^2(\prob)}^2
    \;\ldef\;
    \mean\!\Big[
    {\textstyle \sum_{x \in \mathbb{Z}^{d}}}\, \om(\{0, x\})\, |\Psi(\om,x)|^2
    \Big]
    \;<\;
    \infty.
  \end{align*}
\end{definition}
It can be checked that $L_{\mathrm{cov}}^2(\prob)$ is a Hilbert space (cf. \cite{BP07, MP07}).  Accordingly, we write
\begin{align*}
  \langle \Phi, \Psi \rangle_{L^2_{\mathrm{cov}}(\prob)}
  \;=\;
  \mean\!\Big[
  {\textstyle \sum_{x \in \mathbb{Z}^{d}}}\, \om(\{0, x\})\, \Phi(\om,x) \Psi(\om,x)
  \Big].
\end{align*}
to denote the scalar product between $\Psi, \Phi \in L^2_{\mathrm{cov}}(\prob)$.  Further, for a function $\vp\!: \Om \to \bbR$ we define a (horizontal) gradient $\mD \vp\!: \Om \times \mathbb{Z}^{d} \to \bbR$ by
\begin{align*}
  \mD \vp (\om,x)
  \;=\;
  \vp(\tau_x \om) - \vp(\om),
  \qquad x \in \mathbb{Z}^{d}.
\end{align*}
Obviously, if the function $\vp$ is bounded, then $\mD \vp \in L_{\mathrm{cov}}^2(\prob)$.  A \emph{local} function on $\Om$ is bounded and depends only on the value of $\om$ at a finite number of edges.  Following \cite{MP07}, the closure in $L^2_{\mathrm{cov}}(\prob)$ of the set of gradients of local functions is called $L^2_{\mathrm{pot}}$, whereas the orthogonal complement  of $L^2_{\mathrm{pot}}$ in $L^2_{\mathrm{cov}}(\prob)$ is called $L^2_{\mathrm{sol}}$.
Notice that by convention $L^2_{\mathrm{pot}}$ and $L^2_{\mathrm{sol}}$ will be always subspaces of $L_{\mathrm{cov}}^2(\prob)$, hence their notation lacks  mentioning the probability measure on $\Omega$.

\begin{remark}[Remark about convention of notation]\label{rem:conditional-l2cov-norm}
 We define $\Norm{\Psi}{L_{\mathrm{cov}}^2(\PP)}$ and $\langle \Phi, \Psi \rangle_{L^2_{\mathrm{cov}}(\PP)}$ analogously, with the expectation with respect to $\mean$ being replaced by the conditional expectation $\EE$. Throughout the paper, whenever we consider the norms under the conditional probability $\PP$ we shall state it explicitly. Otherwise, we write $L_{\mathrm{cov}}^2$, that is the probability measure $\prob$ is omitted.       
\end{remark}

The following remark will be useful when we shall move from random fields defined on $\Omega_0$ to random fields defined on $\Omega$ and vice versa.      
\begin{remark}[Remark about extension and restriction]\label{rem:extension-and-restriction}
 Every $\Psi\in L_{\mathrm{cov}}^2(\prob)$ can be considered as a function $\Psi\in L_{\mathrm{cov}}^2(\PP)$ by taking the restriction to $\Omega_0 \times \Z^d$. 
 Further, if $\Psi\in L_{\mathrm{pot}}^2$ then $\Psi\indizeroo\in L_{\mathrm{pot}}^2$, and analogously  
 if $\Psi\in L_{\mathrm{sol}}^2$ then $\Psi\indizeroo\in L_{\mathrm{sol}}^2$.
 Conversely, any $\Psi\!: \Om_0 \times \mathbb{Z}^{d} \to \bbR$ that satisfies the cocycle property from Definition \ref{def:cocycle_property} for all $x,y\in\cC_{\infty}$ can be  extended to a function $\bar\Psi\!: \Om \times \mathbb{Z}^{d} \to \bbR$ by the identification  
 $\bar\Psi=\Psi\indizeroo$. 
 With this identification, we consider $\Psi$ to be in $L_{\mathrm{cov}}^2(\prob)$, $L_{\mathrm{pot}}^2$ or $L_{\mathrm{sol}}^2$ whenever its extension $\bar\Psi$ is in $L_{\mathrm{cov}}^2(\prob)$,  $L_{\mathrm{pot}}^2(\prob)$ or $L_{\mathrm{sol}}^2$, respectively.

 \end{remark}
\begin{lemma}[{\cite[Prop.~3.7]{Bi11}}]
  For all $\Psi \in L^2_{\mathrm{sol}}$ and $\prob$-a.e.\ $\om$, we have
  \begin{align}\label{eq:harmonic}
    \big(L^{\om} \Psi(\om, \cdot)\big)(x) \;=\;  0
    \qquad \text{for all $x \in \mathbb{Z}^{d}$.}
  \end{align}
\end{lemma}
The \emph{increment field} $\Pi\!: \Om \times \mathbb{Z}^{d} \to \bbR^d$ is defined by $\Pi(\om, x) \ldef x$.  Write $\Pi^j$ for the $j$-th component of $\Pi$.  Since $\Pi^j$ satisfies trivially the cocycle property for all $\om \in \Om$ and $x, y \in \mathbb{Z}^{d}$ and, in view of Assumption~\ref{ass:general}(iii), $\|\Pi^j\|_{L^2_{\mathrm{cov}}} < \infty$, we have that $\Pi^j \in L^2_{\mathrm{cov}}$ for any $j = 1, \ldots, d$.  For any $j = 1, \ldots, d$, we define $\chi^j\!: \Om \times \mathbb{Z}^{d} \to \bbR$ to be the unique element in $L^2_{\mathrm{pot}}$ such that $\Phi^j \ldef \Pi^j - \chi^j \in L^2_{\mathrm{sol}}$.
\begin{definition}\label{def:corrector+matrix}
  The function $\Phi = (\Phi^1, \ldots, \Phi^d)\!: \Om \times \mathbb{Z}^{d} \to \bbR^d$ is called the \emph{harmonic coordinates}  and the function $\chi = (\chi^1, \ldots, \chi^d)\!: \Om \times \mathbb{Z}^{d} \to \bbR^d$ is called the \emph{corrector}.  Moreover, set
  \begin{align}\label{eq:def:Si+Ga}
    \Si^2
    \;\ldef\;
    \big( \langle \Phi^i, \Phi^j \rangle_{L^2_{\mathrm{cov}}(\PP)} \big)_{i,j = 1}^d
    \qquad \text{and} \qquad
    \Ga
    \;\ldef\;
    -\frac{1}{2}  \big( \langle \chi^i, \chi^j \rangle_{L^2_{\mathrm{cov}}(\PP)} \big)_{i,j = 1}^d.
  \end{align}
\end{definition}
\begin{remark}
  The non-degeneracy of $\Sigma^2$ follows under additional conditions, e.g $\mean[(1/\om(\{0,x\}))\indicator_{\om(\{0,x\})>0}] < \infty$ for all $x \in \mathbb{Z}^{d}$ with $|x| = 1$. Further sufficient conditions for non-degeneracy of $\Si^2$ is available in \cite[Lemma~5.5]{FHS17}.
  $\ell^{1}$-sublinearity of the corrector implies non-degeneracy of $\Sigma^{2}$, cf \cite{DNS18}. However, $\Gamma$ is only positive semidefinite. For instance, $\Gamma=0$ for the line model, that is whenever the conductances are constant in any direction but the directions are independent. Indeed, in this case the walk is already a martingale and so the corrector is identically zero.
\end{remark}

\subsection{Invariance principles for rough walks}
We are interested in proving an \emph{invariance principle} for the lifted random walk in the $p$-variation rough path topology.
For this purpose, let us first briefly recall the definition of the homogeneous $p$-variation rough path norm. 

Let $T>0$ and define 
\begin{align*}
  \Delta_{[0,T]}\coloneqq\{(s,t):0\le s<t\le T\}.
\end{align*}
Let $(E,|\cdot|_{E})$ be a normed space. Consider  
$\Xi : \Delta_{[0,T]} \to E$.
For $0<p<\infty$
we define the $p$-variation of $\Xi$ by
\begin{align}\label{eq:def_p-var-norm}
  \|\Xi\|_{p\text{-var}, [0,T]}
  \;: =\;
  \bigg(
  \sup_{\cP} \sum_{[s,t] \in \cP} |\Xi_{s,t}|_{E}^p
  \bigg)^{\!\!1/p}
  \;\in\;
  [0, +\infty],
\end{align}
where the supremum is taken over all finite partitions $\cP$ of $[0,T]$ and the summation is over all intervals $[s,t] \in \cP$. 

In the same way 
\begin{align}\label{eq:def_sup-var-norm}
  \|\Xi\|_{\infty\text{-var}, [0, T]} 
  \;\ldef\;  
  \sup_{(s,t) \in \Delta_{[0,T]}} |\Xi_{s,t}|_E
  \;\in\;
  [0, +\infty],
\end{align}
We also identify a path $\Xi:[0,T]\to E$ with a function of two variables $\Delta_{[0,T]} \to E$ by using the convention  
\begin{align*}
  \Xi_{s,t}:=\Xi_{t}-\Xi_{s}, \quad (s,t)\in\Delta_{[0,T]}.
\end{align*}
Using this convention, we also define the $p$-variation of $\Xi$, where $p\in(0,+\infty]$, as in (\ref{eq:def_p-var-norm}) or (\ref{eq:def_sup-var-norm}), respectively.
Note that for any $0 < p \leq q < \infty$, we have that
\begin{equation}\label{eq:inequalities on norms}
  \|\Xi\|_{\infty\text{-var},[0,T]} \leq \|\Xi\|_{q\text{-}\mathrm{var}, [0,T]} \leq \|\Xi\|_{p\text{-}\mathrm{var},[0,T]}.
\end{equation}
For a path $\Xi:[0,T]\to E$ denote by
\begin{align*}
  \|\Xi\|_{\mathrm{unif}, [0,T]}
  :=
  \sup_{t \in [0,T]} |\Xi_{t}|_E
\end{align*}
the standard uniform norm. 
We note that the $p$-variation norm $\|\cdot\|_{p\text{-var},[0,T]}$ of any constant path in $E$ is $0$ and so it is not a norm. However, we can define a norm on $\left\{  \Xi:[0,T]\to E\right\}$  by
\begin{align*}
  \|\Xi\|_{p\text{-var}, [0, T]} + | \Xi_0| .
\end{align*}
is a norm for every $p\in(0,\infty]$. In fact, the case $p=\infty$, this norm is equivalent to the uniform norm. Indeed, by the triangle inequality
\begin{align*}
  \|\Xi\|_{\mathrm{unif}, [0,T]}
  \le
  \|\Xi\|_{\infty\text{-var}, [0, T]}
  + | \Xi_0 |
  \le
  3  \|\Xi\|_{\mathrm{unif}, [0,T]},
\end{align*}
whenever all terms above are well-defined and finite.

Let $\cD([0,T], \bbR^d)$ be the collection of c\`adl\`ag functions $X$ from $[0, T]$ into the metric space $(\bbR^d, |\cdot|)$ equipped with the Euclidean norm.
Similarly,
let $\cD(\Delta_{[0,T]}, \bbR^{d \times d})$ be the collection of c\`adl\`ag functions $\bbmsl{X}$ from $\Delta_{[0,T]}$ into the metric space $(\bbR^{d \times d}, |\cdot|)$ equipped with the Euclidean operator norm, where we say that $\bbmsl{X}$ is c\`adl\`ag (or continuous) whenever so are all the paths $[0,T-s]\ni t\mapsto \bbmsl{X}_{s,s+t}$, $s\in[0,T]$.

\begin{definition}\label{def:p-variation rough path space} 
  For $p \in [2,3)$, the space 
  $\cD_{p\text{-var}}([0,T], \bbR^d\times \bbR^{d \times d})$ of c\`adl\`ag
  % (resp. the space $\cC_{p\text{-var}}([0,T], \bbR^d\times \bbR^{d \times d})$ of continuous) 
  $p$-variation rough paths is defined by the subspace of all 
  % (resp. continuous) 
  $(X, \bbX) \in \cD([0,T], \bbR^d) \times \cD(\Delta_{[0,T]}, \bbR^{d \times d})$ satisfying Chen's relation, that is,
  \begin{equation}\label{eq:Chen}
    \bbX_{r,t} - \bbX_{r,s} - \bbX_{s,t} = X_{r,s} \otimes X_{s,t}
  \end{equation}
  for all $0 \leq r \leq s \leq t \leq T$, and
  \begin{align}
    |X_0| \,+\,
    \|X\|_{p\text{-var}, [0,T]}
    \,+\, \|\bbX\|_{p/2\text{-var}, [0,T]}^{1/2}
    \;<\;
    \infty.
  \end{align}
  Above and in the following we use the notation $x \otimes y \in \bbR^{d \times d}$ for the tensor product of $x=(x^1,...,x^d), y=(y^1,...,y^d) \in \bbR^d$. More explicitly, $x \otimes y$ is the $d\times d$ matrix whose $(i,j)$ entry is $x^i y^j\in\bbR$. 
  We endow %
  $\cD_{p\text{-var}}([0,T], \bbR^d\times \bbR^{d \times d})$ with 
  the $p$-variation Skorohod distance
  \begin{align*}
    \sigma_{p\text{-var},[0,T]}((X,\mathbb X), (Y, \bbY)) \coloneqq \inf_{\lambda \in \Lambda_T} \big\{&|\lambda| \vee \big( |X_0-Y_0| + \|X - Y \circ \lambda\|_{p\text{-var},[0,T]}\\
    & + \|\bbX - \bbY\circ(\lambda,\lambda)\|_{p/2\text{-var},[0,T]}^{1/2} \big) \big\},
  \end{align*}
  where $\Lambda_T$ are the strictly increasing bijective functions from $[0,T]$ onto itself, and $|\lambda| = \sup_{t \in [0,T]} |\lambda(t) - t|$. 
\end{definition}

The Skorohod distance $\sigma_{\text{unif},[0,T]}$ can be defined similarly, except with the $p$-variation replaced by the uniform norm and the square root of the $p/2$-variation replaced by the infinity variation. See \cite[Section~5]{FZ18} for details.
We shall use the notation 
$D\!\big([0,T],\R^d\times\R^{d\times d}\big)$
for the space of all c\`adl\`ag functions  
$[0,T]\to\R^d\times\R^{d\times d}$
endowed with the standard (AKA the $J_1$) Skorohod topology (cf.\ \cite[Chapter 12]{billingsley1999convergence}. 
In particular, by taking the bijection function $\lambda(t)=t$, we have that for deterministic continuous $(X, \bbX)$ and $(Y, \bbY)$
\begin{align}\label{Skorohod-dist-to-determ-cont-is-bounded-by-uniform-dist}
  \sigma_{\text{unif},[0,T]}((X,\mathbb X), (Y, \bbY)) & 
  \le  
  \|X - Y \|_{\text{unif}, [0,T]} 
  + \|\bbX - \bbY\|_{\infty\text{-var}, [0,T]}\\
  & \le 
  |X_0-Y_0| 
  + \|X - Y \|_{\infty\text{-var}, [0,T]} 
  + \|\bbX - \bbY\|_{\infty\text{-var}, [0,T]}
\end{align}

\begin{remark}
  \label{rmk:one-parameter}Note that by Chen's relation $\mathbb{X}_{s, t}
  =\mathbb{X}_{0, t} -\mathbb{X}_{0, s} - X_{0, s} \otimes X_{s, t}$ whenever
  $0 \leqslant s \leqslant t \leqslant T$, and therefore
  \begin{align*}
    \| \mathbb{X}-\mathbb{Y} \|_{\infty\text{-var}, [0,T]} & = \sup_{0 \leqslant s < t
      \leqslant T} | \mathbb{X}_{s, t} -\mathbb{Y}_{s, t} |
    \le 2 \| \mathbb{X}_{0, \cdot} -\mathbb{Y}_{0, \cdot}
    \|_{\text{unif}, [0,T]}\\
    & \quad + 2 (\| X \|_{\text{unif}, [0,T]} \vee \| Y \|_{\text{unif}, [0,T]}) \| X - Y    \|_{\text{unif}, [0,T]}.
  \end{align*}
  Consequently, the (`classical' Skorohod) distance in $D([0,T],\R^d\times\R^{d\times d})$ of the (one-parameter) paths $(X,
  \mathbb{X}_{0, \cdot}),(Y, \bbY_{0,\cdot}):[0,T]\to \bbR^d\times \bbR^{d \times d} $ controls the Skorohod distance $\sigma_{\text{unif},[0,T]}$ of $(X,\bbX)$ and $(Y,\bbY)$. 
\end{remark}
For $X,Y\in D([0,T],\bbR^d)$ we use the notation $\int_0^t Y_{s-} \otimes \md X_s$ for the left-point Riemann integral, that is
\begin{align}\label{def:left-point-integral}
  \int_0^t Y_{s-} \otimes \md X_s \coloneqq {\int_{(0,t]} Y_{s-}\otimes \md X_s \coloneqq }
  {\lim_{n \to \infty}\big\{\sum_{[u,v]\in\cP_n} Y_{u} \otimes \big(X_{v} - X_{u}\big)\big\},}
\end{align}
whenever this limit is well-defined along an implicitly fixed deterministic sequence of partitions $(\cP_n)$ of $[0,t]$ with mesh size going to zero. Whenever $X$ and $Y$ are also of finite variation, the limit exists and is independent on the choice of the partitions as long as their mesh size converges to zero. This will be always the case for the processes considered in the paper, before taking the scaling limit.  
 Note that if $X$ is a semimartingale and $Y$ is adapted to
the same filtration, then this definition coincides with the Itô integral.
Note that if $X$ is a semimartingale and $Y$ is adapted to the same filtration, then this definition coincides with the It\^o integral.
We remark also that for the iterated integrals
\begin{align*}
  \bbX_{s,t}\coloneqq
  \int_s^t X_{s,u-} \otimes \md X_u 
  = 
  \int_{(s,t]} X_{u-}\otimes \md X_u - X_{s}\otimes X_{s,t},
\end{align*}
$(X,\bbX)$ satisfy Chen's relation \eqref{eq:Chen}. Moreover, so does $(X,\tilde\bbX)$ for $\tilde \bbX_{s,t} \coloneqq \bbX_{s,t} + Y_{s,t}$, whenever  $Y_{s,t}=Y_t-Y_s$ are the increments of a path $[0,T]\to \bbR^{d \times d}$ (since in this case $Y_{r,t}-Y_{s,r}-Y_{s,t}=0$). In particular, this holds for $Y$ of the form $Y_{s,t} \coloneqq (t-s) \Gamma$, for any fixed matrix $\Gamma$.

\section{Main results: invariance principles for the lifted walk}

\subsection{The annealed invariance principle}
Remember the convention
\begin{align}\label{eq:path_increments}
  X_{s,t} \;\ldef\; X_t - X_s
\end{align}
for the increments of a path $X$. %, and $X_{0-} \ldef X_0$.
We also use the standard convention $X_{t-} \ldef \lim_{s \uparrow t} X_s$ for $t>0$.
Note that Assumption~\ref{ass:general} 
% and \ref{ass:process} 
ensures that the stochastic process $X$ is honest, that is, for $\PP$-a.e.~$\om$, it does not explode in finite time $P_0^{\om}$-almost surely. 
In particular, $P_0^{\om}$-a.e.~realization of $X$ gives rise to a piecewise constant path which has only finitely many discontinuities on any finite time horizon, cf. Equation~\ref{eq:number:jumps}.  Hence, the doubly-indexed stochastic process $\bbmsl{X}$ %$(\bbmsl{X}_{s,t} : (s,t)\in \Delta_{[0,T]})$
on $\bbR^{d \times d}$ which is given as the It\^o sum
\begin{align}\label{eq:path_iterated_integral}
  \Delta_{[0,T]}\ni(s,t)
  \;\longmapsto\;
  \bbmsl{X}_{s,t}
  \;\ldef\;
  \int_s^t X_{s,r-} \otimes \md X_r
  \;\ldef\;
  \sum_{s < r \leq t} X_{s,r-} \otimes X_{r-,r}
\end{align}
is well-defined since in the sum over times $r$ only a finite number of terms $X_{r-,r}$ are non-zero. Here and throughout, we use the convention $\sum_{r \in I} a_r\coloneqq\sum_{r \in I, a_r\neq 0}a_r$ for an index set $I$ and a function $a_r$ which is non-zero for at most countably many $r\in I$.

We will call the process $(X, \bbmsl{X}) = ((X_{s,t}, \bbmsl{X}_{s,t}) : (s,t)\in\Delta_{[0,T]})$ the \emph{lifted walk}. We shall also call it \emph{It\^o lifted walk} to stress the type of integral used.
We will also use the notation $u^{\otimes 2}\coloneqq u \otimes u$ for $u \in \R^d$.

We shall now state the first main result of the paper, the \emph{annealed or averaged invariance principle} for the lifted random walk in the $p$-variation rough path topology.

We denote by $\prob_0$ the marginal law of the walk starting at $0$ whenever the environment sampled from $\prob$. That is 
\[
\prob_0(\cdot):=\mean [ P^\om_0(\cdot)].
\]
Correspondingly, $\mean_0$ is the expectation with respect to $\PP_0$. Analogously
\[
\PP_0(\cdot):=\EE [ P^\om_0(\cdot)]
\]
and $\EE_0$ is the expectation with respect to $\PP_0$.
Notice that in particular
\[
\PP_0(\cdot):=\prob_0 ( \cdot \,|\, \Om_0) \quad \text{ and }\quad
\EE_0[\cdot]:=\mean_0 [ \cdot \,|\, \Om_0].
\]

Given any process $X$ and $n \in \bbN$, we define the rescaled process by
\begin{align}\label{eq:def_rescaled_process}
  X_t^n \ldef \frac{1}{\sqrt n} X_{n t}
  \; \text{ and }
  \bbmsl{X}_{s,t}^n \ldef \int_s^t X_{s,r-}^n \otimes \md X_r^n = \frac{1}{n} \int_{ns}^{nt} X_{ns,r-} \otimes \md X_r.
\end{align}

\begin{theorem}\label{thm:annealed}
  Suppose that Assumption~\ref{ass:general} 
  % and \ref{ass:process} 
  holds. Then, under $\PP_0$, for any $p > 2$ and $T > 0$
  \begin{align*}
    % \Big(
    % \Big(
    % X_t^n, {\displaystyle \int_0^t X_{s-}^n \otimes \md X_s^{n}}
    % \Big)
    % : t \geq 0
    % \Big)
    \big(
    X^n, \bbX^{n}
    \big)
    \underset{n \to \infty}{\;\Longrightarrow\;}
    \Big(
    B, \big( {\displaystyle \int_s^t B_{s,r} \otimes \md B_r} + (t-s)\, \Ga
    \big)
    _{(s,t) \in \Delta_{[0,T]}}
    \Big)
    % \big( \big( B_t, \bbB_t + t \Ga \big) : t \geq 0 \big)
  \end{align*}
  in $\cD_{p\text{-var}}([0,T], \bbR^d\times \bbR^{d \times d})$, where $(B_t : t \geq 0)$ is a Brownian motion on $\bbR^d$ with deterministic covariance matrix $\Si^2$ and $\Ga \in \bbR^{d \times d}$ - also called area anomaly - are as given in Equation~\eqref{eq:def:Si+Ga}. Further, the integration with respect to $B$ is in the It\^o sense. 
\end{theorem}

It is well known that the diffusively rescaled processes $(X^n)_{n\in\mathbb N}$
satisfy an annealed invariance principle under Assumption~\ref{ass:general}.
The first such results for random walks in random environments were
established by De~Masi, Ferrari, Goldstein and Wick for reversible
nearest-neighbor models \cite{DeMasi1989_JSP}. However, an extension including long-range jumps is rather straightforward.
More generally, the Kipnis-Varadhan theory \cite{kipnis1986central}
provides an invariance principle for additive functionals of
reversible Markov processes, and applies in particular to the random
conductance model under appropriate moment assumptions.
Extensions to long-range jumps follow by the same martingale
approximation argument once the corresponding $L^2_{\mathrm{cov}}$
corrector theory is available.

As mentioned in the introduction, the annealed invariance principle in the rough path topology for the lifted random walk in nearest-neighbor i.i.d.\ uniformly elliptic conductances was proven in \cite{deuschel2021additive} as a corollary from a Kipnis-Varadhan type invariance principle for additive functionals of Markov processes extended to the rough path topology, which was the main result of that paper.

\begin{remark}[On non-uniqueness of the infinite cluster]
If Assumption~\ref{ass:general}(i) (uniqueness of the infinite cluster)
is removed, but the remaining parts of Assumption~\ref{ass:general}
are retained, one can still obtain the annealed invariance principle
after conditioning on the event $\{0 \in \mathcal C_\infty(\omega)\}$.
However, in the absence of uniqueness, the law of the environment
conditioned on $\{0 \in \mathcal C_\infty(\omega)\}$ need not be ergodic.
In that case, the random walk started at the origin remains confined
to the infinite connected component $\mathcal C_0(\omega)$ of $0$,
and the effective coefficients in the limiting Brownian rough path
may depend on the ergodic component of the conditioned law selected
by $\omega$.
More precisely, the diffusion matrix and the area correction
matrix are shift-invariant quantities and are therefore deterministic on each ergodic component.  If the conditional law is a mixture of ergodic components, then they may become random under the conditional measure, with their values determined by
the ergodic component associated with $\mathcal C_0(\omega)$.
In particular, uniqueness of the infinite cluster guarantees that the
conditional law remains ergodic and hence that the limiting Brownian
rough path has deterministic coefficients.
\end{remark}

\subsection{The quenched invariance principle}
The second main result is a \emph{quenched invariance principle} for the lifted random walks. 

The quenched classical invariance principle (that is, the invariance pronciple for the process alone, in the Skorohod topology) is substantially more delicate than the annealed one. For uniformly elliptic nearest-neighbor conductances, it was proved by
Sidoravicius and Sznitman \cite{sidoravicius2004quenched} and 
by Berger and Biskup \cite{berger2007quenched}.
For supercritical percolation clusters, quenched invariance principles
were established by Mathieu and Piatnitski \cite{MP07} and by
Barlow \cite{barlow2004random} under suitable heat kernel bounds.
In degenerate ergodic environments, the quenched invariance principle
was obtained by Andres, Deuschel and Slowik \cite{ADS15} and further
developed in subsequent works such as \cite{DNS18}.
A comprehensive overview on the random conductance model
can be found in the still relevant 2011 survey of Biskup \cite{Bi11}.

Beyond the settings of Assumption \ref{ass:general}, here we shall need to assume additionally that there exists a potential for the corrector with finite moments.
\begin{condition}[Regular potential]\label{cond:moment-potential}
The corrector from Definition \ref{def:corrector+matrix} satisfies 
  $\chi=\mD\phi$ on $\Omega_0$ for a function $\phi: \Omega_0 \rightarrow \R^d$ so that $\mD\phi\in L_{\mathrm{cov}}^2(\PP)$ and moreover, there is some $\epsilon>0$ so that  
  \begin{align}\label{eq:conditionSP}
      \kappa
      %\;=\; \kappa(\epsilon)
      \;:=\;
      \E\bigg[
  \sum_{x\in\bbZ^d }\omega(0,x) \big(
  |\phi(\tau_x\omega)|^2 + |\phi(\omega)|^{2+\epsilon} 
  \big)\indizeroo
  \bigg]
  \;<\; \infty. 
  \end{align}
\end{condition}

With this condition on the corrector one can then prove a quenched invariance principle for the lifted random walk.

\begin{theorem}\label{theorem}
  Let $X$ be the random walk defined by \eqref{eq:process_generator}, let $\X$ be the It\^o sum defined by \eqref{eq:path_iterated_integral} and consider the rescaled process (\ref{eq:def_rescaled_process}). Assume Condition \ref{cond:moment-potential} holds. Then for any $T > 0$ and $p > 2$ and for $\mathbf{P}$-almost every $\omega$ under $P_{0}^{\omega}$ 
  \begin{align*}
    % \Big(
    % \Big(
    % X_t^n, {\displaystyle \int_0^t X_{s-}^n \otimes \md X_s^{n}}
    % \Big)
    % : t \geq 0
    % \Big)
    \big(
    X^n, \bbX^{n}
    \big)
    \underset{n \to \infty}{\;\Longrightarrow\;}
    \Big(
    B, \big( {\displaystyle \int_s^t B_{s,r} \otimes \md B_r} + (t-s)\, \Ga
    \big)
    _{(s,t) \in \Delta_{[0,T]}}
    \Big)
    % \big( \big( B_t, \bbB_t + t \Ga \big) : t \geq 0 \big)
  \end{align*}
  in $\cD_{p\text{-var}}([0,T], \bbR^d\times \bbR^{d \times d})$, where $B_t$ is a $d$-dimensional Brownian motion with covariance matrix $\Sigma^2$ 
  and area anomaly $\Gamma$, coincide with the ones in Theorem~\ref{thm:annealed}. 
 \end{theorem}

%%%%%%%%%%%%%%%%%%%
%%%%%%%%%%%%%%

\begin{remark}
    Suppose we are in the settings of Theorem \ref{theorem}. That is, in addition to Assumption \ref{ass:general}, we assume that \eqref{eq:conditionSP} holds true (even with $\epsilon=0$). Then, the covariance matrix $\Sigma^2$ from Theorem \ref{theorem} is non-degenerate.    
    \end{remark}
\begin{proof}
    Assume by contradiction that 
    $v^\top\Sigma^2 v=0$ for some $v\in\bbR^d$ nonzero. Then 
 \begin{align*}
    0=v^\top\Sigma^2 v\Si^2
    \;=\;
    v^\top \big( \langle \Phi^i, \Phi^j \rangle_{L^2_{\mathrm{cov}}(\PP)} \big)_{i,j = 1}^d v
  \;=\;
    v^\top \big( \langle \Pi^i-\chi^i, \Pi^j-\chi^j \rangle_{L^2_{\mathrm{cov}}(\PP)} \big)_{i,j = 1}^d v
    \end{align*}
and therefore, using Condition \ref{cond:moment-potential} and the definition of $\Pi$, we obtain 
\begin{align*}
    \mean \big[ \sum_{x\in\bbZ^d}\omega(0,x) \big(v \cdot \big( x + \phi(\omega)-\phi(\tau_x\omega)\big)\big)^2
    \indizeroo\big]
    \;=\; 
    0,  
    \end{align*}
    where $v\cdot \phi(\omega):=\sum_{i=1}^d v_i\phi^i(\omega)$.
  This implies that on the event  $\{0\in\cC_\infty\}$, we have 
  \[
  v\cdot x = v\cdot \phi(\tau_x\omega) - v\cdot \phi(\omega),
  \] 
  whenever $\omega(0,x)>0$. Summing over finite paths one deduces the last equality holds on $\{0\in\cC_\infty\}$ $\prob$ a.s.\ for every $x\in\cC_\infty$, and in particular
  \[
  \mu^\omega(x)|v\cdot x|^2 = \mu^\omega(x) |v\cdot \phi(\tau_x\omega) - v\cdot \phi(\omega)|^2
  \le 
  2 
 \mu^\omega(x) |v\cdot \phi(\tau_x\omega)|^2 +2\mu^\omega(x) | v\cdot \phi(\omega)|^2.
  \] 
Now, let $B_N=\{-N,...-1,0,1,...,N\}^d$. 
Let
\[
M_N:= \frac{1}{N^d}\sum_{x\in B_N} \left(v \cdot \frac{x}{N}\right)^2 \mu^\omega(x)\indicator_{\{x\in\cC_\infty\}}.
\]
By the ergodic theorem with weighted
average (cf.\ Theorem 2.3 of Faggionato \cite{faggionato2025ergodic}), we have 
\begin{align*}
    M_N  
    \;\xrightarrow[N\to\infty]{\prob-
    \text{a.s.}}\; & 
    2^d \mean[\mu^\omega(0) \indizeroo]
    \int_{[-1,1]^d} (v \cdot u)^2 \md u \\
&    \;=\;  \mean[\mu^\omega(0) \indizeroo] 
    |v|^2 2^{2d-1}/3
    =: \ell.
    \end{align*}
Notice that $\mu^\omega(0)>0$ on $\{0\in\cC_\infty\}$. Therefore $\prob(0\in\cC_\infty)>0$ implies $\ell>0$.   
Also, 
\[
\PP(\lim_{N\to\infty}M_N= \ell) = 1.
\]
 On the other hand, the ergodic theorem implies that
\begin{align*}
    &L_N:= \frac{1}{N^d}\sum_{x\in B_N}  \indicator_{\{x\in\cC_\infty\}}\mu^\omega(x)
    \;\xrightarrow[N\to\infty]{\prob-
    \text{a.s.}}\;
     3^d \mean [\mu^\omega(0) \indizeroo] 
    \end{align*}
and
 \begin{align*}
 & \tilde L_N:= \frac{1}{N^d}\sum_{x\in B_N} |v\cdot \phi(\tau_x\omega)|^2 \mu^\omega(x) \indicator_{\{x\in\cC_\infty\}}      
 \;\xrightarrow[N\to\infty]{\prob-
    \text{a.s.}}\;
  3^d \mean [|v\cdot \phi(\omega)|^2\mu^\omega(0)  \indizeroo]. 
 \end{align*}
Therefore

\begin{comment}
\begin{align*}
    &\indizeroo|v\cdot \phi(\omega)|^2 \frac{1}{N^d}\sum_{x\in B_N}  \indicator_{\{x\in\cC_\infty\}}\mu^\omega(x)
    \;\xrightarrow[N\to\infty]{\prob-
    \text{a.s.}}\;
    \indizeroo|v\cdot \phi(\omega)|^2 3^d \mean [\mu^\omega(0) \indizeroo].    
    \end{align*}
 For the second term, again by the ergodic theorem 
 \begin{align*}
 & \indizeroo \mu^\omega(0) \frac{1}{N^d}\sum_{x\in B_N} |v\cdot \phi(\tau_x\omega)|^2 \indicator_{\{x\in\cC_\infty\}}      
 \;\xrightarrow[N\to\infty]{\prob-
    \text{a.s.}}\;
    \indizeroo \mu^\omega(0) 3^d \mean [|v\cdot \phi(\omega)|^2  \indizeroo]. 
 \end{align*}
\end{comment}
 
Notice that
 \begin{align*}
N^2 M_N \indizeroo 
& 
\;=\; \indizeroo  \frac{1}{N^d}\sum_{x\in B_N} \mu^\omega(x) |v\cdot (\phi(\tau_x\omega)-\phi(\omega))|^2 \indicator_{\{x\in\cC_\infty\}}\\
& 
\;\le\;
2 \indizeroo ( \tilde L_N + |v\cdot \phi(\omega)|^2 L_n ).
    \end{align*}
In particular we obtain that 
 \begin{align*}
 \prob\left(\sup_{N\ge 1} \{N^2 M_N\indizeroo\}  
\;<\;
\infty\right)=1.
    \end{align*}
Therefore  
\begin{align*}
 \PP\left(\sup_{N\ge 1}\{ N^2 M_N\} 
\;<\;
\infty\right)=1,
\end{align*}
which contradicts the fact that $\ell>0$. This concludes the proof.   
\end{proof}

%%%%%%%%%
%%%%%%%%%%%%%%%
%%%%%%%%%%%%%%%%%%%

\section{Sources of regular stationary potential: Historical background and a construction}\label{sec:Souces-SP}

In reversible nearest-neighbor random conductance models, the corrector is
classically constructed in the Hilbert space $L^2_{\mathrm{cov}}(\prob)$ of
square-integrable covariant cocycles; see \cite{kipnis1986central} and, for the
random conductance model, the $L^2_{\mathrm{cov}}$ framework in
\cite{BP07,MP07,Bi11}.  In our notation (cf.\ Definition~\ref{def:corrector+matrix}),
one defines $L^2_{\mathrm{pot}}$ as the closure in $L^2_{\mathrm{cov}}$ of
gradients of local functions and sets $L^2_{\mathrm{sol}}:=(L^2_{\mathrm{pot}})^\perp$.
The corrector $\chi$ is then the unique element in $L^2_{\mathrm{pot}}$ such that
$\Pi-\chi\in L^2_{\mathrm{sol}}$, that is, $\chi$ is the $L^2_{\mathrm{cov}}$-limit
of gradients of local functions.  Equivalently, there exists a potential $\phi$
in the associated energy space (the completion of local functions with respect to
the seminorm $\|\mD \phi\|_{L^2_{\mathrm{cov}}}$) such that
\begin{align}\label{eq:hist_energy_potential}
  \chi(\omega,x)=\mD\phi(\omega,x)=\phi(\tau_x\omega)-\phi(\omega)
  \qquad \text{in } L^2_{\mathrm{cov}}.
\end{align}
The potential $\phi$ is unique up to an additive constant.

It is important to stress that the representation \eqref{eq:hist_energy_potential}
arises at the level of the above energy space and does \emph{not} automatically
imply that $\phi\in L^2(\Omega)$, nor any higher moment bounds.  Such
integrability properties typically require additional assumptions on the law of
the environment, for instance spectral gap or Poincar\'e inequalities for the
shift action.

In uniformly elliptic nearest-neighbor i.i.d.\ environments, Gloria-Otto \cite{GloriaOtto2011}
prove quantitative estimates and high moments for the massive (regularized)
corrector, using the classical $L^2_{\mathrm{cov}}$ theory as an underlying
framework.  A closely related environment-space viewpoint, based on resolvent
equations for regularized potentials, is developed by Mourrat \cite{Mourrat2012}.

For degenerate nearest-neighbor conductance models under suitable spectral gap
and moment assumptions, Andres-Neukamm \cite{AndresNeukamm2019} construct directly a potential 
 $\phi$ with finite moments of arbitrary order and establish explicitly
the representation $\chi=\mD\phi$. In fact, to the best of our knowledge, their and our quenched results are the only applications of the moment bounds for the corrector which is the subject of various results in quantitative homogenization theory as mentioned above.  

The transfer lemma proved below is conceptually different from the classical $L^2$ projection argument (Künnemann-type decomposition) used in uniformly elliptic settings \cite{Kunnemann1983}. Our construction is based on spatial averaging and moment bounds for the corrector, and does not rely on spectral or Hilbert space methods.

\subsection{From spatial corrector estimates to stationary potentials}

In the quenched part of the paper we assume Condition~\ref{cond:moment-potential}, i.e.\
the existence of a stationary potential $\phi$ with $\chi=\mD\phi$ on $\Omega_0$
and suitable moment control.  As discussed above, such a potential is available
in several nearest-neighbor settings under additional structural assumptions
(e.g.\ spectral gap hypotheses, as in the work of Andres-Neukamm \cite{AndresNeukamm2019}).

In several important models, such as supercritical percolation clusters
\cite{barlow2004random,berger2007quenched,dario2018optimal},
quantitative information on the corrector is obtained first in its spatial
formulation, namely as moment bounds for $\chi(\omega,x)$ (uniformly in $x$
on the event $\{x\in\mathcal C_\infty(\omega)\}$), together with geometric
control of the underlying infinite cluster.  The purpose of the next result
is to provide a direct mechanism that upgrades such \emph{spatial} corrector
bounds to a stationary potential satisfying
Condition~\ref{cond:moment-potential}.

The interest of this transfer lemma is twofold.  First, it does not rely on
Hilbert space projection arguments or spectral gap assumptions for the shift
action.  Second, it is modular: whenever spatial moment estimates for the
corrector and mild volume regularity are available, regardless of how they are
proved, the result produces a stationary potential with corresponding moment
bounds.  This isolates the analytic input needed for the quenched rough path
invariance principle from the specific method used to construct the corrector.

\begin{lemma}[Transfer Lemma: stationary potential from a spatial corrector]
\label{lem:transfer-principle}

Assume $d \ge 3$ and let $\chi:\Omega_0 \times \mathbb Z^d \to \mathbb R^d$
be the corrector from Definition~\ref{def:corrector+matrix}.
Assume the following.

\smallskip
\noindent
\textbf{ (i) Uniform spatial moment bound.}
There exists $\epsilon_0>0$ such that
\begin{align}
\sup_{x \in \mathbb Z^d}
\EE\!\left[
|\chi(\omega,x)|^{\,2+\epsilon_0}
\mathbf 1_{\{x \in \mathcal C_\infty(\omega)\}}
\right]
< \infty .
\label{eq:spatial-moment}
\end{align}

\smallskip
\noindent
\textbf{ (ii) Volume lower deviation.}
Let
\[
B_n(\omega)
:=
\{-n,\ldots,n\}^d \cap \mathcal C_\infty(\omega).
\]
There exist constants $c,\alpha>0$ such that
\begin{align}
\prob\big(|B_n| \le c n^d \big)
\le
e^{-n^\alpha}
\qquad \text{for all sufficiently large } n .
\label{eq:volume-deviation}
\end{align}

\smallskip
Then there exists a measurable function
$\phi:\Omega_0 \to \mathbb R^d$ such that

\begin{align}
\chi(\omega,x)
=
\mD\phi(\omega,x)
=
\phi(\tau_x \omega)-\phi(\omega)
\qquad
\text{for all } x \in \mathcal C_\infty(\omega),
\end{align}

and moreover $\phi$ satisfies Condition~\ref{cond:moment-potential} with $0<\epsilon<\epsilon_0$.
\end{lemma}

\begin{proof}
      We shall first construct a function $\phi$ so that its gradient is $\chi $. Remember $B_n(\om)= \left\{-n,\ldots,n\right\}^d\cap \mathcal{C}_\infty(\omega)$ and define
  \begin{equation*}
    \phi_n\left(\omega\right) = -  \left( \frac{1}{|B_n(\om)|}\sum_{x \in B_n(\om)} \chi  \left( \omega , x \right) \right) \indicator_{B_n(\omega)\ne\emptyset}.
  \end{equation*}
  Let $0<\epsilon<\epsilon_0$. For $q:=2+\epsilon$, let $\beta>0$ so that $2+\epsilon_0 = q(1+\beta)$. Using the triangle and the H\"older inequality we can observe that
  \begin{align*}
    \textbf{E}_0 & \left[\left\| \frac{1}{|B_n(\om)|} \sum_{x\in B_n(\om)} \chi  \left( \omega , x \right) \right\|^q \indicator_{B_n(\omega)\ne\emptyset}\right]^{1/q}  
    \\
    & = 
    \frac{1}{(2n+1)^d} \textbf{E}_0\left[\left\|  \sum_{x\in \left\{-n,\ldots,n\right\}^d} \frac{(2n+1)^d}{|B_n(\om)|}  \mathbbm{1}_{\left\{x \in \mathcal{C}_\infty \right\}} \chi  \left( \omega,x \right) \indicator_{B_n(\omega)\ne\emptyset}\right\|^q \right]^{1/q} \\
    & \leq 
    \frac{1}{(2n+1)^d}  \sum_{x\in \left\{-n,\ldots,n\right\}^d} \textbf{E}_0\left[\left\| \frac{(2n+1)^d}{|B_n(\om)|}  \mathbbm{1}_{\left\{x \in \mathcal{C}_\infty \right\}} \chi  \left( \omega, x \right) \indicator_{B_n(\omega)\ne\emptyset} \right\|^q \right]^{1/q} \\
    & \leq 
    \frac{1}{(2n+1)^d} \sum_{x\in \left\{-n,\ldots,n\right\}^d} \textbf{E}_0\left[\left\| \mathbbm{1}_{\left\{x \in \mathcal{C}_\infty \right\}} \chi  \left( \omega, x \right) \right\|^{q(1+\beta)} \right]^{\frac{1}{q(1+\beta)}} 
    \textbf{E}_0\left[ \left( \frac{(2n+1)^d}{|B_n(\om)|} \right)^{q\frac{1+\beta}{\beta}} \indicator_{B_n(\omega)\ne\emptyset}\right]^{\frac{\beta}{q(1+\beta)}} \ .
  \end{align*}
From the uniform spatial moment assumption in (i) it follows that 
  \begin{align*}
    \frac{1}{(2n+1)^d} \sum_{x\in \left\{-n,\ldots,n\right\}^d} \textbf{E}_0
    \left[
    \left\| \mathbbm{1}_{\left\{x \in \mathcal{C}_\infty \right\}} \chi  \left( \omega , x \right) \right\|^{q(1+\beta)} \right]^{\frac{1}{q(1+\beta)}} 
  \end{align*} 
  is uniformly bounded. In order to bound the second term 
  \begin{align*}
    \textbf{E}_0\left[ \left( \frac{(2n+1)^d}{|B_n(\om)|} \right)^{q\frac{1+\beta}{\beta}} \indicator_{B_n(\omega)\ne\emptyset}\right]^{\frac{\beta}{q(1+\beta)}}
  \end{align*}
  observe that the probability that $|B_n(\om)|$ is smaller than $c n^d$ is stretched exponentially small by the volume lower deviation assumption (ii). Consequently,  $\textbf{E}_0\left[ \left( \frac{(2n+1)^d}{|B_n(\om)|} \right)^{q\frac{1+\beta}{\beta}} \indicator_{B_n(\omega)\ne\emptyset}\right]^{\frac{\beta}{q(1+\beta)}}$ is uniformly bounded.\\
  So $(\phi_n)_{n\geq 1}$ is a bounded sequence in $L^q(\mathbf{P})$. In particular, there is a subsequence $(\phi_{n_k})_{k \geq 1}$ that converges weakly in $L^q$ (remember the theorem of Banach-Alaoglu and that in $L^q$ for $1<q<\infty$ weak- and weak-$\star$-topologies are the same) to some limit $\phi$.

  Let $\omega$ be such that $0\in \mathcal{C}_\infty(\omega)$, and let $x \in \mathcal{C}_\infty(\omega)$. We claim that
  \begin{equation}\label{1.3}
    \phi(\tau_x\omega) = \phi(\omega) + \chi (\omega,x) \ .
  \end{equation}
  Indeed, note that 
  \begin{equation*}
    \lim_{n \to \infty}\left( \phi_n(\tau_x\omega) - \phi_n(\omega) \right) =   \chi (\omega,x) \ .	
  \end{equation*}
  in $L^2(\Omega,\mathbf{P})$ and thus also 	in the weak sense. To see this, write the difference 
  \begin{equation*}
    E_n(\omega,x):= \phi_n(\tau_x\omega) - \phi_n(\omega)  -   \chi (\omega,x).	
  \end{equation*}	
  Using the cocycle property, we have
    \[ 
E_n(\omega,x)=
   \frac{1}{|B_n(\om)|}\sum_{y \in B_n(\om)} \chi  \left( \omega , y \right)\indicator_{B_n(\om)\ne \emptyset}   
   - \frac{1}{|B_n(\tau_x\om)+x|}\sum_{y \in B_n(\tau_x\om)+x} \chi  \left( \omega , y \right)\indicator_{B_n(\tau_x\om)\ne \emptyset}. 
     \]    
A consequence of Cauchy-Schwarz inequality yields
  \begin{align*}
    |E_n(\omega,x)|^2
    \le&
    \frac{|B_n(\om) \,\triangle\, (B_n(\tau_x\om)+x)|}{|B_n(\om)|^2}
    \sum_{y \in B_n(\om) \,\triangle\, (B_n(\tau_x\om)+x)}
    |\chi(\omega,y)|^2\,\indicator_{B_n(\omega)\ne\emptyset}\\
    &+
    \frac{|(B_n(\tau_x\om)+x)\triangle B_n(\om)|^2}{|B_n(\om)|\cdot|(B_n(\tau_x\om)+x)|^2}
  \sum_{y \in B_n(\omega)} |\chi  \left( \omega , y \right)|^2\indicator_{B_n(\omega)\ne\emptyset}\indicator_{(B_n(\tau_x\omega)+x)\ne\emptyset},
  \end{align*}
  where $ B_n(\om) \,\triangle\, (B_n(\tau_x\om)+x)$ is the symmetric difference of the sets $B_n(\om)$ and $B_n(\tau_x\om)+x$.
  Let 
  \begin{align*}
    \rho:= \sup_{y\in\mathbb{Z}^{d}}\,
    \mathbf{E}_0 
    [
    |\chi(\omega,y)|^2 \mathbbm{1}_{\left\{y \in \mathcal{C}_\infty \right\}}  
    ] 
  \end{align*}
  which is finite by assumption (i).
  Using again the fact the probability that $|B_n(\omega)|$ is smaller than 
  $c n^d$ 
  is stretched-exponentially small by assumption (ii) while the symmetric difference of $B_n(\omega)$ and $B_n(\tau_x\omega)+x$ being of order at most $n^{d-1}$   we have a constant $C$ such that 
  \begin{align*}
    \mathbf{E}_0\!\left[\,|E_n(\omega,x)|^2\,\right]
    \le
    \rho  \,C \,\bigg(\frac{n^{d-1}}{n^d}\bigg)^2\;\xrightarrow[n\to\infty]{} 0.
  \end{align*}
  Therefore (\ref{1.3}) holds and $\chi=\mD\phi$. 
  \end{proof}

\begin{remark}
  If condition (i) of uniform spatial bound holds for all moments, that is for all $\epsilon_0>0$, then $\phi$ can be constructed to have all moments. Indeed, the argument of the proof holds for all integers $q>2$  and therefore by constructing a diagonal sequence we can furthermore make sure that the weak convergence to $\phi$ takes place for all integer $q>2$ simultaneously, and thus $\textbf{E}_0\left[\left\| \phi \right\|^q\right]< \infty$ for all $q \geq 2$.   
\end{remark}

\medskip

The previous proposition provides a way to verify
Condition~\ref{cond:moment-potential} from spatial moment bounds for the
corrector and mild volume regularity, without invoking spectral gap
assumptions on the environment.  In particular, once
Condition~\ref{cond:moment-potential} is available, the quenched rough path
invariance principle can be proved independently of how the potential
$\phi$ is constructed.

As an additional consequence of the spatial moment bounds, one can
combine them with quenched heat kernel estimates to obtain uniform
in time moment bounds for the corrector evaluated along the walk.
This estimate is not required in the proof of the quenched rough path
invariance principle, but it follows naturally from the same input
and may be of independent interest.

Below we shall use the symbol $" \lesssim "$; for two functions $f,g:M \rightarrow \R_{\geq 0}$ defined on some set $M$ we will write $f(x) \lesssim g(x)$ if there exists a constant $C\in(0,\infty)$ such that $f(x) \leq C g(x)$ for all $x \in M$. Whenever we would like to stress that the constant $C$ depends on $\omega$, we shall use the symbol $" \lesssim_{\om} "$. 

\begin{corollary}[Uniform quenched moments of the corrector along the walk]
\label{cor:uniform-quenched-moments}
Fix $q\ge 1$. Assume that the corrector $\chi$ satisfies the uniform spatial
moment bound from Lemma~\ref{lem:transfer-principle}\emph{(i)}, namely
\begin{align}\label{eq:cor_spatial_moment}
\sup_{x \in \mathbb Z^d}
\EE\!\left[
\|\chi(\omega,x)\|^{\,q}\,
\mathbf 1_{\{x \in \mathcal C_\infty(\omega)\}}
\right]
< \infty.
\end{align}
Assume moreover the following quenched heat kernel upper bound: there exist
(random) constants $S=S(\omega)\in(0,\infty)$ and deterministic constants $c,C>0$
such that for $\prob$-a.e.\ $\omega$, for all $t\ge S(\omega)$ and all
$x\in\mathbb Z^d$,
\begin{align}\label{eq:HK_assumption_cor}
  P_0^\omega(X_t=x)\le C\,t^{-d/2}\exp\!\left(-c\frac{|x|^2}{t}\right).
\end{align}
Assume in addition that there exist deterministic constants $b<\infty$ and
$\tilde c>0$ such that for $\prob$-a.e.\ $\omega$, for all $t\ge 0$ and all
$x\in\mathbb Z^d$ with $|x|>3dbt$,
\begin{align}\label{eq:offdiag_assumption_cor}
  P_0^\omega(X_t=x)\le e^{-\tilde c |x|}.
\end{align}
Then, for $\prob$-a.e.\ $\omega$,
\begin{align}\label{eq:finiteness_cor}
  \sup_{t\ge 0}E_0^\omega\big[\|\chi(\omega,X_t)\|^q\big]<\infty.
\end{align}
\end{corollary}

\begin{proof}
We prove \eqref{eq:finiteness_cor}.  We show the bound for $t$ large enough; the
case of bounded $t$ is straightforward.
Let $B_n(\om)= \left\{-n,\ldots,n\right\}^d\cap \mathcal{C}_\infty(\omega)$.
By the spatial ergodic theorem and \eqref{eq:cor_spatial_moment} we have the a.s.\ limit,
\begin{align*}
  \lim_{n\to\infty}\frac{1}{|B_n(\om)|} \indicator_{B_n(\om)\ne\emptyset}\sum_{x\in B_n(\om)}\|\chi(\omega,x)\|^q
 <\infty,
\end{align*}
and hence a.s.
\begin{align}\label{Msup}
  M=M(\omega):=\sup_{n}\frac{1}{|B_n(\om)|} \indicator_{B_n(\om)\ne\emptyset}\sum_{x\in B_n(\om)}\|\chi(\omega,x)\|^q<\infty.
\end{align}
In particular, for any $R\ge 1$,
\begin{align}\label{eq:boxsum_bound}
  \sum_{x\in B_R(\om)}\|\chi(\omega,x)\|^q\le M\,|B_R(\om)|\le M(2R+1)^d.
\end{align}

Fix $t>S(\omega)$ and set $N=\lfloor \sqrt t\rfloor$.  For $k\ge 1$ define the
annuli $A_k(\omega):=B_{2^kN}(\omega)\setminus B_{2^{k-1}N}(\omega)$ and $A_0(\omega):=B_N(\omega)$.  Let $K$ be the
smallest integer such that $2^K N\ge 3dbt$.  Moreover, for $k\ge 1$ set
$D_k(\omega):=B_{2^k 3dbt}(\omega)\setminus B_{2^{k-1}3dbt}(\omega)$.  Then
\begin{align*}
  E_0^\omega\big[\|\chi(\omega,X_t)\|^q\big]
  &=\sum_{x\in\mathbb Z^d}P_0^\omega(X_t=x)\,\|\chi(\omega,x)\|^q\\
  &\le \sum_{k=0}^K\sum_{x\in A_k(\omega)}P_0^\omega(X_t=x)\,\|\chi(\omega,x)\|^q
      +\sum_{k=1}^\infty\sum_{x\in D_k(\omega)}P_0^\omega(X_t=x)\,\|\chi(\omega,x)\|^q.
\end{align*}
For $x\in A_k(\omega)$ we use \eqref{eq:HK_assumption_cor} and $|x|\ge 2^{k-1}N$:
\begin{align*}
  P_0^\omega(X_t=x)\lesssim_{\omega} t^{-d/2}\exp\!\left(-c\frac{N^2 2^{2k}}{t}\right).
\end{align*}

For $x\in D_k(\omega)$ we use \eqref{eq:offdiag_assumption_cor} and $|x|\ge 2^{k-1}3dbt$:
\begin{align*}
  P_0^\omega(X_t=x)\lesssim_{\omega} e^{-\tilde c\,2^k dbt}.
\end{align*}
Combining these bounds with \eqref{eq:boxsum_bound} yields 
\begin{align*}
  E_0^\omega\big[\|\chi(\omega,X_t)\|^q\big]
  &\lesssim_{\omega}
  \sum_{k=0}^K t^{-d/2}\exp\!\left(-c\frac{N^2 2^{2k}}{t}\right)
  \sum_{x\in A_k}\|\chi(\omega,x)\|^q
  +\sum_{k=1}^\infty e^{-\tilde c\,2^k dbt}\sum_{x\in D_k(\omega)}\|\chi(\omega,x)\|^q\\
  &\le
  M(\omega)\sum_{k=0}^K (2^{k+1}N+1)^d\,t^{-d/2}\exp\!\left(-c\frac{N^2 2^{2k}}{t}\right)
  +M(\omega)\sum_{k=1}^\infty (2^k dbt)^d e^{-\tilde c\,2^k dbt}.
\end{align*}
Since $N=\lfloor \sqrt t\rfloor$, the first series is bounded by a constant
independent of $t$,
\begin{align*}
  \sum_{k=0}^\infty 2^{kd} e^{-c2^{2k}}<\infty,
\end{align*}
and the second series is uniformly bounded in $t$ by exponential decay,
\begin{align*}
  \sup_{t\ge 1}\sum_{k=1}^\infty (2^k t)^d e^{-\tilde c\,2^k t}<\infty.
\end{align*}
This proves \eqref{eq:finiteness_cor} for all large $t$, and hence for all
$t\ge 0$.
\end{proof}

Notice that the last theorem implies that (and in fact, is equivalent to) the function $\phi$ having uniformly bounded moments over the process of the environment seen from the walker:   
\begin{equation}\label{finiteness-phi}
  \sup_t E_{0}^{\omega}\left[ \left\| \phi(\tau_{X_t}\omega) \right\|^q \right] < \infty 	\end{equation}
for $\mathbf{P}$-a.e. $\omega$ and all $1\leq q < \infty$. 
Indeed, if $\chi=D\phi$, then	
\begin{align*}
  \sup_t E_{0}^{\omega} \left[ \left\| \phi(\tau_{X_t}\omega) \right\|^q \right]    
  & \lesssim_{\omega}
  \sup_t E_{0}^{\omega} \left[ \left\| \chi(\omega,X_t) \right\|^q \right]  +  \left\| \phi(\omega) \right\|^q <\infty.  
\end{align*}

\begin{example}[Supercritical percolation with uniformly elliptic weights]
\label{ex:percolation_transfer_cor}
Consider i.i.d.\ nearest-neighbor conductances on $\mathbb Z^d$ of the form
\[
\omega(e)\in \{0\}\cup[a,b],\qquad 0<a\le b<\infty,
\]
and assume $\prob(\omega(e)>0)>p_c(d)$ so that $\mathcal C_\infty(\omega)$ exists
and is unique. Under the conditional probability measure $\PP=\prob(\,\cdot \mid 0\in\mathcal C_\infty)$, sharp (in particular, uniform in $x$) moment estimates for the
spatial corrector on the infinite cluster are available; see Dario
\cite{dario2018optimal}.  Moreover, the required volume lower deviation for
$|B_n|$ holds with exponential tails; see Penrose-Pisztora
\cite{penrose1996large}.  Consequently, the assumptions of
Lemma~\ref{lem:transfer-principle} are satisfied and we obtain a
stationary potential $\phi$ on $\Omega_0$ with $\chi=\mD\phi$ and
Condition~\ref{cond:moment-potential}.

In addition, quenched heat kernel upper bounds for the random walk on
$\mathcal C_\infty(\omega)$ were proved by Barlow \cite{barlow2004random} (strictly speaking the result of \cite{barlow2004random} deals with the case $a=b=1$ but the adaptation to the weighted supercritical percolation $0<a\le b<\infty$ is straight-forward). 
Combined with the spatial moment estimate from Lemma~\ref{lem:transfer-principle}(i),
this verifies the assumptions of Corollary~\ref{cor:uniform-quenched-moments},
and hence yields
\[
\sup_{t\ge 0}E_0^\omega\big[\|\chi(\omega,X_t)\|^q\big]<\infty
\qquad\text{for $\prob$-a.e.\ $\omega$.}
\]
\end{example}

\section{Key lemma in unified settings}
In order to pinpoint what allows the lifted invariance principle in each of the two settings, annealed and quenched, our strategy is to formulate sufficient conditions in unified settings, while limiting the conditions to the level of the process as much as possible. This may be of independent interest as it may be useful for proving similar results for other models.

For $f,g\in D([0,T],\bbR)$ we shall use the notation  
\begin{equation}\label{eq:Istfg}
  I_{s,t}(f,g)
  \;\coloneqq\;
  \int_{(s,t]} f_{u-}\,\md g_u \;-\; f_s\, g_{s,t},
  \qquad 0\le s<t\le T,
\end{equation}
that is, the left–point Riemann integral of $f$ with respect to $g$
on $(s,t]$ defined in \eqref{def:left-point-integral}, after subtracting the term $f_s\,g_{s,t}$. 

\smallskip
\noindent
In particular, when $f$ and $g$ are piecewise constant c\`adl\`ag functions,
\begin{align}\label{eq:def-I}
  I_{s,t}(f,g)
  \;=\;
  \sum_{s<u\le t} f_{u-}\, g_{u-,u}
  \;-\;
  f_s\, g_{s,t},
\end{align}
where the sum runs over all jump times of $g$ in $(s,t]$.

\smallskip
\noindent
Moreover, summation by parts yields the identity
\begin{equation}\label{eq:summation-by-parts}
  f_{s,t}\, g_{s,t}
  \;=\;
  I_{s,t}(f,g)
  \;+\;
  I_{s,t}(g,f)
  \;+\;
  Q_{s,t}(f,g),
\end{equation}
where $Q_{s,t}(f,g)$ denotes the quadratic covariation term
\begin{align}\label{eq:def-Q}
  Q_{s,t}(f,g)
  \;\coloneqq\;
  \sum_{s<u\le t} f_{u-,u}\, g_{u-,u}.
\end{align}

For $X,Y\in D([0,T],\bbR^d)$ we shall use consistently the abuse of notation 
$I(X,Y)$ and $Q(X,Y)$ for the matrix whose entries are $I(X^i,Y^j)$ and $Q(X^i,Y^j)$, respectively so that 
\eqref{eq:summation-by-parts} becomes 
\begin{equation}\label{eq:multidim-summation-by-parts}
  I_{s,t}(X,Y) + I_{s,t}(Y,X)^\top
  \;=\;
  X_{s,t}\,\otimes Y_{s,t}
  \;-\;
  Q_{s,t}(X,Y),
\end{equation}
where for a matrix $A$ we write $A^\top$ for its transpose.

We will now see that, in both of our settings, the conditions are significantly simplified. For the first level, in addition to the Lindeberg condition and control of the quadratic variation of the martingale, we need the vanishing limit of the $p$-variation norm of the corrector. For the second level, we require $p/2$-variation tightness and convergence in probability in the uniform norm for the iterated integral of the corrector, as well as vanishing limiting quadratic covariation of the corrector and the martingale in the $p/2$-variation norm.
Here is the statement.

\begin{lemma}\label{lem:abstract-limit-specific}
  Let $X = (X_t : t \geq 0), M = (M_t : t \geq 0), \text{ and } R = (R_t : t \geq 0)$ be c\`adl\`ag random processes on $\bbR^d$ with respect to the filtration $\cF=\{\cF_t: t\ge 0\}$, so that $X$ is $\mathbb{Z}^{d}$-valued, and $X_t = M_t+R_t$. Assume further 
  that $M$ is a square integrable martingale and 
  that $X_0=M_0=R_0=0$ almost surely. Let 
  \begin{align}\label{eq:rescaled-X-M-R}
    M_t^n \ldef \frac{1}{\sqrt n} M_{n t},
    \; 
    R_t^n \ldef \frac{1}{\sqrt n} R_{n t}
    \; \text{ and }
    X_t^n \ldef M_t^n+R_t^n = \frac{1}{\sqrt n} X_{n t}.
  \end{align}
  Define 
  \begin{align}\label{eq:rescaled-I-for-X-M-R}
    \bbM^n \ldef I(M^n,M^n),
    \; 
    \bbR^n \ldef I(R^n,R^n),
    \; \text{ and }
    \bbX^n \ldef I(X^n,X^n).
  \end{align}
  Let $p>2$.  Assume the following three conditions hold:  
  \smallskip

  \noindent\textbf{(1) Martingale.}
  \begin{enumerate}[(i)]

  \item Uniformly controlled variations. The UCV condition holds:
    \begin{equation}\label{ass:UCV}
      \sup_{n} \bbE \left[ [M^{n,i},M^{n,j}]_{T} \right]  %\rightarrow  (\Si^2)^{i,j} \,T, \quad \text{ a s} n\to\infty,
      \le C \, T.
    \end{equation}

  \item Limit of quadratic variations. There is a $d\times d$ symmetric real matrix $\Si^2$ so that for all $1\le i,j\le d$ and $T>0$
    \begin{equation}\label{ass:LimitQV}
      \big\|\big( [M^{n,i},M^{n,j}]_{t} -
      (\Si^2)^{i,j} \,t\big)_{t\in[0,T]} \big\|_{\mathrm{unif}, [0,T]} \xrightarrow[n\to\infty]{\;\prob\;} 0, 
    \end{equation}
    where $[M^{n,i},M^{n,j}]_t=Q_{0,t}(M^{n,i},M^{n,j})=\sum_{0 < s \leq t}  M^{n,i}_{s-,s}\, M^{n,j}_{s-,s}$ denotes the quadratic covariation of $M^{n,i}$ and $M^{n,j}$, the $i$-th and $j$-th, respectively, component of the process $M^n$.

  \item Lindeberg condition:
    for all $\delta > 0$, $T>0$ and $v\in\bbZ^d$  
the compensator of          \begin{equation}\label{ass:Lindeberg}
      \sum_{0 < s \leq T}  (v\cdot M^{n}_{s-,s})^2 \mathbbm{1}_{|v\cdot M^{n}_{s-,s}| > \delta} %\xrightarrow[n\to\infty]%{\;\prob\;} 0. 
    \end{equation}
    converges to $0$ in $\prob$-probability as $n\to\infty$.
     \end{enumerate}

  \noindent\textbf{(2) Corrector.}
  \begin{enumerate}[(i)]
  \item For every $i$ we have
    \begin{align*}
      \E \big[\,\|R^{n,i}\|_{p\text{-}\mathrm{var},[0,T]}\,\big]
      \;\longrightarrow\; 0
      \quad\text{as }n\to\infty.
    \end{align*}
  \item For every $i$ and $j$ the family $\{\|(\bbR^n)^{i,j}\|_{p/2\text{-}\mathrm{var},[0,T]}\}_{n\ge1}$ is tight under $\prob$.

  \item There is a deterministic $d\times d$ real matrix $\Gamma$ so that for every $i$ and $j$ we have
    \begin{align*} 
      \big\|    \big( (\bbR^n_{s,t})^{i,j} - (t-s)\Ga^{i,j}\big)_{(s,t)}
      \big\|_{\infty\text{-}\mathrm{var},[0,T]}
      \;\longrightarrow\; 0
      \quad\text{in }\prob\text{-probability as }n\to\infty.
    \end{align*}
  \end{enumerate}

  \smallskip
  \noindent\textbf{(3) Mixed terms covariation}:

  \noindent Vanishing quadratic covariations of mixed terms in $p/2$-variation norm:
  With
  \begin{align*}
    % \bbV^n \ldef I(R^n,M^n), 
    % \qquad 
    \bbQ^n_{s,t} \ldef Q_{s,t} (M^n,R^n), \quad (s,t)\in\Delta_{[0,T]},
  \end{align*}
  where $Q$ is defined in (\ref{eq:def-Q}), we have
  \begin{align*}
    \E\!\Big[
    \Norm{\bbQ^n}{p/2\text{-var},[0,T]}^{1/2}
    \Big]
    \xrightarrow[n\to\infty]{} 0.
  \end{align*}
  \smallskip
  Then for every $p<p'<3$
  the sequence 
  \begin{align*}
    \big( \Norm{(X^n,\bbX^n)}{p'\text{-var},[0,T]}\big)_{n\ge 1} 
  \end{align*} 
  is tight and moreover
  \begin{align*}
    (X^n,\bbX^n)
    \;\underset{n \to \infty}{\;\Longrightarrow\;} 
    \big(B,\; \bbB + \cdot\Ga\big)
    \quad\text{in }
    \cD_{p'\text{-var}}([0,T], \bbR^d\times \bbR^{d \times d}),
  \end{align*}
  where $B$ is a Brownian motion with covariance matrix $\Si^2$, 
  $\bbB$ is its It\^o iterated integral and
  \begin{align*}
    \big(B,\; \bbB + \cdot\Ga\big) = 
    \big((B_t)_{t\in [0,T]},\; (\bbB_{s,t} + (t-s)\Ga)_{(s,t)\in\Delta_{[0,T]}}\big).
  \end{align*}
  Further, assume that $\Ga$ is symmetric.
  Let 
  \begin{align*}
    \bar\bbX^n_{s,t} \ldef \bbX^n_{s,t}+\frac12 Q_{s,t}(X^n,X^n).
  \end{align*}
  Then, for every $p<p'<3$
  \begin{align*}
    (X^n,\bar\bbX^n)
    \;\underset{n \to \infty}{\;\Longrightarrow\;} 
    \big(B,\; \bbB^{\mathrm{STR}} \big)
    \quad\text{in }
    \cD_{p'\text{-var}}([0,T], \bbR^d\times \bbR^{d \times d}),
  \end{align*}
  where $B$ is as above, a Brownian motion with covariance matrix $\Si^2$, and $\bbB^{\mathrm{STR}}$
  is its Stratonovich iterated integral.
\end{lemma}

We shall now prepare for the proof of the Lemma, which will be given in the end of this section. The starting point is the abstract result \cite[Proposition 6.1]{FZ18}. We shall use a slight modification of the formulation of \cite[Lemma 2.3]{deuschel2021additive}, with the difference being that when convergence in probability is assumed rather than convergence in law, then it is the case also in the assertion. Here is the formulation.
\begin{proposition}[Friz-Zhang 2018]  \label{lem:rp-conv}
  Let $ (Z^n, \mathbb{Z}^n)$ be a sequence of c\`adl\`ag rough paths and let $p<3$. Assume that
  there exists a c\`adl\`ag rough path $(Z, \mathbb{Z})$ such that $(Z^n,
  \mathbb{Z}^n_{0,\cdot}) \rightarrow (Z, \mathbb{Z}_{0,\cdot})$ in law (or probability) in the Skorohod (resp. uniform) topology and that the family of real valued
  random variables $(\| (Z^n, \mathbb{Z}^n) \|_{p, [0, T]})_n$ is tight. Then
  $(Z^n, \mathbb{Z}^n) \rightarrow (Z, \mathbb{Z})$ in law (or probability) in the
  $p'$-variation Skorohod (resp. uniform) topology for all $p' \in (p, 3)$.
\end{proposition}
%\begin{proof} As mentioned in the proof of \cite[Lemma 2.3]{deuschel2021additive} this follows from a simple interpolation argument as in the proof of Theorem 6.1 in \cite{FZ18}.
%\end{proof}

\begin{lemma}\label{lem:abstract-limit-general}
  Let $X = (X_t : t \geq 0), M = (M_t : t \geq 0),R = (R_t : t \geq 0)$ be c\`adl\`ag random processes on $\bbR^d$ with respect to the filtration $\cF=\{\cF_t: t\ge 0\}$, so that $X$ is $\mathbb{Z}^{d}$-valued, and $X_t = M_t+R_t$. Assume further 
  % that $M$ is a squared integrable martingale and 
  that $X_0=M_0=R_0=0$ almost surely. Let 
  $M^n,R^n,X^n,\bbM^n,\bbR^n$ and $\bbX^n$ be defined as in (\ref{eq:rescaled-X-M-R}) and (\ref{eq:rescaled-I-for-X-M-R}).
  Let $p>2$.  Assume the following three conditions hold:  
  \smallskip

  \noindent\textbf{(a) The component $M$.}
  \begin{align*}
    \big(\Norm{(M^n,\bbM^n)}{p\text{-var},[0,T]}\big)_n
  \end{align*}
  is a tight sequence and furthermore
  \begin{align*}
    (M^n , \, \bbM^n_{0,\cdot})
    \;\underset{n \to \infty}{\;\Longrightarrow\;} 
    (B,\bbB_{0,\cdot})
    \quad\text{in } 
    D\!\big([0,T],\R^d\times\R^{d\times d}\big),
  \end{align*}
  where $B$ is a Brownian motion with covariance matrix $\Si^2$ and 
  $\bbB$ is its It\^o iterated integral.

  \smallskip
  \noindent\textbf{(b) $R$ and mixed terms vanish in the uniform norm.}
  With
  \begin{align*}
    \bbW^n \ldef I(R^n,M^n) + I(M^n,R^n),
  \end{align*}
  we have
  \begin{align*}
    \big\|
    (R^n,\; \bbR^n_{0,\cdot} + \bbW^n_{0,\cdot} - t\Ga)
    \big\|_{\mathrm{unif}, [0,T]}
    \;\xrightarrow[n\to\infty]{\;\prob\;} 0,
  \end{align*}
  where
  $\Gamma$ is a deterministic $d\times d$ real matrix.

  \smallskip
  \noindent\textbf{(c) Uniform $p$-variation control of $R$ and mixed terms.}

  \noindent (i) Uniform bound on the pair $(R^n,\bbR^n)$:
  \begin{align*}
    \sup_{n}\,
    \E\!\Big[
    \Norm{(R^n,\bbR^n)}{p\text{-var},[0,T]}
    \Big] 
    < \infty .
  \end{align*}

  \noindent (ii) Tightness of $\bbW^n$:
  \begin{align*}
    \text{ every entry of} 
    \quad
    \big\{
    \Norm{\bbW^n}{p/2\text{-var},[0,T]}
    \big\}_{n\ge1}
    \quad \text{is a tight sequence of real valued random variables}.
  \end{align*}
  \smallskip
  Then for every $p<p'<3$
  \begin{align*}
    (X^n,\bbX^n)
    \;\underset{n \to \infty}{\;\Longrightarrow\;} 
    \big(B,\; \bbB + \cdot\Ga\big)
    \quad\text{in }
    \cD_{p'\text{-var}}([0,T], \bbR^d\times \bbR^{d \times d}),
  \end{align*}
  where 
  \begin{align*}
    \big(B,\; \bbB + \cdot\Ga\big) = 
    \big((B_t)_{t\in [0,T]},\; (\bbB_{s,t} + (t-s)\Ga)_{(s,t)\in\Delta_{[0,T]}}\big).
  \end{align*}
\end{lemma}

\begin{proof}
  % By the monotonicity in $p$ of the $\|\cdot\|_{p\text{-var},[0,T]}$, see (\ref{eq:inequalities on norms}), and since all the processes are assumed to start at 0, we have that   
  % \begin{align*}
  %   (M^n , \, \bbM^n)
  %   \;\underset{n \to \infty}{\;\Longrightarrow\;} 
  %   (B,\bbB)
  %   \quad\text{in } 
  %   D_{p\text{-var}}\!\big([0,T],\R^d\times\R^{d\times d}\big)
  % \end{align*}
  % 
  % 
  % Since a.s.\ convergence implies also convergence in probability, 
  We have 
  \begin{align*}
    \mathbf{Y}^n\ldef (R^n,\; \bbR^n_{0,\cdot} +  \bbW^n_{0,\cdot})
    \;\xrightarrow[n\to\infty]{\;\bbP\;} (0,\cdot\Ga)
  \end{align*}
  in the uniform topology. 
  Therefore, together with 
  \begin{align*}
    (M^n , \, \bbM^n_{0,\cdot})
    \;\underset{n \to \infty}{\;\Longrightarrow\;} 
    (B,\bbB_{0,\cdot})
    \quad\text{in } 
    D\!\big([0,T],\R^d\times\R^{d\times d}\big),
  \end{align*}
  we have by Slutsky's theorem in the Skorohod topology, Theorem~\ref{thm:slutsky}, 
  that  
  \begin{align*}
    (X^n , \, \bbX^n_{0,\cdot})
    = (M^n + R^n, \, \bbM^n_{0,\cdot} + \bbR^n_{0,\cdot} + \bbW^n_{0,\cdot})
    \;\underset{n \to \infty}{\;\Longrightarrow\;} 
    (B,\bbB_{0,\cdot}+{\cdot}\Ga)
  \end{align*}
  in $ D\!\big([0,T],\R^d\times\R^{d\times d}\big)$.
  On the other hand, since 
  \begin{align*}
    (X^n_{s,t},\bbX^n_{s,t})
    =
    (M^n_{s,t}+R^n_{s,t}, \bbM^n_{s,t} + \bbR^n_{s,t} + \bbW^n_{s,t})
    =
    (R^n_{s,t},\bbR^n_{s,t}) + (M^n_{s,t}, \bbM^n_{s,t})   + (0,\bbW^n_{s,t}),
  \end{align*} 
  and $X_0=M_0=R_0=0$ so that 
  \begin{align*}
    \Norm{(X^n,\bbX^n)}{p\text{-var},[0,T]} = \Norm{X^n}{p\text{-var},[0,T]} + \Norm{\bbX^n}{p/2\text{-var},[0,T]}^{1/2},  
  \end{align*}
  tightness of $\left(\Norm{(X^n,\bbX^n)}{p\text{-var},[0,T]}\right)_n$ follows from that of 
  $\left(\Norm{(M^n,\bbM^n)}{p\text{-var},[0,T]}\right)_n$ and 
  $\left(\Norm{(R^n,\bbR^n)}{p\text{-var},[0,T]}\right)_n$ together with that of
  $\left(\Norm{\bbW^n}{p/2\text{-var},[0,T]}\right)_n$.
  The proof is now concluded by applying Proposition \ref{lem:rp-conv}.    
\end{proof}

Next, we go back to the settings of Lemma~\ref{lem:abstract-limit-specific}. In particular $M$ is a a c\`adl\`ag square integrable $d$-dimensional martingale with $M_0=0$. 
We remark that the Burkholder-Davis-Gundy inequality together with L\'epingle inequality implies the following.
For any $p > 2$ there exists $c = c(p) < \infty$ such that
\begin{align}\label{eq:Lepingle:BDG}
  \E_0\!\big[
  \|M^n\|_{p\text{-}\mathrm{var}, [0, T]}
  \big]
  &\;\leq\;
  c\,
  \E_0\!\big[
  \|M^n\|_{\infty\text{-var}, [0, T]}
  \big]
  % \nonumber\\begin{align*}1ex]
  % &\;\leq\;
  % c\,
  % \E_0\!\big[
  % \|M^n\|_{\infty\text{-var}, [0, T]}^2
  % \big]^{1/2}
  \;\leq\;
  2c\,
  \E_0\!\big[
  [ M^n ]_T
  \big]^{1/2}.
\end{align}
We will refer to the inequalities in \eqref{eq:Lepingle:BDG} as the L\'epingle-BDG-inequality.
We will need the following Theorem by Friz and Zorin-Kranich which is an improvement of L\'epingle-BDG-inequality.
\begin{theorem} [A special case of \cite{friz2020rough} Theorem 1.1]\label{thm:Friz-Zorin-Kranich}
  Let $p>2$ and let $2<p_1<p$ so that $\frac2p < \frac12 + \frac{1}{p_1}$.
  Let $G$ be a c\`adl\`ag adapted process and $M$ be a c\`adl\`ag martingale, both with respect to same filtration $\{ \cF_t : t \geq 0 \}$.
  Set
  \begin{align*}
    I_{s,t} \coloneqq I_{s,t}(G,M), %\int_s^{t} G_{s,u-} \, \md M_u.
  \end{align*}
  where $I_{s,t}(\cdot,\cdot)$ was defined in (\ref{eq:def-I}).
  Then,
  \begin{align}\label{eq:Friz-Zorin-Kranich}
    \E\big[
    \big\|
    I
    \big\|_{p/2\text{-}\mathrm{var}, [0,T]}
    ^\frac12 \big]
    \;\leq\;
    C_0(p,p_1) \
    \E\bigg[
    \big\|
    G
    \big\|_{p_1\text{-}\mathrm{var}, [0,T]}
    \bigg] ^\frac12
    \E\bigg[
    \big[
    M
    \big]_T
    ^\frac12
    \bigg] ^\frac12,
  \end{align}
  where $C_0(p,p_1)>0$ is a constant depending only on $p$ and $p_1$.
\end{theorem}

We are now ready to prove Lemma~\ref{lem:abstract-limit-specific}. 

\begin{proof}[Proof of  Lemma~\ref{lem:abstract-limit-specific}]
  The proof of the It\^o part follows from Lemma~\ref{lem:abstract-limit-general} once we show conditions (a),(b) and (c) of that lemma. We will start with the martingale part. First, the sequence of martingales $(M^n)_{n \in \bbN}$ converges in law in the Skorohod topology $\cD ([0, T], \bbR^d)$ to a Brownian motion with deterministic covariance matrix $\Si^2$ as defined in \eqref{eq:def:Si+Ga}. Indeed, this is a consequence of the martingale functional central limit theorem by Helland, see \cite[Theorem~5.1a)]{He82} using the limiting quadratic variations (\ref{ass:LimitQV}) together with the Lindeberg condition (\ref{ass:Lindeberg}). Using the UCV condition (\ref{ass:UCV}), condition (a) of Lemma~\ref{lem:abstract-limit-general} is a consequence of Theorem 4.11 of Chevyrev-Friz \cite{ChevyrevFriz2019} (whose proof can be seen by applying the main result of Kurtz-Protter \cite{KP91} together with the L\'epingle-BDG-inequalities (\ref{eq:Lepingle:BDG}) and (\ref{eq:Friz-Zorin-Kranich})).

  Next, we treat $\bbW^n$.
  Applying the summation by parts identity (\ref{eq:multidim-summation-by-parts}), we have
  for all $(s,t)\in\Delta_{[0,T]}$,
  \begin{equation}\label{eq:IBP-MR}
    I_{s,t}(M^n,R^n)
    +
    I_{s,t}(R^n,M^n)^\top
    \;=\;
    M^n_{s,t}\otimes R^n_{s,t}
    -
    Q_{s,t}(M^n,R^n),
  \end{equation}
  where,
 as defined in (\ref{eq:def-Q}),
\begin{align*}
    Q_{s,t}(M^n,R^n)
    \;=\;
    \sum_{s<u\le t} M^n_{u-,u} \otimes  R^n_{u-,u}.
  \end{align*}
  Thus, with 
  \begin{align*}
    \bbV^{n}\ldef I(R^n,M^n),
  \end{align*}
  we have
  \begin{align*}
    \bbW^n_{s,t} 
    & \;=\;
    I_{s,t}(R^n,M^n)
    \;+\;
    I_{s,t}(M^n,R^n) \\
    & \;=\;
    \bbV^{n}_{s,t}   
    \;-\; 
    (\bbV^{n})^\top_{s,t}
    \;+\; 
    M^n_{s,t} \otimes R^n_{s,t}
    \;-\; 
    Q_{s,t}(M^n,R^n).
  \end{align*}
  By the Vanishing quadratic covariations assumption (3) we have
  \begin{align*}
    \E\!\Big[
    \Norm{
      \bbQ^n}{\infty\text{-var},[0,T]}^{1/2}
    \Big]
    \xrightarrow[n\to\infty]{} 0.
  \end{align*}
  Hence it is left to show that \begin{align*}
    \E\!\Big[
    \Norm{\bbV^n}{\infty\text{-var},[0,T]}^{1/2}
    \Big]
    \xrightarrow[n\to\infty]{} 0 
    \quad 
    \text{ and }
    \quad 
    \E\!\Big[
    \Norm{(M^n_{s,t} \otimes R^n_{s,t})_{(s,t)}}{\infty\text{-var},[0,T]}^{1/2}
    \Big]
    \xrightarrow[n\to\infty]{} 0. 
  \end{align*}

  By Theorem~\ref{thm:Friz-Zorin-Kranich} with
  \begin{align*}
    G \coloneqq R^n,
    \qquad
    M \coloneqq M^n,
  \end{align*}
  for any $2<p_1<p$ with $\frac{2}{p}<\frac{1}{2}+\frac{1}{p_1}$ there exists $C_0=C_0(p,p_1)>0$ such that
  \begin{align*}
    \E\big[
    \|(\bbV^n)^{i,j}\|_{p/2\text{-}\mathrm{var},[0,T]}^{1/2}
    \big]
    \;\le\;
    \E\big[\,\|R^{n,i}\|_{p_1\text{-}\mathrm{var},[0,T]} \big]^{1/2}
    C_0\,\sup_m \E\big[\,[ (M^m)^j]_T \big]^{1/4} 
  \end{align*}
  which vanishes as $n\to\infty$ by the UCV assumption (\ref{ass:UCV}) together with assumption (i) of (2) Corrector. 
  On the other hand, assumptions (2) (i) and (iii) on the Corrector imply for every $i$ and $j$ we have
  \begin{align*} 
    \big\| 
    \big( 
    R^{n,i}_t,   
    (\bbR^n_{0,t})^{i,j} - t\Ga^{i,j}
        \big)_{t}
    \big\|_{\mathrm{unif}, [0,T]}
    \;\xrightarrow[n\to\infty]{\,\bbP\,}\; 0.
  \end{align*}
  Additionally, Cauchy-Schwarz inequality gives
  \begin{align*}
    \Norm{M \otimes R}{p/2\text{-var},[0,T]}
    \le 
    \,  \Norm{M}{p\text{-var},[0,T]}
    \otimes
    \Norm{R}{p\text{-var},[0,T]}. 
  \end{align*}
  Thanks to the UCV assumption (\ref{ass:UCV}), L\'epingle-BDG-inequality \eqref{eq:Lepingle:BDG} and assumption (i) of (2) Corrector, we can apply Cauchy-Schwarz inequality again to obtain 
  \begin{align*}
    \E\big[ \Norm{M^n \otimes R^n}{p/2\text{-var},[0,T]}^{1/2} \big] 
    \;\le\;
    \E\big[\,\|R^n\|_{p\text{-}\mathrm{var},[0,T]} \big]^{1/2}
   \otimes \,\sup_m \E\big[\,\|M^m\|_{p\text{-}\mathrm{var},[0,T]} \big]^{1/2}  
  \end{align*}
  which vanishes as $n\to\infty$. Thus we have proved simultaneously condition (b) and condition (c) (ii) of Lemma~\ref{lem:abstract-limit-general} 
  \begin{align*}
    \big\|
    (R^n,\; \bbR^n_{0,\cdot} + \bbW^n_{0,\cdot} - \cdot\Ga)
    \big\|_{\mathrm{unif}, [0,T]}
    \;\xrightarrow[n\to\infty]{\;\prob\;} 0 
  \end{align*}
  and 
  \begin{align*}
    \{\Norm{\bbW^n}
    {p/2\text{-var},[0,T]}\}_n\quad \text { is tight. }
  \end{align*}
  Finally, the uniform bound on the corrector pair, condition (c) (i) of Lemma~\ref{lem:abstract-limit-general} follows by assumption (2) (i) and (ii). 
  Hence we have shown the conditions of Lemma~\ref{lem:abstract-limit-general} hold and this concludes the proof of the It\^o part. 

  For the Stratonovich part, first note that the Stratonovich and It\^o iterated integrals of $B$ differ by half the quadratic variation:   
  \begin{align*}
    \bbB^{\mathrm{STR}}_{s,t}\;=\;\bbB_{s,t} + \frac12 (t-s)\Sigma^2.
  \end{align*}
  Next, summation by parts (\ref{eq:multidim-summation-by-parts}) gives us
  \begin{equation*}
    I_{s,t}(X^n,X^n) + \frac12 Q_{s,t}(X^n,X^n)
    \;=\;
    \frac12 I_{s,t}(X^n,X^n)
    - \frac12 I_{s,t}(X^n,X^n)^\top
    - \frac12 X^n_{s,t}\otimes X^n_{s,t}. 
  \end{equation*}
  Hence, the sequence 
  \begin{align*}
    \left(\Norm{(X^n,\bar\bbX^n)}{p\text{-var},[0,T]}\right)_n  
  \end{align*}
  is tight. 
  Now, 
  \begin{align*}
    Q (X^n,X^n) = Q (M^n,M^n) + Q (R^n,R^n) +
    Q (M^n,R^n) + Q (M^n,R^n)^\top.
  \end{align*}
  We know
  \begin{align*}
    \E\!\Big[
    \Norm{
      Q (M^n,R^n)}{\infty\text{-var},[0,T]}^{1/2}
    \Big]
    \xrightarrow[n\to\infty]{} 0.
  \end{align*}
  Further
  \begin{align*}
    Q (R^n,R^n) = - I(R^n,R^n)- I(R^n,R^n)^\top + R^n\otimes R^n.  
  \end{align*}
  and Cauchy-Schwartz inequality together with monotonicity of the p-variation norm in $p$ yields
  \begin{align*}
    \E\!\Big[
    \Norm{
      (R^n\otimes R^n)}{\infty\text{-var},[0,T]}^{1/2}
    \Big]
    \leq
    \E\!\Big[
    \Norm{
      R^n}{p\text{-var},[0,T]}
    \Big] ^{1/2} \otimes
    \E\!\Big[
    \Norm{ R^n}{p\text{-var},[0,T]}
    \Big]^{1/2}
    \xrightarrow[n\to\infty]{\bbP} 0
  \end{align*}
and
  \begin{align*}
    \Norm{
      \big( Q_{0,\cdot} (R^n,R^n) + \cdot(\Ga + \Ga^\top)\big)
    }{\text{unif},[0,T]}
    \xrightarrow[n\to\infty]{\bbP} 0.
  \end{align*}
  Therefore, since $\Ga$ is symmetric,   
  \begin{align*}
    \Norm{
      \big( 0, \frac12 Q_{0,t} (X^n,X^n)- (\frac12\Si^2-\Ga) t \big)
    }{\text{unif},[0,T]}
    \xrightarrow[n\to\infty]{\bbP} 0
  \end{align*} 
 Therefore, Slutsky's theorem in the Skorohod topology, Theorem~\ref{thm:slutsky}, gives us 

  \begin{align*}
    (X^n , \, \bar\bbX^n_{0,\cdot})
    = (X^n, \, \bbX^n_{0,\cdot} + \frac12 Q_{0,\cdot}(X,X))
    \underset{n \to \infty}{\Longrightarrow}
    (B,\bbB_{0,\cdot}+{\cdot}(\Ga + \frac12\Si^2-\Ga))
    = (B,\bbB_{0,\cdot}^{\mathrm{STR}}) \end{align*}
  in $ D\!\big([0,T],\R^d\times\R^{d\times d}\big)$.
  Applying Proposition \ref{lem:rp-conv}, the proof is now completed.   
\end{proof}

\section{The annealed settings - qualitative approach}
In this section we will prove Theorem \ref{thm:annealed}. Recall we suppose that Assumption~\ref{ass:general} holds true.

\subsection{Harmonic coordinates}
Recall that $X_t = \Pi(\om, X_t) = \Phi(\om, X_t)+ \chi(\om, X_t)$. 
Define, for any $t \geq 0$ and $\om \in \Om$, 
\[
M_t \ldef \Phi(\om, X_t)\indizeroo\quad \text{and} \quad 
R_t \ldef \chi(\om, X_t)\indizeroo.
\]  
In view of \eqref{eq:harmonic}, it follows that, for $\prob$-a.e.\ $\om$, the process $M \coloneqq (M_t : t \geq 0)$ is a $(P_0^{\om}, \{\cF_t : t \geq 0\})$-martingale, where $\cF_t \ldef \si(X_s : s \leq t)$ is the natural filtration generated by the process $X$.  In particular, the predictable quadratic variation process is given by
\begin{align}\label{eq:qv_process}
  \langle M^i, M^j \rangle_t
  \;=\;
  \int_0^t
  \Big(
  {\textstyle \sum_{x \in \mathbb{Z}^{d}}}\, \om(\{0,x\})\,
  \Phi^i(\om, x)\, \Phi^j(\om, x)
  \indizeroo \Big) \circ \tau_{X_s}\;
  \md s.
\end{align}
Note that the predictable quadratic variation process is written in terms of the environment process $(\tau_{X_t} \om : t \geq 0)$ which is a Markov process taking values in $\Om$ with generator $\cL\!: \mathop{\mathrm{dom}}(\cL) \to L^2(\Om,\prob)$,
\begin{align*}
  (\cL \vp)(\om)
  \;=\;
  \sum_{x \in \mathbb{Z}^{d}} \om(\{0,x\})\, \big( \vp(\tau_x \om) - \vp(\om) \big)
\end{align*}
where $\mathop{\mathrm{dom}}(\cL) = \{  \vp \in L^2(\prob) : \|\mD \vp\|_{L^2_{\mathrm{cov}}} < \infty\}$.  The following lemma summarizes the properties of the environmental process, cf. \cite[Prop.~2.3]{Bi11} and \cite[Prop.~2.1]{ADS15}.
\begin{lemma}\label{lemma:inv:environment_process}
  The measure $\prob$ is stationary, reversible and ergodic for $(\tau_{X_t} \om : t \geq 0)$.
\end{lemma}
For $1\le i,j\le d$ we denote 
\begin{align*}
  [M^{i},M^{j}]_t=Q_{0,t}(M^{i},M^{j})=\sum_{0 < s \leq t}  M^{i}_{s-,s}\, M^{j}_{s-,s}
\end{align*} 
the quadratic covariation of $M^{i}$ and $M^{j}$, the $i$-th and $j$-th component of the process $M$, respectively.

\begin{prop}[Martingale part: annealed Lindeberg, UCV and QV limit]\label{prop:first_level_conv_of_hamonic}
  For any $p > 2$ and $T > 0$, the sequence of $d$-dimensional martingales $(M^n)_{n \in \bbN}$ satisfies:
  \begin{enumerate}[(i)]

  \item Uniformly controlled variations. The UCV condition holds:
    \begin{equation}\label{annealed:UCV}
      \sup_{n} \bbE_0 \left[ [M^{n,i},M^{n,i}]_{T} \right]  %\rightarrow  (\Si^2)^{i,j} \,T, \quad \text{ a s} n\to\infty,
      \le C \, T,
    \end{equation}
    where $M^{n,i}$ is the $i$-th component of the process $M^n$.
  \item Limit of quadratic covariations. There is a $d\times d$ symmetric real matrix $\Si^2$ so that for all $1\le i,j\le d$ and $T>0$
    \begin{equation}\label{annealed:LimitQV}
      \big\|\big( [M^{n,i},M^{n,j}]_{t} -
      (\Si^2)^{i,j} \,t\big)_{t\in[0,T]} \big\|_{\mathrm{unif}, [0,T]} \xrightarrow[n\to\infty]{\;\PP_0\;} 0. 
    \end{equation}

  \item\label{annealed:Lindeberg} Lindeberg condition:
    for all $\delta > 0$, $T>0$ and $v\in\bbZ^d$  
    the compensator of \begin{equation*}
      \sum_{0 < s \leq T}  (v\cdot M^{n}_{s-,s})^2 \mathbbm{1}_{|v\cdot M^{n}_{s-,s}| > \delta} %\xrightarrow[n\to\infty]%{\;\p_0\;} 0.  
    \end{equation*}
converges to $0$ in $\p_0$-probability as $n\to\infty$.
  \end{enumerate}
\end{prop}

\begin{proof}
  (i) follows from stationarity. For example, it can be read from the formula for predictable quadratic variation (\ref{eq:qv_process}) and the fact its expected value coincides with that of the quadratic variation (since their difference is a martingale). Hence
  \begin{align*}
    \E_0\big[\big[ M^{n,i}, M^{n,j} \big]_T\big]
    & \;=\; 
    \E_0\big[\langle M^{n,i}, M^{n,j} \rangle_T\big]\\
    & \;=\;
    \frac{1}{n}\E_0\bigg[\int_0^{nT}
    \Big(
    {\textstyle \sum_{x \in \mathbb{Z}^{d}}}\, \om(\{0,x\})\,
    \Phi^i(\om, x)\, \Phi^j(\om, x)
    \indizeroo\Big) \circ \tau_{X_s}\;
    \md s\bigg].
    \end{align*}
    Therefore,
     \begin{align*}
    \EE_0\big[\big[ M^{n,i}, M^{n,j} \big]_T\big]
    &\;=\;
    T    \langle \Phi^i,\Phi^j \rangle_{L_{\mathrm{cov}}^2(\PP)}  .
  \end{align*}

  For (ii), Limit of quadratic covariations, for all $1\le i,j\le d$ this is a special case of Lemma~\ref{cor:ergodic_p-var} applied to the processes $\Psi=\Phi^i$ and $\Xi=\chi^j$. Indeed, by the monotonicity of the $q$-variations norms we can take the uniform norms which converge in $L^1(\PP_0)$ and in particular in $\PP_0$-probability.

  It remains to show the Lindeberg condition (iii). 
For this purpose, let us first recall the L\'evy system theorem, cf.~\cite[Theorem~VI.28.1]{RW00}. In our settings it states that for any measurable function $f\!: \mathbb{Z}^{d} \times \mathbb{Z}^{d} \to \bbR$ that vanishes on the diagonal, for $\prob$-a.e.\ $\om$
  \begin{align}\label{eq:Levy:system}
    E_{0}^{\omega}
    \bigg[
    \sum_{0 < s \leq T} f(X_{s-}, X_s)
    \bigg]
    \,=\,
    E_{0}^{\omega}
    \bigg[
    \int_{(0,T]}
    \sum_{y \in \mathbb{Z}^{d}} \om(\{X_{s-}, y\})\, f(X_{s-}, y)\,
    \md s
    \bigg]
  \end{align}
  and moreover,
  the process
  \begin{align}\label{eq:Levy:system:martingale}
    \sum_{0 < s \leq t} f(X_{s-}, X_s)
    \,-\,
    \int_{(0,t]}
    \sum_{y \in \mathbb{Z}^{d}} \om(\{X_{s-}, y\})\, f(X_{s-}, y)\,
    \md s
  \end{align}
  is a local $E_{0}^{\omega}$-martingale whenever the left term has a $E_0^{\om}$ finite mean.
In fact, recall that, by definition, $M_t \ldef \Phi(\om, X_t)\indizeroo$ and $\Phi(\om,0) = 0$ for $\prob$-a.e. $\omega$.  Hence, by choosing 
  \[
  f(x,y) = (v\cdot\Phi(\tau_x\om,y-x))^2\,\indizeroo \,\indicator_{|v\cdot\Phi(\tau_x \om,y-x)| > \de \sqrt{n}}
  \]
  we find that the the compensator of \begin{equation*}
      \sum_{0 < s \leq T}  (v\cdot M^{n}_{s-,s})^2 \mathbbm{1}_{|v\cdot M^{n}_{s-,s}| > \delta}   
    \end{equation*}
is 
 \begin{equation*}
   \frac{1}{n}
    \int_0^{T n}
    \sum_{y \in \mathbb{Z}^{d}} \om(\{0, y\})\, (v\cdot\Phi(\om, y))^2\,
    \indizeroo\,
  \indicator_{v\cdot|\Phi(\om, y)| > \de \sqrt{n}} \circ \tau_{X_{s-}}\;
    \md s.
 \end{equation*}
 Therefore
  \begin{align*}
&    
    \E_0\!\bigg[\frac{1}{n}
    \int_0^{T n}
    \sum_{y \in \mathbb{Z}^{d}} \om(\{0, y\})\, (v\cdot\Phi(\om, y))^2\,
    \indizeroo\,
  \indicator_{v\cdot|\Phi(\om, y)| > \de \sqrt{n}} \circ \tau_{X_{s-}}\;
    \md s
    \bigg]
    \\
    &\mspace{36mu}=\;
    T\,
    \mean\!\bigg[
    \sum_{y \in \mathbb{Z}^{d}} \om(\{0,y\})\, (v\cdot\Phi(\om, y))^2\,
    \indizeroo\,
    \indicator_{|v\cdot\Phi(\om,y)| > \de \sqrt{n}}
    \bigg]
    \;\underset{n \to \infty}{\longrightarrow}\;
    0,
  \end{align*}
  by stationarity and dominant convergence, which concludes the proof of (\ref{annealed:Lindeberg}). 
   \end{proof}
\subsection{Corrector}

\subsubsection{Corrector p-var vanishes in probability}
We now address the question of sublinearity of the corrector along the trajectory of the process.  We will show that for \emph{any} element of $L^2_{\mathrm{pot}}$ the $p$-variation of the corresponding process under diffusive rescaling vanishes in $L^1(\p_0)$.  
\begin{prop}[Corrector vanishes in $p$-variation]\label{prop:sublinearity:pvar}
  For any $\Psi \in L^2_{\mathrm{pot}}$, consider the process $\Psi(\om, X) \equiv (\Psi(\om, X_t) : t \geq 0)$.  Then, for any $T > 0$ and $p > 2$,
  \begin{align*}
    \lim_{n \to \infty}
    \E_0\!\Big[
    \big\|
    \tfrac{1}{\sqrt n} \Psi(\om, \sqrt n X^n)
    \big\|_{p\text{-}\mathrm{var},\, [0, T]}
    \Big]
    \;=\;
    0.
  \end{align*}
\end{prop}
The proof of Proposition~\ref{prop:sublinearity:pvar} relies on the following estimate which we will refer to as Kipnis-Varadhan inequality.
\begin{lemma}[Kipnis-Varadhan inequality in $p$-variation]\label{lemma:KV:ineq:pvar}
  There exists $c < \infty$ such that for any $T > 0$, $p > 2$ and $\Psi \in L^2_{\mathrm{pot}}$,
  \begin{align}\label{eq:KV:ineq:pvar}
    \E_0\!\bigg[
    \Big\|
    \int_0^{\cdot}
    \big(L^{\om} \Psi(\om, \cdot)\big)(0) \circ \tau_{X_s}\,
    \md s\,
    \Big\|_{p\text{-}\mathrm{var}, [0,T]}
    \bigg]
    \;\leq\;
    2c\, \sqrt{T}\, \Norm{\Psi}{L^2_{\mathrm{cov}}}.
  \end{align}
\end{lemma}
\begin{proof}
  Our proof is similar to the one given in \cite[Lemma~2.4]{KLO12}, cf.\ \cite[Corollary~3.5]{GP18}.  Since $\Psi \in L^2_{\mathrm{pot}}$, for every $\ve>0$ there exists a bounded function $\psi_\ve\!: \Om \to \bbR$ such that $\left\| \mD \psi_\ve - \Psi \right\|_{L^2_{\mathrm{cov}}} < \ve$.  Let us introduce the notation
  \begin{align*}
    R_\ve(\om)
    \;\ldef\;
    \sum_{x \in \mathbb{Z}^{d}} \om(\{0,x\})\,
    \big| \Psi(\om, x) - \mD \psi_\ve(\om, x) \big|,
    \qquad \ve > 0 .
  \end{align*}
  Then, for any $\ve > 0$,
  \begin{align}\label{eq:pvar:approx}
    \E_0\!&\bigg[
    \Big\|
    \int_0^{\cdot}
    \big(L^{\om} (\Psi - \mD \psi_\ve)(\om, \cdot)\big)(0)
    \circ \tau_{X_s}\,
    \md s\,
    \Big\|_{p\text{-}\mathrm{var}, [0,T]}
    \bigg]
    \nonumber\\ 
    &\leq\;
    \E_0\!\bigg[
    \Big\|
    \int_0^{\cdot}
    \big(L^{\om} (\Psi - \mD \psi_\ve)(\om, \cdot)\big)(0)
    \circ \tau_{X_s}\,
    \md s\,
    \Big\|_{1\text{-}\mathrm{var}, [0,T]}
    \bigg]
    \;\leq\;
    T\, \mean\!\big[R_\ve(\om)\big],
  \end{align}
  where we used in the last step both Fubini's theorem and the stationarity of the process as seen from the particle, cf.\ Lemma~\ref{lemma:inv:environment_process}.  Since $\mean[R_\ve(\om)]$ is bounded from above by $\sqrt{\mean[\mu^{\om}(0)]} \Norm{\Psi - \mD \psi_\ve}{L^2_{\mathrm{cov}}}$, the right-hand side of \eqref{eq:pvar:approx} vanishes as $\ve \to 0$.

  On the other hand, since, for any $\ve > 0$, the function $\psi_\ve$ belongs to the domain of the generator $\cL$ and
  \begin{align*}
    (L^{\om} \mD \psi_\ve(\om, \cdot))(x)
    \;=\;
    (\cL \psi_\ve)(\tau_x \om)
    \qquad \forall\, x \in \mathbb{Z}^{d},
  \end{align*}
  we are left with establishing an upper bound of
  \begin{align*}
    \E_0\!\bigg[
    \Big\|
    \int_0^{\cdot} (\cL \psi_\ve)(\tau_{X_s} \om)\, \md s\,
    \Big\|_{p\text{-}\mathrm{var}, [0,T]}
    \bigg].
  \end{align*}
  Let $\{ \cF_t : t \geq 0 \}$ be the natural filtration of $X$, that is $\cF_t = \si(X_s : s \leq t)$.  Then,
  \begin{align*}
    M_t^{\psi_\ve}
    &\ldef\;
    \mD \psi_\ve(\om, X_t) - \mD \psi_\ve(\om, X_0)
    - \int_0^t \big(L^{\om} \mD \psi_\ve(\om, \cdot) \big)(X_s)\, \md s
    \\
    &\;=\;
    \psi_\ve(\tau_{X_t} \om) - \psi_\ve(\tau_{X_0} \om)
    - \int_0^t \big(\cL \psi_\ve\big)(\tau_{X_s} \om)\, \md s
  \end{align*}
  is a $(\prob, \{\cF_t : t \geq 0\})$-martingale.  Similarly, denote by $\{\cF_t^T : t \in [0, T]\}$ the augmented backward filtration generated by $(X_{T-t} : t \in [0,T])$.  Since $\cL$ is a symmetric operator in $L^2(\prob)$, the process $(N_t^{\psi_\ve} : t \in [0, T])$ defined by
  \begin{align*}
    N_t^{\psi_\ve}
    \;\ldef\;
    \psi_\ve(\tau_{X_{T-t}} \om) - \psi_\ve(\tau_{X_T} \om)
    - \int_0^t \big(\cL \psi_\ve\big)(\tau_{X_{T-s}} \om)\, \md s
  \end{align*}
  is a $(\prob, \{\cF_t^{T} : t \in [0, T]\})$-martingale.  Notice that a change of variables yields
  \begin{align*}
    N_{T}^{\psi_\ve} - N_{T-t}^{\psi_\ve}
    \;=\;
    \psi_\ve(\tau_{X_0} \om) - \psi_\ve(\tau_{X_t} \om)
    - \int_0^t \big(\cL \psi_\ve\big)(\tau_{X_{s}} \om)\, \md s.
  \end{align*}
  Thus, the additive functional appearing in both the forward and backward martingale can be expressed as
  \begin{align*}
    - 2 \int_0^t \big(\cL \psi_\ve\big)(\tau_{X_{s}} \om)\, \md s
    \;=\;
    M_t^{\psi_\ve} + N_T^{\psi_\ve} - N_{T-t}^{\psi_\ve}
    \;\rdef\;
    M_t^{\psi_\ve} - \tilde{N}_t^{\psi_\ve}.
  \end{align*}
  Thus, by using Minkowski's inequality and the L\'epingle-BDG-inequality  (\ref{eq:Lepingle:BDG}), there exists $c < \infty$ such that
  \begin{align}\label{eq:pvar:approx:martingale}
    &\E_0\!\bigg[
    \Big\|
    \int_0^{\cdot} \big(\cL \psi_\ve\big)(\tau_{X_s} \om)\, \md s\,
    \Big\|_{p\text{-}\mathrm{var}, [0,T]}
    \bigg]
    \nonumber\\
    &\leq\;
    \frac{1}{2}\,
    \Big(
    \E_0\!\big[
    \|M^{\psi_\ve}\|_{p\text{-}\mathrm{var}, [0, T]}
    \big]
    \,+\,
    \E_0\!\big[
    \|\tilde{N}^{\psi_\ve}\|_{p\text{-}\mathrm{var}, [0, T]}
    \big]
    \Big)
    \nonumber\\
    &\leq\;
    % \frac{c}{2}\,
    % \Big(
    % \E_0\!\Big[ \| M^{\psi_\ve} \|_{\infty\text{-var}, [0, T]} \Big]
    % \,+\,
    % \E_0\!\Big[ \| N^{\psi_\ve} \|_{\infty\text{-var}, [0,T]} \Big]
    % \Big)
    c\,
    \Big(
    \E_0\!\big[ [ M^{\psi_\ve} ]_{T}
    \big]^{\!1/2}
    \,+\,
    \E_0\!\big[ [ N^{\psi_\ve} ]_{T}
    \big]^{\!1/2}
    \Big).
  \end{align}
  
  Note that as $M^{\psi_\ve}$ and $N^{\psi_\ve}$ are Dynkin martingale with respect to the same infinitesimal generator by reversibility of the process seen from the process, by the Dynkin formula their predictable quadratic variation coincide. Moreover it has the form 
  \begin{align}\label{eq:predic-qv:of:dynkin}
    \E_0\!\big[
    \langle M^{\psi_\ve} \rangle_T
    \big]
    \;=\;
    \E_0\big[
    \langle N^{\psi_\ve} \rangle_T
    \big]
    \;=\;
    T\,     \E_0\big[\cL (\psi_\ve^2) -2
    \psi_\ve \cL \psi_\ve 
    \big]
    \;=\;
    T\, \Norm{\mD \psi_\ve}{L^2_{\mathrm{cov}}}^2,
  \end{align}
  where for the last equality we used the  computation in (\ref{eq:ii-formula})  together with
  $\E_0\big[\cL (\psi_\ve^2) \big]=0$, which follows by stationarity of the process seen from the particle.
  As the predictable quadratic variation and the quadratic variation have the same mean we get %
  \begin{align}\label{eq:qv:of:dynkin}
    \E_0\!\big[
    [ M^{\psi_\ve} ]_T
    \big]
    \;=\;
    \E_0\big[
    [ N^{\psi_\ve} ]_T
    \big]
    \;=\;
    T\, \Norm{\mD \psi_\ve}{L^2_{\mathrm{cov}}}^2.
  \end{align}
  Since \begin{align*}
    T\Norm{\mD \psi_\ve}{L^2_{\mathrm{cov}}}^2 \underset{\ve \to 0}{\;\longrightarrow\;} T\, \Norm{\Psi}{L^2_{\mathrm{cov}}}^2,
  \end{align*}
  the assertion follows from \eqref{eq:pvar:approx} and \eqref{eq:pvar:approx:martingale}.
\end{proof}
\begin{remark}
  Along the line of the proof of Lemma~\ref{lemma:KV:ineq:pvar} it easily follows that
  \begin{align*}
    \E_0\!\bigg[
    \Big\|
    \int_0^{\cdot}
    \big(L^{\om} \Psi(\om, \cdot)\big)(0) \circ \tau_{X_s}\,
    \md s\,
    \Big\|_{\infty\text{-var}, [0,T]}
    \bigg]
    \;\leq\;
    \frac{3}{2}\, \sqrt{T}\, \Norm{\Psi}{L^2_{\mathrm{cov}}}.
  \end{align*}
\end{remark}
\begin{lemma}\label{lemma:p-var-estimated-for-L2pot}
  There exists $0<c < \infty$ such that for any $T > 0$, $p > 2$ and $\Psi \in L^2_{\mathrm{pot}}$,
  \begin{align}\label{eq:ub-p-var-ell-two-pot}
    \E_0\!\bigg[
    \big\|
    \Psi(\om, X_\cdot)
    \big\|_{p\text{-}\mathrm{var}, [0, T]}
    \bigg]
    \;\leq\;
    c \sqrt{T}\, \Norm{\Psi}{L^2_{\mathrm{cov}}}.
  \end{align}
\end{lemma}
\begin{proof}
  Let
  \begin{align*}
    M_t^{\Psi}
    \ldef\;
    \Psi (\om, X_t) %- \Psi (\om, X_0)
    -
    \int_0^t
    \big(L^{\om} \Psi (\om, \cdot)\big)(X_s)\,
    \md s
  \end{align*}
  be the $(P_0^{\om}, \{\cF_t : t \geq 0\})$-Dynkin martingale with respect to the natural augmented filtration generated by the process $X$.  Recall that $\Psi (\om, 0) = 0$ for $\prob$-a.e. realization $\om$.  Hence,
  \begin{align*}
    &\E_0\!\bigg[
    \big\|
    \Psi (\om,  X_\cdot)
    \big\|_{p\text{-}\mathrm{var}, [0, T]}
    \bigg]
    \nonumber\\
    &\mspace{36mu}\leq\;
    \E_0\!\bigg[
    \big\|M^{\Psi}\big\|_{p\text{-}\mathrm{var}, [0, T n]}
    \bigg]
    \nonumber\\
    &\mspace{72mu}+\,
    \E_0\!\bigg[
    \Big\|
    \int_0^{\cdot}
    \big(L^{\om} \Psi (\om, \cdot) \big)(0)
    \circ \tau_{X_s}\,
    \md s
    \Big\|_{p\text{-}\mathrm{var}, [0, T ]}
    \bigg]
    \nonumber\\
    &\mspace{36mu}\overset{\!\!\eqref{eq:KV:ineq:pvar}\!\!}{\leq\;}
    \E_0\!\Big[
    \big\| M^{\Psi} \big\|_{p\text{-}\mathrm{var}, [0, T ]}
    \Big]
    \,+\,
    2 c \sqrt{T}\,
    \Norm{\Psi}{L^2_{\mathrm{cov}}}.
  \end{align*}
  Using the L\'epingle-BDG-inqueality, we conclude that
  \begin{align*}
    \E_0\!\bigg[
    \big\|
    \Psi (\om, X_\cdot)
    \big\|_{p\text{-}\mathrm{var}, [0, T]}
    \bigg]
    \;\leq\;
    c^\prime \sqrt{T}\, \Norm{\Psi}{L^2_{\mathrm{cov}}},
  \end{align*}
  as required.
\end{proof}

\begin{proof}[Proof of Proposition~\ref{prop:sublinearity:pvar}]
  Fix some $T > 0$ and $\de > 0$.  Moreover, for a given $\Psi \in L^2_{\mathrm{pot}}$ let $(\psi_\ve : \ve > 0) \subset L^{\infty}(\prob)$ be such that $\left\| \mD \psi_\ve - \Psi \right\|_{L^2_{\mathrm{cov}}} < \ve$.  In the sequel, we consider separately the processes
  \begin{align*}
    \mD \psi_\ve(\om, X)
    &\;\equiv\;
    \big(\mD \psi_\ve(\om, X_t) : t \geq 0\big),
    \\
    (\Psi - \mD \psi_\ve)(\om, X)
    &\;\equiv\;
    \big((\Psi - \mD \psi_\ve)(\om, X_t): t \geq 0\big).
  \end{align*}

  First, by applying \eqref{eq:Levy:system} to the function $f(x,y) = \indicator_{x \ne y}$ we find that
  \begin{align}\label{eq:number:jumps}
    \E_0\!\bigg[
    \sum_{0<s \leq t} \indicator_{X_{s-} \ne X_s}
    \bigg]
    \;=\;
    \E_0\!\bigg[
    \int_{0}^t
    \sum_{y \in \mathbb{Z}^{d}} \om(X_{s-}, y)
    \indicator_{X_{s-} \ne y}\,
    \md s
    \bigg]
    \;=\;
    t \mean\!\big[\mu^{\om}(0)\big],
  \end{align}
  where we used in the last step again Fubini's theorem and and the stationarity of the process as seen from the particle.  Note that
  \begin{align*}
    &\big\| \mD \psi_\ve(\om, X) \big\|_{p\text{-}\mathrm{var}, [0, T n]}
    \;\leq\;
    4\, \Norm{\psi_\ve}{L^{\infty}(\prob)}\,
    \bigg(
    \sum_{0 < s \leq T n} \indicator_{X_{s-} \ne X_s}
    \bigg)^{\!\!1/p}.
  \end{align*}
  Hence, for any $p > 2$ we obtain
  \begin{align*}
    \E_0
    \!\bigg[
    \big\|
    \tfrac{1}{\sqrt{n}} \mD \psi_\ve(\om, \sqrt{n} X^n)
    \big\|_{p\text{-}\mathrm{var}, [0, T]}
    \bigg]
    & \leq \frac{4 \Norm{\psi_\ve}{L^{\infty}(\prob)}}{n^{1/2}} \E_0\!\bigg[
    \bigg(
    \sum_{0 < s \leq T n} \indicator_{X_{s-} \ne X_s}
    \bigg)^{\!\!1/p}
    \bigg]\\
    & \overset{\eqref{eq:number:jumps}}{\;\leq\;}
    \frac{4 \Norm{\psi_\ve}{L^{\infty}(\prob)}}{n^{1/2 - 1/p}}\,
    \Big(T \mean\!\big[\mu^{\om}(0)\big]\Big)^{\!1/p}
    \underset{n \to \infty}{\;\longrightarrow\;}
    0.
  \end{align*}
  On the other hand, applying Lemma~\ref{lemma:p-var-estimated-for-L2pot} to $\Psi - \mD \psi_\ve\in L^2_{\mathrm{pot}}$, we obtain
  \begin{align}
    \E_0\!\bigg[
    \big\|
    \tfrac{1}{\sqrt{n}} (\Psi-\mD\psi_\ve)(\om, \sqrt{n} X^n)
    \big\|_{p\text{-}\mathrm{var}, [0, T]}
    \bigg]
    \;\leq\;
    c^\prime \sqrt{T}\, \Norm{\Psi - \mD \psi_\ve}{L^2_{\mathrm{cov}}}
  \end{align}
  for some $c'>0$.
  Hence,
  \begin{align*}
    & \limsup_{n \to \infty}
    \E_0\!\Big[
    \big\|\tfrac{1}{\sqrt{n}}\Psi(\om, \sqrt{n} X^n)\big\|_{p\text{-}\mathrm{var}, [0, T]}
    \Big]
    \;\le\;
    c^\prime \sqrt{T}\, \Norm{\Psi - \mD \psi_\ve}{L^2_{\mathrm{cov}}}.
  \end{align*}
  Taking $\ve \to 0$ results in
  \begin{align*}
    & \lim_{n \to \infty}
    \E_0\!\Big[
    \big\|\tfrac{1}{\sqrt{n}}\Psi(\om, \sqrt{n} X^n)\big\|_{p\text{-}\mathrm{var}, [0, T]}
    \Big]
    \;=\;
    0,
  \end{align*}
  as required.
\end{proof}

\subsubsection{Tightness of iterated integral of the corrector}\label{subsubseq:tightness_ii_of_corrector}

The main goal of this section is to prove Proposition \ref{prop:correct-estimate-for-pot} below, which provides a key estimate of the integrals of the process in $L^2_\mathrm{pot}$ coordinates after rescaling.

For ease of notation, for $\Xi,\Psi\in L^2_{\mathrm{cov}}$
\begin{align}\label{eq:notation-i-Psi}
  \cI_{s,t}(\Psi,\Xi)\ldef
  I_{s,t}(F,G),
\end{align} 
where $F$ and $G$ are given by $F_t:=\Psi(\omega,X_t)$ and $G_t:=\Xi(\omega,X_t)$. 

\begin{prop}\label{prop:correct-estimate-for-pot}
  Let $\Psi,\Xi\in L^2_\mathrm{pot}$ and $p>2$. Then
  \begin{align*}
    % \label{eq:correct-estimate-for-pot}
    \notag \E_0\bigg[
    \Big\|
    \frac1n \cI(\Psi, \Xi)
    \Big\|_{p/2\text{-}\mathrm{var}, [0,nT]}
    ^{\frac12}
    \bigg]
    \;\leq\;
    C(p)^{\frac12}
    T^{\frac12}
    \Norm{\Psi}{L_\mathrm{cov}^2}^{\frac12}
    \Norm{\Xi}{L_\mathrm{cov}^2}^{\frac12}.
  \end{align*}
\end{prop}

We shall now state two lemmas, prove the last proposition assuming the lemmas, and then turn to proving the lemmas.

The first lemma provides the correct bound for gradients. 

\begin{lemma}\label{lem:correct-estimate-for-gradients}
  Let $\psi,\varphi\in L^2(\prob)$ so that $\mD \psi, \mD \varphi \in L^2_{\mathrm{cov}}$. For all $p>2$
  \begin{align*}%\label{eq:correct-for-grad}
    \notag \E_0\bigg[
    \Big\|
    \cI\, ( \mD\psi,\mD \varphi)
    \Big\|_{p/2\text{-}\mathrm{var}, [0,T]}
    ^{\frac12}
    \bigg]
    \;\leq\;
    C(p)
    T^{\frac12}
    \Norm{D\varphi}{L_\mathrm{cov}^2}^{\frac12}
    \Norm{D\psi}{L_\mathrm{cov}^2}^{\frac12}.
  \end{align*}

\end{lemma}

The second lemma is an a priori bound for general $L^2_\mathrm{cov}$ functions.

\begin{lemma}\label{lem:apriori-estimate-for-pot}
  Let $\Psi,\Xi\in L^2_\mathrm{cov}$ and $p>2$. Then
  \begin{align*}
    \E_0\bigg[ 
    \big\|
    \frac1n \cI\,(\Psi,\Xi)
    \big\|_{p/2\text{-}\mathrm{var}, [0,nT]}^{\frac12}
    \bigg] 
    \; & \leq\;
    \E_0\big[ \mu_0(\omega) \big]\,
    n^{\frac12} T\,
    \Norm{\Psi}{L_\mathrm{cov}^2}^{\frac12}\,
    \Norm{\Xi}{L_\mathrm{cov}^2}^{\frac12}.
  \end{align*}
\end{lemma}

\begin{proof}[Proof of Proposition \ref{prop:correct-estimate-for-pot} assuming Lemmas \ref{lem:correct-estimate-for-gradients} and \ref{lem:apriori-estimate-for-pot}]
  Fix $n$. Let $\epsilon>0$ and approximate $\Psi$ and $\Xi$ by bounded local gradients: let
  $\psi,\xi:\Omega\to \bbR$ be local and bounded so that
  \begin{align*}
    \Norm{\Psi-D\psi}{L_\mathrm{cov}^2},\,
    \Norm{\Xi-D\xi}{L_\mathrm{cov}^2}
    \;\le\;\epsilon.
  \end{align*}
  Since $(\Psi-D\psi), D\psi, (\Xi-D\xi)\in L_\mathrm{pot}^2$, we may use the a priori
  estimate of Lemma~\ref{lem:apriori-estimate-for-pot} to get
  \begin{align*}
    &\E_0\bigg[
    \bigg\|
    \frac1n \cI\,(\Psi-\mD\psi, \Xi)
    \bigg\|_{p/2\text{-}\mathrm{var}, [0,nT]}^{\frac12}
    \bigg]
    \;\leq\;
    \E_0\big[ \mu_0(\omega) \big]\,
    T\, n^{\frac12}\,
    \epsilon^{\frac12}\,
    \Norm{\Xi}{L_\mathrm{cov}^2}^{\frac12}
  \end{align*}
  and
  \begin{align*}
    \E_0\Big[
    \Big\|
    \frac1n 
    \cI\,( \mD\psi,\Xi-D\xi)
    \Big\|_{p/2\text{-}\mathrm{var}, [0,nT]}^{\frac12}
    \Big]
    \;&\leq\;
    \E_0\big[ \mu_0(\omega) \big]\,
    T\, n^{\frac12}\,
    \Norm{D\psi}{L_\mathrm{cov}^2}^{\frac12}\,
    \epsilon^{\frac12}
    \\
    \;& \leq\;
    \E_0\big[ \mu_0(\omega) \big]\,
    T\, n^{\frac12}\,
    \big(\Norm{\Psi}{L_\mathrm{cov}^2} + \epsilon\big)^{\frac12}\,
    \epsilon^{\frac12}.
  \end{align*}
  Hence, 
  \begin{align*}
    \E_0\bigg[ 
    \big\|
    \frac1n \cI\,(\Psi,\Xi)
    \big\|_{p/2\text{-}\mathrm{var}, [0,nT]}^{\frac12}
    \bigg] 
    &
    \;\leq\;
    \E_0[\mu_0(\omega)]\,
    T n^{\frac12}\,
    \epsilon^{\frac12}\,
    \Norm{\Xi}{L_\mathrm{cov}^2}^{\frac12}
    \;+\;
    \E_0[\mu_0(\omega)]\,
    T n^{\frac12}\,
    \big(\Norm{\Psi}{L_\mathrm{cov}^2} + \epsilon\big)^{\frac12}\,
    \epsilon^{\frac12}
    \\
    &\qquad
    +\;
    \E_0\bigg[
    \big\|
    \frac1n \cI\,(\mD\psi,\mD\xi)
    \big\|_{p/2\text{-}\mathrm{var}, [0,nT]}^{\frac12}
    \bigg]
    \\
    &\;\leq\;
    \E_0 [ \mu_0(\omega) ]\,
    T n^{\frac12}\,
    \epsilon^{\frac12}\,
    \Norm{\Xi}{L_\mathrm{cov}^2}^{\frac12}
    \;+\;
    \E_0 [ \mu_0(\omega)]\,
    T n^{\frac12}\,
    \big(\Norm{\Psi}{L_\mathrm{cov}^2} + \epsilon\big)^{\frac12}\,
    \epsilon^{\frac12}
    \\
    &\qquad
    +\;
    C(p)\, T^{\frac12}\,
    \Norm{D\psi}{L_\mathrm{cov}^2}^{\frac12}\,
    \Norm{D\xi}{L_\mathrm{cov}^2}^{\frac12}
    \\
    &\;\leq\;
    \E_0 [ \mu_0(\omega)]\,
    T n^{\frac12}\,
    \epsilon\, \Norm{\Xi}{L_\mathrm{cov}^2}^{\frac12}
    \;+\;
    \E_0 [ \mu_0(\omega) ]\,
    T n^{\frac12}\,
    \big(\Norm{\Psi}{L_\mathrm{cov}^2} + \epsilon\big)^{\frac12}\,
    \epsilon^{\frac12}
    \\
    &\qquad\quad
    +\;
    C(p)\, T^{\frac12}\,
    \big(\Norm{\Psi}{L_\mathrm{cov}^2}+\epsilon\big)^{\frac12}\,
    \big(\Norm{\Xi}{L_\mathrm{cov}^2}+\epsilon\big)^{\frac12},
  \end{align*}
  where the second inequality follows by Lemma~\ref{lem:correct-estimate-for-gradients} with the change of variables $T\to Tn$.
  Taking $\epsilon\to0$, the first two terms vanish and the last term
  converges to the desired upper bound, which completes the proof.
\end{proof}

To prove Lemma~\ref{lem:apriori-estimate-for-pot}, we first state some easy estimates.

\begin{lemma}\label{lem:one-variation-computation}
  Let $\Psi\in L^2_\mathrm{cov}$. Then
  \begin{align*}
    \E_0\Bigg[
    \Big\|
    \Psi(\omega, X_{\cdot})
    \Big\|_{1\text{-}\mathrm{var}, [0,T]}
    \Bigg]
    \;\leq\;
    T\, \E_0\big[ \mu_0(\omega) \big]^{\frac12}\,
    \Norm{\Psi}{L_\mathrm{cov}^2}.
  \end{align*}
\end{lemma}

\begin{proof}
  Let $f^\omega\!: \mathbb{Z}^{d} \times \mathbb{Z}^{d} \to \bbR$ be defined by $f^\omega(x,y)\ldef |\Psi(\omega, y)-\Psi(\omega, x)|$. Then  
  % We first define the function $\eta : \Omega \times \Omega \to \R^d$ by
  % \begin{align}\label{eq:fcn_jump_from_envi}
  %   \eta (\omega,\omega')
  %   \;=\;
  %   \sum_{y\in\mathbb{Z}^{d}} y\, \indicator_{\{\tau_{y}\omega %= \omega'\}}.
  % \end{align}
  we have
  \begin{align*}
    \E_0\Bigg[
    \Big\|
    \Psi(\omega, X_{\cdot})
    \Big\|_{1\text{-}\mathrm{var}, [0,T]}
    \Bigg]
    &=\;
    \E_0\Bigg[
    \Big\|
    \Big(
    \Psi(\omega, X_t)-\Psi(\omega, X_s)
    \Big)_{(s, t)\in\Delta_{[0,T]}}
    \Big\|_{1\text{-}\mathrm{var}, [0,T]}
    \Bigg]
    \\
    &=\;
    \E_0\Bigg[
    \sum_{0< u \le T,\; X_u\ne X_{u-}}
    \big| \Psi(\omega, X_u)-\Psi(\omega, X_{u-}) \big|
    \Bigg]
    \\
    &=\;
    \E_0\Bigg[
    \sum_{0< u \le T}
    f^\omega(X_{u-},X_{u})
    \Bigg]
    \\
    &=\;
    \E\Bigg[
    E_{0}^{\omega}
    \bigg[
    \int_{0}^{T}
    \sum_{x \in \mathbb{Z}^{d}} \om(\{X_{u-}, x\})\, f^\omega(X_{u-}, x)\,
    \md u
    \bigg]
    \Bigg]
    \\
    &=\;
    \E\Bigg[
    E_{0}^{\omega}
    \bigg[
    \int_{0}^{T}
    \sum_{x \in \mathbb{Z}^{d}} \om(\{X_{u-}, X_{u-}+x\})\, f^\omega(X_{u-}, X_{u-}+x)\,
    \md u
    \bigg]
    \Bigg]
    \\
    &=\;
    \int_{0}^{T}
    \E_0\Bigg[
    \sum_{x\in\mathbb{Z}^{d}}
    \tau_{X_{u-}}\omega(\{0,x\})\,
    \big|
    \Psi\big(\tau_{X_{u-}}\omega,x)\big)
    \big|
    \Bigg] \md u
    \\
    &=\;
    \int_{0}^{T}
    \E_0\Bigg[
    \sum_{x\in\mathbb{Z}^{d}} \omega(\{0,x\})\, |\Psi(\omega, x)|
    \Bigg] \md u
    \\
    &\leq\;
    T\, \E_0\big[ \mu_0(\omega) \big]^{\frac12}\,
    \Norm{\Psi}{L_\mathrm{cov}^2},
  \end{align*}
  where the forth equality is an application of the L\'evy system
  theorem (\ref{eq:Levy:system}), the fifth is a chage of variables in the summation, the sixth is the the cocycle property together with Fubini's theorem. Finally the last equality is by stationarity of the process from the point of view of the walker $\tau_{X_{t}}\omega$
  and the inequality is an application of the Cauchy-Schwarz inequality.
\end{proof}

Observe the following simple estimate. Let
$F,G:[0,T]\to\bbR$ be c\`adl\`ag functions with
$F_{u-,u}= 0= G_{u-,u}$ for all but finitely many $u\in[0,T]$. Let
\begin{align*}
  I_{s,t}(F,G)\coloneqq \sum_{ s<u\le t} F_{s,u-}\, G_{u-,u}.
\end{align*}
Then
\begin{align}\label{eq:ii-one-variation-simple-estimate}
  \|I(F,G)\|_{1\text{-}\mathrm{var}, [0,T]}
  \;\leq\;
  \|F\|_{1\text{-}\mathrm{var}, [0,T]}\,
  \|G\|_{1\text{-}\mathrm{var}, [0,T]}.
\end{align}
Indeed, for any partition $\pi$ of $[0,T]$ 
\begin{align*}
  \sum_{[s,t]\in \pi}\big| \sum_{ s<u\le t} F_{s,u-}\, G_{u-,u}\big| 
  \le 
  \sum_{[s,t]\in \pi} \sum_{ s<u\le t} \|F\|_{1\text{-}\mathrm{var}, [0,T]}\, |G_{u-,u}| 
  =
  \|F\|_{1\text{-}\mathrm{var}, [0,T]}\, 
  \|G\|_{1\text{-}\mathrm{var}, [0,T]}\, 
\end{align*}

We can now supply the

\begin{proof}[Proof of Lemma~\ref{lem:apriori-estimate-for-pot}]
  Let $\Psi,\Xi\in L^2_\mathrm{cov}$ and $p>2$. Then
  \begin{align*}
    &\E_0\Bigg[
    \Big\|
    \Big(
    \frac1n \int_s^{t}
    \big( \Psi(\omega, X_{u-})-\Psi(\omega, X_s) \big)\,
    \md \Xi(\omega, X_u)\,
    \Big)_{(s, t)\in\Delta_{[0,nT]}}
    \Big\|_{p/2\text{-}\mathrm{var}, [0,nT]}^{\frac12}
    \Bigg]
    \\
    &\qquad\leq\;
    n^{-\frac12}\,
    \E_0\Bigg[
    \Big\|
    \Big(
    \int_s^{t}
    \big( \Psi(\omega, X_{u-})-\Psi(\omega, X_s) \big)\,
    \md \Xi(\omega, X_u)\,
    \Big)_{(s, t)\in\Delta_{[0,nT]}}
    \Big\|_{1\text{-}\mathrm{var}, [0,nT]}^{\frac12}
    \Bigg]
    \\
    &\qquad\leq\;
    n^{-\frac12}\,
    \E_0\Bigg[
    \big\|
    \Psi(\omega, X_{\cdot})
    \big\|_{1\text{-}\mathrm{var}, [0,nT]}^{\frac12}
    \big\|
    \Xi(\omega, X_{\cdot})
    \big\|_{1\text{-}\mathrm{var}, [0,nT]}^{\frac12}
    \Bigg]
    \\
    &\qquad\leq\;
    n^{-\frac12}\,
    \E_0\Bigg[
    \big\|
    \Psi(\omega, X_{\cdot})
    \big\|_{1\text{-}\mathrm{var}, [0,nT]}
    \Bigg]^{\frac12}
    \E_0\Bigg[
    \big\|
    \Xi(\omega, X_{\cdot})
    \big\|_{1\text{-}\mathrm{var}, [0,nT]}
    \Bigg]^{\frac12}
    \\
    &\qquad\leq\;
    n^{-\frac12}\,
    \E_0\big[ \mu_0(\omega) \big]\,
    n T\,
    \Norm{\Psi}{L_\mathrm{cov}^2}^{\frac12}\,
    \Norm{\Xi}{L_\mathrm{cov}^2}^{\frac12}.
  \end{align*}
\end{proof}

Next, we turn to proving Lemma~\ref{lem:correct-estimate-for-gradients}.

\begin{proof}[Proof of Lemma~\ref{lem:correct-estimate-for-gradients}]
  Let $\psi,\varphi\in L^2(\prob)$ so that $\|\mD \psi\|_{L^2_{\mathrm{cov}}}, \|\mD \varphi \|_{L^2_{\mathrm{cov}}} < \infty$ and fix $p>2$.
  We use the forward-backward presentation.
  Let $\{ \cF_t : t \geq 0 \}$ be the natural filtration of $X$, that is $\cF_t = \si(X_s : s \leq t)$.  Then,
  \begin{align*}
    M_t^{\varphi}
    &\ldef\;
    \mD \varphi(\om, X_t) - \mD \varphi(\om, X_0)
    - \int_0^t \big(L^{\om} \mD \varphi(\om, \cdot) \big)(X_s)\, \md s
    \\
    &\;=\;
    \varphi(\tau_{X_t} \om) - \varphi(\tau_{X_0} \om)
    - \int_0^t \big(\cL \varphi\big)(\tau_{X_s} \om)\, \md s
  \end{align*}
  is a $(\prob, \{\cF_t : t \geq 0\})$-martingale.  Similarly, denote by $\{\cF_t^T : t \in [0, T]\}$ the augmented backward filtration generated by $(X_{T-t} : t \in [0,T])$.  Define the process $(N_t^{\varphi} : t \in [0, T])$ by 
  \begin{align*}
    N_t^{\varphi} \;\ldef\; \bar N_{t^+}^{\varphi}
    \;=\;
    \lim_{\epsilon\to 0} \bar N_{t+\epsilon}^{\varphi},
  \end{align*}
  where 
  \begin{align*}
    \bar N_t^{\varphi}
    \;\ldef\;
    \varphi(\tau_{X_{T-t}} \om) - \varphi(\tau_{X_T} \om)
    - \int_0^t \big(\cL \varphi\big)(\tau_{X_{T-s}} \om)\, \md s.
  \end{align*}
  Since $\mathcal{L}$ is symmetric in $L^2(\prob)$, the process $(\tau_{X_t}\omega)$ is
  reversible under the invariant measure. Consequently, the time-reversed process has the same law as the forward process, ensuring that $N^\varphi$ is a well-defined, square-integrable martingale with respect to the backward filtration.   
  Notice that a change of variables yields
  \begin{align*}
    N_{T}^{\varphi} - N_{T-t}^{\varphi}
    \;=\;
    \varphi(\tau_{X_0} \om) - \varphi(\tau_{X_t} \om)
    - \int_0^t \big(\cL \varphi\big)(\tau_{X_{s}} \om)\, \md s.
  \end{align*}
  Therefore
  \begin{align*}
    \varphi(\tau_{X_t} \om) - \varphi(\tau_{X_0} \om)
    \;=\;
    \frac12M_t^{\varphi} + \frac12 N_{T-t}^{\varphi} - \frac12 N_{T}^{\varphi}
  \end{align*}
  and in particular for $0< s<t\le T$
  \begin{align*}
    \varphi(\tau_{X_t} \om) - \varphi(\tau_{X_s} \om)
    \;=\;
    \frac12M_{s,t}^{\varphi} - \frac12 N_{T-t,T-s}^{\varphi}.
  \end{align*}
  For ease of notation let us write
  \begin{align*}
    \om_t\coloneqq\tau_{X_t} \om
    \;\text{ and }\;
    \om^*_t\coloneqq\tau_{X_{(T-t)^+}} \om = \tau_{X_{{T-t^-}}} \om.
  \end{align*}
  More explicitly,
  \begin{align*}
    \om^*_t(x,y) \coloneqq \lim_{\epsilon\downarrow0} \tau_{X_{{T-(t-\epsilon)}}} \om \,(x,y)
  \end{align*}
  for every $x,y\in\mathbb{Z}^d$.

  \begin{align*}
    2\int_s^{t}
    &
    ( D\psi(\omega, X_{u-})-D\psi(\omega, X_s) )\,
    \md D\varphi(\omega, {X_u})\, \\
    \; & =\;
    2\int_s^{t}
    ( \psi(\om_{u-})-\psi(\om_{s}) )\,
    \md \varphi(\om_u)\,
    \\
    \;& =\;
    2\sum_{s<u\le t: \om_{u-} \ne \om_{u}}	( \psi(\om_{u-})-\psi(\om_{s}) )\,
    (\varphi(\om_u)-\varphi(\om_{u-}))
    \\
    \;& =\;
    \int_s^{t}( \psi(\om_{u-})-\psi(\om_{s}) )\, \md M^\varphi_u
    \;+\;
    \int_s^{t}( \psi(\om_{u-})-\psi(\om_{s}) )\, \md N^\varphi_{T-u}
    \\
    \;& =\;
    \int_s^{t}( \psi(\om_{u-})-\psi(\om_{s}) )\, \md M^\varphi_u
    \;+\;
    \int_s^{t}( \psi(\om_{u-})-\psi(\om_t) )\,\md N^\varphi_{T-u}
    \\
    & \;+\; 
    ( \psi(\om_{t}) - \psi(\om_{s}) ) N^\varphi_{T-s,T-t}
    \\
    \;& =\;
    \int_s^{t}( \psi(\om_{u-})-\psi(\om_{s}) )\, \md M^\varphi_u
    \;-\;
    \int_{T-t}^{T-s}( \psi(\om^*_{v})-\psi(\om^*_{T-t}) )\,\md N^\varphi_{v} \\
    & \;+\;
    (\psi(\om^*_{T-s}) -\psi(\om^*_{T-t})) N^\varphi_{T-t,T-s}
    \\
    \;& =\;
    \int_s^{t}( \psi(\om_{u-})-\psi(\om_{s}) )\, \md M^\varphi_u
    \; - \;
    \int_{T-t}^{T-s}( \psi(\om^*_{v-})-\psi(\om^*_{T-t}) )\,\md N^\varphi_{v}
    \\
    & \;-\;
    \int_{T-t}^{T-s}( \psi(\om^*_{v})-\psi(\om^*_{v-}) )\,\md N^\varphi_{v}
    \;+\;
    (\psi(\om^*_{T-s}) -\psi(\om^*_{T-t})) N^\varphi_{T-t,T-s}
  \end{align*}
  Next, define $p<p_1=p_1(p)$ canonically so that $\frac{2}{p}<\frac12 + \frac{1}{p_1}$, e.g.\ by taking it to be the midpoint of $[p,\frac{2p}{4-p}]$. Using \eqref{eq:Friz-Zorin-Kranich} we obtain
  \begin{align*}
    \E_0\bigg[
    \Big\|
    \Big(
    \int_s^{t}( 
    &
    \psi(\om_{u-})-\psi(\om_{s}) )\, \md M^\varphi_u
    \Big)_{(s, t)\in\Delta_{[0,T]}}
    \Big\|_{p/2\text{-}\mathrm{var}, [0,T]}
    ^{\frac12}
    \bigg]
    \\
    \;& \leq\;
    C_0(p,p_1)
    \E_0\big[
    \big\|
    \psi(\om_{\cdot})
    \big\|_{p_1\text{-}\mathrm{var}, [0,T]}
    \big] ^\frac12
    \E_0\big[
    \big<
    M^\varphi
    \big>_T
    ^\frac12
    \big] ^\frac12
    \\
    \;& \leq\;
    C_0(p,p_1) c'(p_1)
    (\sqrt{T} \Norm{D\psi}{L_\mathrm{cov}^2})  ^\frac12
    (\sqrt{T}\Norm{D\varphi}{L_\mathrm{cov}}) ^\frac12 	,
    \\
    \;& =:\;
    C(p)
    \sqrt{T} \Norm{D\psi}{L_\mathrm{cov}^2}^\frac12   \Norm{D\varphi}{L_\mathrm{cov}^2}^\frac12 ,
  \end{align*}
  where we used Lemma~\ref{lemma:p-var-estimated-for-L2pot} for $D\psi(\om_{\cdot})\in L_\mathrm{pot}^2$ and \eqref{eq:qv:of:dynkin} for the predictable quadratic variation martingale estimate.
  Similarly,
  \begin{align*}
    \E_0\bigg[
    \Big\|
    \Big(
    \int_{T-t}^{T-s}( \psi(\om^*_{v-})-\psi(\om^*_{T-t}) )\,\md N^\varphi_{v}
    \Big)_{(s, t)\in\Delta_{[0,T]}}
    \Big\|_{p/2\text{-}\mathrm{var}, [0,T]}
    ^{\frac12}
    \bigg]
    \\
    \;\leq\;
    C(p)
    \sqrt{T} \Norm{D\psi}{L_\mathrm{cov}^2}^\frac12   \Norm{D\varphi}{L_\mathrm{cov}^2}^\frac12.
  \end{align*}
  Next, using Cauchy-Schwarz inequality twice we obtain
  \begin{align*}
    & \E_0\bigg[
    \Big\|
    \Big(
    (\psi(\om^*_{T-s}) -\psi(\om^*_{T-t}) N^\varphi_{T-t,T-s}
    \Big)_{(s, t)\in\Delta_{[0,T]}}
    \Big\|_{p/2\text{-}\mathrm{var}, [0,T]}
    ^{\frac12}
    \bigg]
    \\
    \;& \leq\;
    \E_0\bigg[
    \Big\|
    \Big(
    (\psi(\om^*_{T-s}) -\psi(\om^*_{T-t})
    \Big)_{(s, t)\in\Delta_{[0,T]}}
    \Big\|_{p\text{-}\mathrm{var}, [0,T]}
    ^{\frac12}
    \Big\|
    \Big(
    N^\varphi_{T-t,T-s}
    \Big)_{(s, t)\in\Delta_{[0,T]}}
    \Big\|_{p\text{-}\mathrm{var}, [0,T]}
    ^{\frac12}
    \bigg]
    \\
    \;& \leq\;
    \E_0\bigg[
    \Big\|
    \Big(
    (\psi(\om^*_{T-s}) -\psi(\om^*_{T-t})
    \Big)_{(s, t)\in\Delta_{[0,T]}}
    \Big\|_{p\text{-}\mathrm{var}, [0,T]}
    \bigg]
    ^{\frac12}
    \E_0\bigg[
    \Big\|
    \Big(
    N^\varphi_{T-t,T-s}
    \Big)_{(s, t)\in\Delta_{[0,T]}}
    \Big\|_{p\text{-}\mathrm{var}, [0,T]}
    \bigg]
    ^{\frac12}
    \\
    \;& \leq\;
    \bar C(p)
    \sqrt{T} \Norm{D\psi}{L_\mathrm{cov}^2}^\frac12   \Norm{D\varphi}{L_\mathrm{cov}^2}^\frac12,
  \end{align*}
  where we used Lemma~\ref{lemma:p-var-estimated-for-L2pot} for $D\psi\in L_\mathrm{pot}^2$ and Lepingle BDG inequality for the martingale term together with \eqref{eq:qv:of:dynkin} for the quadratic variation. Here $\bar C(p)$ is a product of the constant $c$ from Lemma~\ref{lemma:p-var-estimated-for-L2pot} and the constant appears in the upper bound of the Lepingle BDG inequality.
  Finally, 
  % since the absolutely continuous part of  $G$ do not contribute to $Q(F,G)$  
  \begin{align*}
    \E_0\bigg[
    \Big\|
    &
    \Big(
    \int_{T-t}^{T-s}( \psi(\om^*_{v})-\psi(\om^*_{v-}) )\,\md N^\varphi_{v}
    \Big)_{(s, t)\in\Delta_{[0,T]}}
    \Big\|_{p/2\text{-}\mathrm{var}, [0,T]}
    ^{\frac12}
    \bigg]
    \\
    \;& \leq\;
    \E_0\bigg[
    \Big\|
    \Big(
    \int_{T-t}^{T-s}( \psi(\om^*_{v})-\psi(\om^*_{v-}) )\,\md N^\varphi_{v}
    \Big)_{(s, t)\in\Delta_{[0,T]}}
    \Big\|_{1\text{-}\mathrm{var}, [0,T]}
    ^{\frac12}
    \bigg]
    \\
    \;& =\;
    \E_0\bigg[
    \sum_{0<v\le T} |\psi(\om^*_{v})-\psi(\om^*_{v-})|\cdot|\varphi(\om^*_{v})-\varphi(\om^*_{v-})|\indicator_{\om^*_{v-}\ne \om^*_{v}}
    \bigg]^{\frac12}
    \\
    \;& \le\;
    \E_0\bigg[
    \int_{0}^{T} \sum_x \om^*_{v}(0,x) |(\psi(\tau_x\om^*_{v})-\psi(\om^*_{v}))|\cdot|(\varphi(\tau_x\om^*_{v})-\varphi(\om^*_{v}))|\md v
    \bigg]^{\frac12}
    \\    
    \;& \le\;
    \sqrt T\Norm{D\psi}{L_\mathrm{cov}^2}^{\frac12} \Norm{D\varphi}{L_\mathrm{cov}^2}^{\frac12}
    % 
    % 
    % \;& \leq\;
    % \E_0\bigg[
    %	\sum_{0<v\le T} (\psi(\om^*_{v})-\psi(\om^*_{v-}))^2\indicator_{\om^*_{v-}\ne \om^*_{v}}
    %	\bigg]^{\frac14}
    % \E_0\bigg[
    %	[ N^\varphi]_T
    %	\bigg]
    % ^{\frac14}
    %	\\
    %	\;& \leq\;
    % \E_0\bigg[
    %	\int_{0}^{T} \sum_x \om^*_{v}(0,x) (\psi(\tau_x\om^*_{v})-\psi(\om^*_{v}))^2\md v
    %	\bigg]^{\frac14}
    %	(T\Norm{D\varphi}{L_\mathrm{cov}^2}^2)^{\frac14}
    % \\
    %	\;& =\;
    % \bigg(
    %	\int_{0}^{T} \E_0\bigg[\sum_x \om^*_{v}(0,x) (\psi(\tau_x\om^*_{v})-\psi(\om^*_{v}))^2\bigg]\md v
    %	\bigg)^{\frac14}
    %	(T\Norm{D\varphi}{L_\mathrm{cov}^2}^2)^{\frac14}
    % \\
    % \;& =\;
    % \sqrt{T} \Norm{D\psi}{L_\mathrm{cov}^2}^\frac12 \Norm{D\varphi {L_\mathrm{cov}^2}^\frac12,
  \end{align*}
  where for the equality we used the fact an absolutely continuous part does not influence the quadratic co-variation of bounded variation functions, the second inequality is an application L\'evy system theorem \eqref{eq:Levy:system} applied to the reversed process together with stationarity of $(\omega^*_t)$, the reversed process from the point of view of the walker. For the last inequality we used again stationarity and the Cauchy-Schwarz inequality. This concludes the proof.
\end{proof}

\subsubsection{Convergence in uniform norm for iterated integrals in probability}\label{subsubseq:Skoro_conver_ii_of_corrector}

% The argument here is as follows:
% \begin{enumerate}
% \item First show convergence in probability in Skorohod for the iterated integral of gradients.
%   
% \item Fix $\epsilon>0$. Approximate corrector $\chi$ by BOUNDED gradients $D\varphi$.
%   
% \item Write integral with 3 terms: only gradients, gradient against residue and residue against corrector.
%   
% \item Use key estimate for integrals involving the residues $\chi^i-D\varphi^i$.
%   
% \item Deduce convergence in probability in uniform norm from the one for the ii for gradients.
%   
% \end{enumerate}
% 

\begin{prop}\label{prop:ii-corrector-convergence-in-prob}
  Let $\Psi,\Xi\in L^2_\mathrm{pot}$. Then
  \begin{align}\label{eq:ii-corrector-convergence-in-prob}
    \E_0\Big[
    \Big\|
    \Big(
    \frac1n \int_0^{nt}
    \Psi(\omega, X_{s-})\,
    \md \Xi(\omega, {X_{s}})
    + \frac12 \big< \Psi, \Xi \big>_{L_\mathrm{cov}^2} t \,
    \Big)_{t\in {[0,T]}}
    \Big\|_{\mathrm{unif}, [0,T]}
    ^\frac12
    \Big]
    \;\xrightarrow[n\to\infty]{} 0.
  \end{align}
\end{prop}

\begin{lemma}\label{lem:bdd-gradients-uniformly-integrable-ii}
  Let $\psi,\varphi\in L^\infty(\prob)$ (and in particular $\|\mD \psi\|_{L^2_{\mathrm{cov}}}, \|\mD \varphi \|_{L^2_{\mathrm{cov}}} < \infty$ since $\E[\mu^\om(0)]<\infty$).	Then
  \begin{align}\label{eq:bdd-grad-unif-integ-ii}
    \E_0\Big[
    \Big\|
    \frac{1}{n}\int_{0}^{nt}\mD \psi(\omega,X_{s-}) \, \md \mD \varphi(\omega,X_{s})
    + \frac12 \big< \mD \psi, \mD \varphi \big>_{L_\mathrm{cov}^2} t
    \Big\|_{\mathrm{unif}, [0,T]}
    % ^\frac12
    \Big]
    \;\xrightarrow[n\to\infty]{} 0.
  \end{align}
\end{lemma}

\begin{proof}[Proof of the Proposition \ref{prop:ii-corrector-convergence-in-prob} assuming Lemma~\ref{lem:bdd-gradients-uniformly-integrable-ii}]
  Let $\epsilon>0$. It is enough to show there are $\psi,\varphi\in L^\infty(\prob)$ so that
  \begin{align*}
    \E_0\Big[
    \Big\|
    \Big(
    \frac1n
    \int_0^{nt}
    \Psi(\omega, X_{s-})\,
    \md \Xi(\omega, {X_{s}})
    -
    \frac1n
    \int_0^{nt}
    D \psi(\omega, X_{s-})\,
    \md D \varphi(\omega, {X_{s}}) \,\Big)_{t\in {[0,T]}}
    \Big\|_{\mathrm{unif}, [0,T]}
    ^\frac12
    \Big]
    \; \le \;
    \epsilon/3
  \end{align*}
  for all $n\in\N$, and so that additionally
  \begin{align*}
    (\frac12 \big|\big(\big< \Psi, \Xi \big>_{L_\mathrm{cov}^2} - \big< D\psi, D\varphi \big>_{L_\mathrm{cov}^2}\big)\big| T )^\frac12
    \; \le\;
    \epsilon/3.
  \end{align*}
  Indeed, take $\psi,\varphi\in L^\infty(\prob)$ with the last estimates, then, using \eqref{eq:bdd-grad-unif-integ-ii} for all $n$ large enough
  \begin{align*}
    \E_0\Big[
    \Big\|
    \Big(
    \frac1n \int_0^{nt}
    \Psi(\omega, X_{s-})\,
    \md \Xi(\omega, {X_{s}}) + \frac12 \big< \Psi, \Xi \big>_{L_\mathrm{cov}^2} t \,
    \Big)_{t\in {[0,T]}}
    \Big\|_{\mathrm{unif}, [0,T]}
    ^\frac12
    \Big]
    \;\le\;
    \epsilon/3 \; + \; \epsilon/3
    \\
    \; + \;
    \E_0\Big[
    \Big\|
    \Big(
    \frac1n \int_0^{nt}
    D\psi(\omega, X_{s-})\,
    \md D\varphi(\omega, {X_{s}}) + \frac12 \big< D\psi, D\varphi \big>_{L_\mathrm{cov}^2} t \,
    \Big)_{t\in {[0,T]}}
    \Big\|_{\mathrm{unif}, [0,T]}
    ^\frac12
    \Big]
    \; < \;
    \epsilon,
  \end{align*}
  as required.
  Next, assume without loss of generality that $\epsilon<1$. Since $\Psi,\Xi\in L^2_\mathrm{pot}$ it is possible to find $\psi,\varphi\in L^\infty(\prob)$ so that $\big\| \Psi-\mD\psi \big\|_{L_\mathrm{cov}^2}<\delta$ and $\big\| \Xi-\mD\varphi \big\|_{L_\mathrm{cov}^2}<\delta$, where $0<\delta < \frac{2(\epsilon/3)^2}{3T C(3/2)(\big\| \Psi \big\|_{L_\mathrm{cov}^2}+\big\| \Xi \big\|_{L_\mathrm{cov}^2})}$, where $C(p)$ is the constant guaranteed from Proposition \ref{prop:correct-estimate-for-pot}. It is easy to verify that the second condition holds with such a choice of $\delta$. For the second condition we note that
  \begin{align*}
    &
    \frac1n
    \int_0^{nt}
    \Psi(\omega, X_{s-})\,
    \md \Xi(\omega, {X_{s}})
    -
    \frac1n
    \int_0^{nt}
    D \psi(\omega, X_{s-})\,
    \md D \varphi(\omega, {X_{s}}) \,
    \\
    \;& = \;
    \frac1n
    \int_0^{nt}
    (\Psi - \mD \psi)(\omega, X_{s-}) 
    \md \Xi(\omega, {X_{s}})
    +
    \frac1n
    \int_0^{nt}
    \mD\psi(\omega, X_{s-})\,
    \md  (\Xi-\mD\varphi)(\omega, {X_{s}}).
  \end{align*}
  Using Proposition \ref{prop:correct-estimate-for-pot} together with the monotonicity of the norms, $\|\cdot\|_{\infty\text{-}\mathrm{var}, [0,T]}\le \|\cdot\|_{q\text{-}\mathrm{var}, [0,T]}$ for any $q>0$, we get also the first condition. This concludes the proof.
\end{proof}

\begin{lemma}\label{lem:l2-uniformly-integrable-ii}
  Let $\psi,\varphi:\Om\to\R$ be $\prob$-measurable functions so that $\|\mD \varphi \|_{L^2_{\mathrm{cov}}} < \infty$ and $\E[\psi^2(\om) \mu^\om(0)]<\infty$.
  Then
  \begin{align*}
    \Big\|
    \frac{1}{n}\int_{0}^{nt}\psi(\omega_{s-}) \, \md \varphi(\omega_{s})
    + \frac12 \big< \mD \psi, \mD \varphi \big>_{L_\mathrm{cov}^2} t
    \Big\|_{\mathrm{unif}, [0,T]}
        \;\xrightarrow[n\to\infty]{} 0
  \end{align*}
  both in $L^1(\prob_0)$ and $\prob_0$-a.s.
\end{lemma}

\begin{proof}[Proof of Lemma~\ref{lem:bdd-gradients-uniformly-integrable-ii} assuming Lemma~\ref{lem:l2-uniformly-integrable-ii}]
  Let $\psi,\varphi\in L^\infty(\prob)$. Note that $\prob$-a.s we have
  \begin{align*}
    %	\E_0\Big[
    \Big\|
    \frac{1}{n} \psi(\omega)(\varphi(\omega_{nt})-\varphi(\omega))
    \Big\|_{\mathrm{unif}, [0,T]}
    % \Big]
    \le
    \frac{2\| \psi \|_{L^\infty(\prob)}\| \varphi \|_{L^\infty(\prob)}}{n}
    \;\xrightarrow[n\to\infty]{} 0.
  \end{align*}
  Hence it remains to show
  \begin{align*}
    \E_0\Big[
    \Big\|
    \frac{1}{n}\int_{0}^{nt}\psi(\omega_{s-}) \, \md \varphi(\omega_{s})
    + \frac12 \big< \mD \psi, \mD \varphi \big>_{L_\mathrm{cov}^2} t
    \Big\|_{\mathrm{unif}, [0,T]}
    % ^\frac12
    \Big]
    \;\xrightarrow[n\to\infty]{} 0,
  \end{align*}
  which is an application of Lemma~\ref{lem:l2-uniformly-integrable-ii} after noting that
  \begin{align*}
    \mean
    \!\big[
    \mu^\om(0)  \psi (\om)^2
    \big]
    & \; \le \;
    \mean\! [\mu^\om(0)]
    \| \psi^2\|_{L^\infty}
    <\infty.
  \end{align*}

\end{proof}

\begin{proof}[Proof of Lemma \ref{lem:l2-uniformly-integrable-ii}]
  We apply Lemma~\ref{lem:ergodic in uniform norm} to the function
  \begin{align*}
    f(\omega,\tilde\omega)
    \; := \;
    \psi(\omega) \, (\varphi(\tilde\omega)-\varphi(\omega))
  \end{align*}
  To do so we need to show that the condition \eqref{eq:assumption-integrability} holds for the function $f$ and then to compute
  $\mean
  \!\big[
  \sum_{x \in \mathbb{Z}^{d}}
  \om(\{0,x\}) f(\om, \tau_x \om)
  \big]
  =
  \mean
  \!\big[
  \sum_{x \in \mathbb{Z}^{d}}
  \om(\{0,x\}) \psi (\om) (\varphi(\tau_x \om)-\varphi(\omega))
  \big]
  $.
  For the former, by Cauchy-Schwartz inequality we obtain
  \begin{align*}
    \mean
    \!\big[
    \sum_{x \in \mathbb{Z}^{d}}
    \om(\{0,x\}) |f(\om, \tau_x \om)|
    \big]
    & \; = \;
    \mean
    \!\big[
    \sum_{x \in \mathbb{Z}^{d}}
    \om(\{0,x\}) |\psi (\om) (\varphi(\tau_x \om)-\varphi(\omega))|
    \big] \\
    & \; \le \;
    \mean
    \!\big[
    \mu^\om(0)  \psi (\om)^2
    \big]^{1/2}
    \|\mD \varphi\|_{L_\mathrm{cov}^2} < \infty.
  \end{align*}
  It is left to show that
  \begin{align}\label{eq:ii-formula}
    \mean
    \!\big[
    \sum_{x \in \mathbb{Z}^{d}}
    \om(\{0,x\}) \psi (\om) (\varphi(\tau_x \om)-\varphi(\omega))
    \big]
    \; = \;
    -\frac12 \big< \mD \psi, \mD \varphi \big>_{L_\mathrm{cov}^2}.
  \end{align}
  Indeed,
  \begin{align*}
    \mean
    \!\big[
    \sum_{x \in \mathbb{Z}^{d}}
    \om(\{0,x\}) \psi (\om) (\varphi(\tau_x \om)-\varphi(\omega))
    \big]
    \; = \; &
    \mean
    \!\big[
    \sum_{x \in \mathbb{Z}^{d}}
    \om(\{0,-x\}) \psi (\om) (\varphi(\tau_{-x} \om)-\varphi(\omega))
    \big]\\
    \; = \; &
    \mean
    \!\big[
    \sum_{x \in \mathbb{Z}^{d}}
    \tau_x\om(\{-x,0\}) \psi (\tau_x\om) (\varphi(\tau_x\tau_{-x} \om)-\varphi(\tau_x\omega))
    \big]\\
    \; = \; &
    \mean
    \!\big[
    \sum_{x \in \mathbb{Z}^{d}}
    \om(\{0,x\}) \psi (\tau_x\om) (\varphi(\om)-\varphi(\tau_x\omega))
    \big],
  \end{align*}
  where for the second equality we first used the symmetry of the conductances and then Fubini's theorem together with translation invariance for each summand.
  % The latter is valid since
  % \begin{align*}
  %   \mean
  %   \!\big[
  %   \sum_{x \in \mathbb{Z}^{d}}
  %   \om(\{0,x\}) |\psi (\om) (\varphi(\tau_x \om)-\varphi(\omega))|
  %   \big] ^2
  %   \; \le \; &
  %   \mean
  %   \!\big[
  %   \mu^\om(0) \psi (\om)^2 	\big]
  %   \|\mD \varphi\|_{L_\mathrm{cov}^2}^2\\
  %   \; \le \; &
  %   \mean
  %   \!\big[
  %   \mu^\om(0)\big] \| \psi \|_{L^\infty(\prob)}^2
  %   \|\mD \varphi\|_{L_\mathrm{cov}^2}^2
  %   <\infty.
  % \end{align*}
  % 
  Hence,
  \begin{align*}
    2\mean
    \!\big[
    \sum_{x \in \mathbb{Z}^{d}}
    \om(\{0,x\}) \psi (\om) (\varphi(\tau_x \om)-\varphi(\omega))
    \big]
    \; = \; &
    \mean
    \!\big[
    \sum_{x \in \mathbb{Z}^{d}}
    \om(\{0,x\}) (\psi (\om) - \psi (\tau_x\om)) (\varphi(\tau_x\omega)-\varphi(\om))
    \big]\\
    \; = \; &
    - \big< \mD \psi, \mD \varphi \big>_{L_\mathrm{cov}^2},
  \end{align*}
  which concludes the proof of the lemma.
\end{proof}

We are finally ready to prove the annealed result, Theorem~\ref{thm:annealed}.

\begin{proof}[Proof of Theorem~\ref{thm:annealed}]
  We will show the conditions of Lemma~\ref{lem:abstract-limit-specific}.
  First, condition (1) follows from Proposition \ref{prop:first_level_conv_of_hamonic}. 
  For condition (2): part (i) follows from Proposition \ref{prop:sublinearity:pvar} applied to the function $\Psi:=\chi\indizeroo$. Part (ii)  follows from Proposition \ref{prop:correct-estimate-for-pot} applied to the function $\Psi=\Xi:=\chi\indizeroo$. Part (iii) follows from Proposition \ref{prop:ii-corrector-convergence-in-prob} applied to the function $\Psi=\Xi:=\chi\indizeroo$.
  Finally, to see condition (3) note that  for $\Pi\indizeroo\in L^2_{\mathrm{cov}}$ its orthogonal representation is given by $\Pi\indizeroo=\Phi\indizeroo+\chi\indizeroo$ with $\Phi\indizeroo\in L^2_{\mathrm{sol}}$ and $\chi\indizeroo\in L^2_{\mathrm{pot}}$.  In particular, $\langle \Phi^i, \chi^j \rangle_{L^2_{\mathrm{cov}(\PP)}}=0$ for all $1\le i,j\le d$. Hence condition (3) follows by applying Lemma~\ref{cor:ergodic_p-var} with $\Psi=\Phi$ and $\Psi=\chi$. Therefore we have shown all the conditions of Lemma~\ref{lem:abstract-limit-specific} hold. This concludes the proof of the Theorem.
\end{proof}

\section{Using regular stationary potential: the quenched settings}\label{rw}

In this section we will prove Theorem \ref{theorem}. For that we place ourselves in the quenched settings. More explicitly, for the rest of this section, in addition to Assumption \ref{ass:general}, we assume that Condition \ref{cond:moment-potential} holds.

%\subsection{Proof of Quenched lifted invariance principle}

\begin{lemma}[Martingale part: quenched Lindeberg, UCV and QV limit]
  \label{lem:quenched-M}
  Fix $T>0$. Then, for $\prob$-a.e.\ $\om$ the following hold for $M_t \coloneqq \Phi(\om,X_t)\indizeroo$:
  \begin{enumerate}[(i)]
  \item \label{helland-copy}
    The Lindeberg condition: for all $\delta > 0$ and $v\in\bbZ^d$ the compensator of 
    \begin{equation*}
      \sum_{0 < s \leq T} (v\cdot M^n_{s-,s})^2 \mathbbm{1}_{|v\cdot M^n_{s-,s}| > \delta} %\;\xrightarrow[n\to\infty]%{\;P_0^\om\;} 0.
    \end{equation*}
    converges to $0$ in $P_0^\om$-probability as $n\to\infty$.
  \item 
    The UCV condition: There exists a constant $C(\om)<\infty$ such that
    \begin{equation}\label{eq:quenched-QV-bound}
      \sup_{n\in\bbN} E_0^\om\big[ [M^n]_T \big]
      \;\le\; C(\om) \, T,
    \end{equation}
    where $[M^n]$ denotes the quadratic variation of $M^n$.
  \item The rescaled quadratic covariation 
    \begin{equation}\label{eq:quenc-qv-conv-a.s}
      \sup_{0\le t\le T}\big|[ M^{n,i},M^{n,j}]_t -\big\langle \Phi^i\indizeroo,\Phi^j\indizeroo \big\rangle_{L^2_{\mathrm{cov}}} t\big|
     \;\xrightarrow[n\to\infty]{\;P_0^\om\text{ a.s.}\;} 0
    \end{equation}
 (and hence also in $P_0^\om$-probability). 
  \end{enumerate}
\end{lemma}

\begin{proof}
We first prove (i). 
By the L\'evy system theorem \eqref{eq:Levy:system} for the function
  \[
  f^\omega(x,y) = v\cdot\Phi(\tau_x\om,y-x)^2\,\indizeroo \,\indicator_{|v\cdot\Phi(\tau_x \om,y-x)| > \de \sqrt{n}}
  \]
  we find that the compensator of 
  \[  \sum_{0 < s \leq T}
     \big(v\cdot M_{s-,s}^n\big)^2\,
    \indicator_{|v\cdot M^n_{s-,s}| \,>\, \de}
  \]
  is
   \[
   \frac{1}{n}
    \int_0^{T n}
    \sum_{y \in \mathbb{Z}^{d}} \tau_{X_{s-}}\omega(\{0, y\})\, (v\cdot \Phi(\tau_{X_{s-}}\omega, y))^2\,
    \indizeroo\,
    \indicator_{|v\cdot \Phi(\om, y)| > \de \sqrt{n}}\;
    \md s.
    \]

For $N\ge 0$ set $F_N:\Omega\to\bbR$ by 
\[
F_N(\omega)= E_0^\omega\!\bigg[
    \int_0^{T}
    \sum_{y \in \mathbb{Z}^{d}}  g_N(v,\Phi,\omega,X_{s-},y)\;
    \md s
    \bigg],
    \]
where 
\[
 g_N(v,\Phi,\omega,x,y):=
\tau_{x} \omega(\{0, y\})\, (v\cdot\Phi(\tau_{x}\omega, y))^2\,
    \indizeroo\,
    \indicator_{|v\cdot\Phi(\om, y)| > N}\;
    \]
Note that,  $g_N(v,\Phi,\omega,x,y)$ is non increasing in $N$. 
by stationarity, 
\begin{align*}
%\mean\!\bigg[F_N(\omega)\bigg]
%\;\le\;
 \mean\!\bigg[F_0(\omega)\bigg]
&  \;=\;
 T \mean\!\bigg[
    \sum_{y \in \mathbb{Z}^{d}} \omega(\{0, y\})\, (v\cdot\Phi(\omega, y))^2\,
    \indizeroo\;
    \bigg] \\
&\;=\; 
T \Norm{v\cdot\Phi\indizeroo}{L_\mathrm{cov}^2(\prob)}^2
\;<\;
\infty.
\end{align*}
%In particular, $F_0(\omega)<\infty$ $\prob$-a.s.\ 
Since $g_N(v,\Phi,\omega,x,y)$ vanishes in he limit as $N\to\infty$, and as  $g_N(v,\Phi,\omega,x,y) \le g_0(v,\Phi,\omega,x,y) <\infty$ for all $N\in\bbN$, then by dominated convergence applied to the product of $\mean_0$, Lebesgue measure on $[0,T]$ and the counting measure on $\bbZ^d$, we obtain 
$\mean\!\bigg[F_N(\omega)\bigg]
\;\xrightarrow[N\to\infty]{}\; 0.$
Further, by the ergodic theorem for the measure $\prob$ under the shift $\omega\mapsto\tau_{X_1}\omega$  applied to $F_N$ we obtain 
 \begin{align*}
    \frac{1}{n}\sum_{k=0}^{n-1} F_N(\tau_{X_k}\omega) \xrightarrow[n\to\infty]{\prob\text{-a.s.}}  &\mean\!\big[F_N(\omega)\big].
\end{align*}
Therefore
  \begin{align*}
   0& \;\le\;
   \limsup_{n\to\infty} E^\omega_0\!\bigg[
      \frac{1}{n}
    \int_0^{T n}
    \sum_{y \in \mathbb{Z}^{d}} \tau_{X_{s-}}(\{0, y\})\, (v\cdot \Phi(\tau_{X_{s-}}, y))^2\,
    \indizeroo\,
    \indicator_{|v\cdot \Phi(\om, y)| > \de \sqrt{n}}\;
    \md s \bigg]
    \\
     & \;\le\;
   \limsup_{n\to\infty} 
  \frac{1}{n}\sum_{k=0}^{n-1} F_N(\tau_{X_k}\omega) 
  \;=\;  \mean\!\big[F_N(\omega)\big],
  \end{align*}
and as the term on the right hand side can be chosen to be arbitrarily small by fixing $N$ to be large enough, this concludes the proof of (\ref{helland-copy}).

  We proceed with the proof of item (ii). Considering the case $N=0$ for the function $F_N$ defined in the proof of (i) above, we have
  $\prob$-a.s.\ that
  \begin{align*}
   & 
    E^\omega_0\!\bigg[
    \sum_{0 < s \leq T}
     (M_{s-,s}^n)^2\,
    \bigg]
 \;=\; 
  \frac{1}{n}\sum_{k=0}^{n-1} F_0(\tau_{X_k}\omega) 
  \;\xrightarrow[n\to\infty]{\prob\text{-a.s.}}\;  \mean\!\big[F_0(\omega)\big]
  \;=\;
  T \Norm{\Phi\indizeroo}{L_\mathrm{cov}^2(\prob)}^2.
  \end{align*}
In particular , (\ref{eq:quenched-QV-bound}) holds with 
$C(\omega):=\sup_n   E^\omega_0\!\bigg[
    \sum_{0 < s \leq T}
     (M_{s-,s}^n)^2\,
    \bigg]
 $ which is finite $\prob$-a.s.

For (iii):  (\ref{eq:quenc-qv-conv-a.s})
  follows by  Lemma~\ref{cor:ergodic_p-var}. 
  \end{proof}

\begin{lemma}[Corrector: quenched $p$-variation of the first level]
  \label{lem:quenched-R-pvar}
  Let $T>0$ and $2<p\le 2+\epsilon$, where $\epsilon>0$ is the one guaranteed by Condition \ref{cond:moment-potential}. Then for $\mathbf{P}$-a.e.\ $\om$,
  \begin{align*}
    E_0^\om\big[\,\|R^n\|_{p\text{-}\mathrm{var},[0,T]}^p\,\big]
   \;\xrightarrow[n\to\infty]{} 0.
   \end{align*}
\end{lemma}

\begin{proof}

  First notice that Minkowski inequality, for any $p>1$, the $p$-variation of a finite sequence of reals $f_0,f_1,f_2,...,f_N$  is bounded by twice their $\ell^p$ norm: 
  \begin{align*}
    \sup\bigg\{\sum_{j=1}^k |f_{n_{j-1},n_j}|^p : 0=n_0<n_1<...,n_k\le N,\,1\le k \le N \bigg\} \le 2^p \sum_{m=0}^N |f_{m}|^p. 
  \end{align*} 
  Indeed, by the elementary inequality $|a-b|^p \le 2^{p-1}(|a|^p+|b|^p)$, for any fixed $k\in\bbN$ and $0=n_0<n_1<...<n_k\le N$ we have
  \begin{align*}
    \sum_{j=1}^k |f_{n_{j-1},n_j}|^p \le 
    2^{p-1} 
    \sum_{j=1}^k (|f_{n_{j-1}}|^p + |f_{n_j}|^p)
    \le  2^{p} \sum_{j=0}^k |f_{n_{j}}|^p
    \le 2^p \sum_{m=0}^N |f_{m}|^p.
  \end{align*} 
  We apply this on $\Om_{0}$ to $R^n$ to get    
\begin{align*}
\|R^n\|_{p\text{-}\mathrm{var},[0,T]}^p
    & \;\le\;
    2^p 
    \sum_{0<u\le T} |n^{-1/2}\phi(\tau_{X_{(nu)-}}\omega)|^p
    \indicator_{X_{(nu)-}\ne X_{nu}}\\
   & \;=\;
    2^p \,n^{-p/2} \sum_{0<u\le nT} 
 |\phi(\tau_{X_{u-}}\omega)|^p
    \indicator_{X_u\ne X_{u-}}.
  \end{align*}
Set
\[
F(\omega):=E_{0}^{\omega} \left[\sum_{0<u\le T}  
 |\phi(\tau_{X_{u-}}\omega)|^p
    \indicator_{X_u\ne X_{u-}}\right]\indizeroo.
\]  
By the L\'evy system theorem \eqref{eq:Levy:system} for the function 
\[
f^{\omega}(x,y) := |\phi(\tau_{x}\omega)|^p\indicator_{x\ne y}
\]
we obtain
\[
F(\omega)
\;=\;
\int_{0}^T
E_{0}^{\omega} \left[
\sum_{x\in\bbZ^d}  \omega(\{X_{u-},X_{u-}+x\})
 |\phi(\tau_{X_{u-}}\omega)|^p
    \indicator_{x\ne 0}\right]\indizeroo\, d u.
\]
Therefore by stationarity of the process from the point of view of the walk we have
\begin{align*}
    \mean \left[F(\omega)\right]
& \;=\;
\mean_0 \left[\sum_{0<u\le T} 
 |\phi(\tau_{X_{u-}}\omega)|^p
    \indicator_{X_u\ne X_{u-}}\indizeroo\right]
    \;=\;
T\mean \big[ |\phi(\omega)|^p
    \mu^\om(0)\indizeroo
    \big]\\
&    \;\le\;
%2dbT\mean \left[ |\phi(\omega)|^p\indizeroo
%    \right]
 T\kappa   \;<\; \infty,
\end{align*}
where $\kappa$ is the constant from \eqref{eq:conditionSP} of Condition \ref{cond:moment-potential}.
By the ergodic theorem we obtain
\begin{align*}
 \frac1n
\sum_{k=0}^n
 F(\tau_{X_k}\omega) 
 \xrightarrow[n\to\infty]{\prob\text{-a.s.}} 
\mean\left[F(\omega)\right]
    <\infty,
  \end{align*}
and in particular on $\Om_0$
\begin{align*}
E_{0}^{\omega} \left[
\|R^n\|_{p\text{-}\mathrm{var},[0,T]}^p\right]
    & \;\le\;
    2^p \,n^{-p/2} E_{0}^{\omega} \left[
    \sum_{0<u\le nT} 
 |\phi(\tau_{X_{u}}\omega)|^p
    \indicator_{X_u\ne X_{u-}}\right]
 \xrightarrow[n\to\infty]{\prob\text{-a.s.}} 0.
  \end{align*}
This concludes the proof.   
\end{proof}
\begin{lemma}[Corrector iterated integrals and the matrix $\Ga$]
  \label{lem:quenched-R-area}
 Let $T>0$ and $2<p\le 2+\epsilon$, where $\epsilon>0$ is the one guaranteed by Condition \ref{cond:moment-potential}.
  Then, for $\mathbf{P}$-a.e.\ $\om$, and every $i,j\in\{1,\dots,d\}$  the following hold.
  \begin{enumerate}[(i)]
  \item We have
    \begin{align*} 
      \big\|    \big( (\bbR^n_{s,t})^{i,j} - (t-s)\Ga^{i,j}\big)_{(s,t)}
      \big\|_{\infty\text{-}\mathrm{var},[0,T]}
     \;\xrightarrow[n\to\infty]{\;P_0^\om\;} 0.
    \end{align*}
  \item The family $\{\|(\bbR^n)^{i,j}\|_{p/2\text{-}\mathrm{var},[0,T]}\}_{n\ge1}$ is tight under $P_0^\om$.
  \end{enumerate}
\end{lemma}

\begin{proof}
  For the proof of (i) fix $i,j\in\{1,\dots,d\}$ and consider the scalar coordinates $R^{n,i},R^{n,j}$ of $R^n$. By definition, and since $R^{n,i}_{0}=\frac{1}{\sqrt{n}}\mD \phi^i(\omega,0)=0$, on $\om\in \Om_0$ we have $P_0^\om$-a.s.,
  \begin{align*}
    (\bbR^n_{0,t})^{ij}
    & \;=\;
    I_{0,t}(R^{n,i},R^{n,j})\\
    &  \;=\;
    \frac1n \sum_{0<u\le nt} \phi^i(\tau_{X_{u-}}\omega)\big(\phi^j(\tau_{X_{u}}\omega) - \phi^j(\tau_{X_{u-}}\omega
)\big) 
    -   
    \frac1n \phi^i(\omega)
    \big( \phi^{j}(\tau_{X_{nt}}\omega)-\phi^{j}(\omega)\big).
  \end{align*} 
  where  %$\om_{u-}=\tau_{X_{u-}}\om$, $\om_u=\tau_{X_u}\om$.
  Next we note that 
  $\phi^i$ is $\prob$-measurable function so that  $\E[(\phi^i)^2(\om) \mu^\om(0)]\le 2db \E[(\phi^i)^2(\om) ]<\infty$ by construction of $\phi$ (see the paragraph above (\ref{1.3})). Furthermore, by translation invariance, nearest neighbor jumps and bounded conductances we get $\mean[\|\mD \phi^i \|_{L^2_{\mathrm{cov}}}\indizeroo] \le 2(2d)^2b \E[|\phi^i(\om)|\indizeroo] < \infty$.
  As in the proof of Lemma~\ref{lem:bdd-gradients-uniformly-integrable-ii}, by Lemma~\ref{lem:l2-uniformly-integrable-ii}, now it is enough to show that
  \begin{align*}
    \Big\|
    \frac1n \phi^i(\omega)
    \big( \phi^{j}(\tau_{X_{nt}}\omega)-\phi^{j}(\omega)\big)\indizeroo
    \Big\|_{\mathrm{unif}, [0,T]}
    \;\xrightarrow[n\to\infty]{\;P_0^\om\;} 0.
  \end{align*}
  For that, we observe first that
   \[
   \mean
  \!\big[
  \sum_{x \in \mathbb{Z}^{d}}
  \om(\{0,x\}) (\phi^j(\tau_x \om)-\phi^j(\om))\indizeroo
  \big] = \cL (\phi^j\indizeroo) = 0 
  \]
  by stationarity. 
  Note also that
  \begin{align}\label{eq:prod-phi}
    \Big\|
    \frac{1}{n} &
    \phi^i(\omega)(\phi^j(\tau_{X_{nt}}\omega)
    -\phi^j(\omega))
    \indizeroo
    \Big\|_{\mathrm{unif}, [0,T]}\\
    \nonumber&\;=\;
    |\phi^i(\omega)|
    \Big\|
    \frac{1}{ n}(\phi^j(\tau_{X_{nt}}\omega)-\phi^j(\omega))\indizeroo
    \Big\|_{\mathrm{unif}, [0,T]}.
  \end{align}
  Applying Lemma~\ref{lem:ergodic in uniform norm} to the function
  \begin{align*}
    f(\omega,\tilde\omega)
    \; := \;
    (\phi^j(\tilde\omega)-\phi^j(\omega))\indizeroo,
  \end{align*}
  we see that the right term of the product in the right hand side of (\ref{eq:prod-phi}) goes to zero $P^\om_0$-a.s for $\prob$-a.e $\om$. This completes the proof.

  For~(ii), we shall show  
  $\sup_n E_0^\om [\|\bbR^n\|_{p/2\text{-}\mathrm{var},[0,T]}]<\infty$
  for $\mathbf{P}$-a.e.\ $\om$. 
  Use again the presentation  
  \begin{align*}
    (\bbR^n_{s,t})^{i,j}
    & \;=\;
    \frac1n \sum_{ns<u\le nt} \phi^i(\tau_{X_{u-}}\omega)\big(\phi^j(\tau_{X_{u}}\omega) - \phi^j(\tau_{X_{u-}}\omega)\big) 
    -   
    \frac1n \phi^i(\tau_{X_{ns}}\omega)
    \big( \phi^{j}(\tau_{X_{nt}}\omega)-\phi^{j}(\tau_{X_{ns}}\omega)\big)
  \end{align*} 
  for $\omega\in\{0\in\cC_{\infty}\}$, to get that  
  \begin{align*}
    \|(\bbR^n)^{i,j}\|_{1\text{-}\mathrm{var},[0,T]} & \le 
    \frac5n \sum_{0<u\le nT:\, X_u\ne X_{u-}} \bigg( ((\phi^i)(\tau_{X_{u}}\omega))^2 + ((\phi^j)(\tau_{X_{u}}\omega))^2\bigg) \\
    & \lesssim 
    \frac1n \sum_{0<u\le nT:\, X_u\ne X_{u-}} \bigg( (\chi^i (\omega,X_{u}))^2 + (\chi^j (\omega,X_{u}))^2 + (\phi^i(\omega))^2+(\phi^j(\omega))^2 \bigg).
  \end{align*} 
  Set
  \[
  F(\omega) := 
    E_0^\om  \left[ 
    \sum_{0<u\le T} 
    \bigg(
      \phi^i(\omega)^2+
      \phi^j(\omega)^2 + |\phi^i(\tau_{X_{u-}}\omega)|^2 + |\phi^j(\tau_{X_{u}}\omega)|^2 \bigg)  \indicator_{X_u\ne X_{u-}}\indizeroo\right].
  \]
 using the  L\'evy system theorem \eqref{eq:Levy:system}
 together with the ergodic theorem for the process from the point of view of the walk we have
 
\begin{align*}
  \mean[F(\omega)] &
  \;=\; 
    T \mean  \left[ 
    \big(
      \phi^i(\omega)^2+
      \phi^j(\omega)^2 
     \big) \mu^\omega(0) \indizeroo
    + \sum_{x\in\bbZ^d}\omega(0,x) \big(|\phi^i(\tau_x\omega)|^2 + |\phi^j(\tau_{x}\omega)|^2 \big) \indizeroo \right]    \\
   & \;\le\; 
    % 4db T \mean_0  
    % \left[ 
    %   \big(\phi^i(\omega)^2+
    %   \phi^j(\omega)^2)\big) \indizeroo 
    %   \right] 
    4T\kappa \;<\; \infty,  
\end{align*}
where $\kappa$ is the constant from \eqref{eq:conditionSP} of Condition \ref{cond:moment-potential}. 
 As in the proof of Lemma~\ref{lem:quenched-R-pvar}, by the ergodic theorem, we obtain
\begin{align*}
    E_0^\om [\|(\bbR^n)^{i,j}\|_{p/2\text{-}\mathrm{var},[0,T]}] &
    \le
    E_0^\om [\|(\bbR^n)^{i,j}\|_{1\text{-}\mathrm{var},[0,T]}] 
     \lesssim 
    \frac1n   \sum_{k=0}^{n-1}
    F(\tau_{X_k}\omega) \;\le\; C(\omega)<\infty  
  \end{align*} 
  $\prob$-a.s. 
  In particular, this shows tightness of $\|(\bbR^n)^{i,j}\|_{p/2\text{-}\mathrm{var},[0,T]}$ under $P^\omega_0$ and completes the proof. 
\end{proof}
The next lemma asserts that for $\mathbf{P}$-a.s. the $p/2$-variation of the quadratic covariation of the mixed terms vanishes in the limit almost surely in the quenched law. 
\begin{lemma}[Mixed quadratic covariation vanishes]
  \label{lem:quenched-mixed}
  Let $p>2$ and $T>0$. 
  Then, 
  \begin{align*}
    \big\| (Q_{s,t}(M^n,R^n))_{(s,t)\in\Delta_{[0,T]}} \big\|_{p/2\text{-}\mathrm{var},[0,T]}
    \;\xrightarrow[n\to\infty]{\;P_0^\om\text{-a.s.}\;} 0
  \end{align*}
  for $\mathbf{P}$-a.e.\ $\om$. 
In particular, $ \big\| (Q_{s,t}(M^n,R^n))_{(s,t)\in\Delta_{[0,T]}} \big\|_{p/2\text{-}\mathrm{var},[0,T]}
$ is tight under $P_0^\om$.
\end{lemma}

\begin{proof}
 Recall that the quadratic covariation is
  \begin{align*}
    Q_{s,t}(M^n,R^n)
    \;=\;
    \frac1n \sum_{sn<u\le nt} M_{u-,u} \otimes R_{u-,u}
    \;=\;
    \frac1n Q_{sn,tn}(M,R).
  \end{align*}
  The orthogonal decomposition of the increment field on $\Omega_0$ is $\Pi \indizeroo=\Phi\indizeroo+\chi\indizeroo$ in $L^2_{\mathrm{cov}}$ (as in the proof of Theorem~\ref{thm:annealed}) and in particular,
  \begin{align*}
    \big\langle \Phi,\chi\big\rangle_{L^2_{\mathrm{cov}(\PP)}} \;=\; 0.
  \end{align*} 
  Applying Lemma~\ref{lem:ergodic in uniform norm}, it follows
  that for $\mathbf{P}$-a.e.\ $\om$,
  \begin{align*}
    \big\| (Q_{s,t}(M^n,R^n))_{(s,t)\in\Delta_{[0,T]}} \big\|_{p/2\text{-}\mathrm{var},[0,T]}
    \;\xrightarrow[n\to\infty]{\;P_0^\om\text{-a.s.}\;} 0.
  \end{align*}
\end{proof}

\begin{proof}[Proof of Theorem~\ref{theorem}]
  It is enough to prove the theorem for $p>2$ arbitrarily close to $2$. We will show it for $2<p<2+\epsilon$ with the $\epsilon$ from Condition \ref{cond:moment-potential}. We will show the conditions of Lemma~\ref{lem:abstract-limit-specific} holds.
  Condition (1) Martingale hold by Lemma~\ref{lem:quenched-M}. Condition (2) (i) follows from Lemma~\ref{lem:quenched-R-pvar}, whereas Conditions (2) (ii) and (iii) follow from Lemma~\ref{lem:quenched-R-area}. Finally, Condition (3) follows from Lemma~\ref{lem:quenched-mixed}.
  Hence, the proof of the Theorem is completed. 
  \end{proof}  

\section{Scalar covariance and area correction in supercritical percolation }\label{subsec:ScalarMarticiesInIID}

\begin{proposition}
In the supercritical percolation with uniformly elliptic weights settings of Example \ref{ex:percolation_transfer_cor}
the covariance matrix and the area correction become
  a multiple of the identity matrix $\Sigma^2=\sigma^2 I$ and $\Gamma=\gamma I$, where $\sigma^2>0$ and $\gamma\le 0$.
\end{proposition}

  \begin{proof} To see that $\Sigma$ is a multiple of the identity matrix, first observe that the construction of the corrector and the distribution of the environment are invariant under a relabeling of the coordinates, which implies that all off-diagonal entries have to be equal and that all on-diagonal entries have to be equal. Furthermore the construction of the corrector (and thus the martingale) and the distribution of the environment are invariant under the  symmetry of reflecting the $i$-th coordinate, i.e. $x_i \mapsto -x_i$. From this we see that for $i \neq j$
  \begin{align*}
   \E_0 
    \left[\sum_{0<s\leq 1} (M_{s-,s})^i(M_{s-,s})^j \right] 
    = 
    \E_0 \left[\sum_{0<s\leq 1} -(M_{s-,s})^i(M_{s-,s})^j \right] . 
  \end{align*}
  Hence  $\Sigma^{ij} 
    = - \Sigma^{ij}$
  which implies that all off-diagonal entries are 0. As the harmonic embedding $x \mapsto x - \chi(x,\omega)$ is not trivial (remember that the corrector grows sublinearly) we also get that $\sigma^2>0$.
  Next,
  \begin{align*}
    \Gamma = -\frac12 \big< \chi, \chi \big>_{L_\mathrm{cov}^2(\PP)}
    =
    -\frac12 (m^2-\sigma^2)I,
  \end{align*}
  $I$ is the identity matrix, $\sigma^2=\EE\![
  {\textstyle \sum_{|e|=1}}\, \om(e)\, |\Phi^1(\om,e)|^2]$ and $m^2= \EE\! [\mu^\om(0)]$.
  Indeed,
  we showed that $\Sigma = \sigma^2 I$.
  Notice that by independent and symmetry of the coordinates we also have
  \begin{align*}
    \big< \Pi, \Pi \big>_{L_\mathrm{cov}^2(\PP)} = \EE\![
    {\textstyle \sum_{|e|=1}}\, \om(e)\, \Pi(\om,e)^{\otimes 2} ] 
    =
    \EE\! [\mu^\om(0)] \, I = m^2 I.
  \end{align*}
  Therefore, the orthogonality $\big< \Phi\indizeroo, \chi\indizeroo \big>_{L_\mathrm{cov}^2(\PP)}=0$ yields 
  \begin{align*}
    \big< \chi, \chi \big>_{L_\mathrm{cov}^2(\PP)} =
    \big< \Pi, \Pi \big>_{L_\mathrm{cov}^2(\PP)}
    -
    \big< \Phi, \Phi \big>_{L_\mathrm{cov}^2(\PP)}.
    =  (m^2-\sigma^2) I,
  \end{align*}
  as claimed, and in particular
  \begin{align*}
    m^2-\sigma^2 =
    \EE\![
    {\textstyle \sum_{|e|=1}}\, \om(e)\, |\chi^1(\om,e)|^2]
    \ge 0.
  \end{align*}
\end{proof}

\section{The Stratonovich lift }\label{subsec:ComparisonII}

\begin{theorem}
  In the settings of Theorem~\ref{thm:annealed} (respectively Theorem $\ref{theorem}$),
  let 
  \begin{align*}
    \bar\bbX^n_{s,t} \ldef \int_{s}^t \tilde X^n_{s,u}\otimes \md X^n_{u},
  \end{align*}
  where $\tilde X^n$ is obtained from $X^n$ by linear interpolations between consecutive jumps and $\int_{s}^t \tilde X^n_{u}\otimes \md X^n_{u}$ is the Riemann-Stieltjes integral.
  Then, for every $p>2$
  \begin{align*}
    (X^n,\bar\bbX^n)
    \;\underset{n \to \infty}{\;\Longrightarrow\;} 
    \big(B,\; \bbB^{\mathrm{STR}} \big)
    \quad\text{in }
    \cD_{p\text{-var}}([0,T], \bbR^d\times \bbR^{d \times d}),
  \end{align*}
  where $B_t$ is the same $d$-dimensional Brownian motion from Theorem~\ref{thm:annealed} (respectively from Theorem $\ref{theorem}$), $\bbB^{\mathrm{STR}}_{s,t}\;=\;\int_{s}^{t} B_{s,u} \otimes \circ \md B_u$
  is its Stratonovich iterated integral, and the convergence holds under $\PP_0$ (respectively, under $P_{0}^{\omega}$ for $\mathbf{P}$-almost every $\omega$).
\end{theorem}
\begin{proof}
  Note that $\Gamma=-\frac12\langle \chi,\chi \rangle_{L^2_\mathrm{cov}(\PP)}$ is symmetric and
  \begin{align*}
    \bar\bbX^n_{s,t} \;=\; \bbX^n_{s,t}+\frac12 Q_{s,t}(X^n,X^n).
  \end{align*}
  Indeed, in the both the proofs of Theorems and \ref{thm:annealed} and \ref{theorem} we have shown the assumptions of Lemma~\ref{lem:abstract-limit-specific} are satisfied. The proof now follows from the Stratonovich part of Lemma~\ref{lem:abstract-limit-specific}.    
\end{proof}

\begin{remark}
  The fact the It\^o lift has a correction in the second level whereas the Stratonovich lift does not is not new and moreover, by now the area correction is well-understood: in fairly general stationary settings, the correction has a simple Green-Kubo type presentation \cite{EFO24}.          
\end{remark}

\section{Acknowledgments}
TO is grateful to Istituto Nazionale di Alta Matematica for financial support in the framework of INDAM-GNAMPA Project CUP-E53C22001930001.

\appendix

\section{Ergodic theorem}
In this appendix we provide an extension of the Birkhoff ergodic theorem along the trajectory of the process. In what follows the probability measure $\prob_0$ is the annealed law of the walk starting at the origin, that is given by $\prob_0(\cdot)=\mean[P^\om_0(\cdot)]$. $\mean_0$ is the corresponding expectation. The notation for $\PP_0$ and $\EE_0$, the conditional annealed law as as before. In particular, $\PP_0(\cdot)=\mean[P^\om_0(\cdot)| 0\in\mathcal{C}_{\infty}]$.   
\begin{lemma}\label{lemma:ergodic}
  Assume assumption \ref{ass:general}. Let $f\!: \Om \times \Om \to \bbR$ be a function that vanishes on the diagonal, that is $f(\om, \om) = 0$.  In addition, assume that
  \begin{align}\label{eq:assumption-integrability}
    \mean\!\big[\sum_{x \in \mathbb{Z}^{d}} \om(\{0,x\}) |f(\om, \tau_x \om)|\big] < \infty.
  \end{align}
  Then,
  \begin{align}\label{eq:ergodic}
    \lim_{t \to \infty}
    \frac{1}{t}
    \sum_{0 < s \leq t} f(\tau_{X_{s-}} \om, \tau_{X_s} \om )
    \;=\;
    \mean\!\bigg[
    \sum_{x \in \mathbb{Z}^{d}} \om(\{0,x\})\, f(\om, \tau_x \om)
    \bigg]
  \end{align}
  $\prob_0$-almost surely and in $L^1(\prob_0)$.

\end{lemma}
\begin{proof}
Notice that the L\'evy system theorem \eqref{eq:Levy:system} implies that
  \begin{align*}
    \sum_{0 < s \leq t, X_{s-}\ne X_s} f(\tau_{X_{s-}} \om, \tau_{X_s} \om )
    \,-\,
    \int_{(0,t]}
    \sum_{x \in \mathbb{Z}^{d}} \om(\{X_{s-},x\})\,
    f(\tau_{X_{s-}} \om, \tau_{x} \om )\,
    \md{s}
  \end{align*}
  is a $P_0^{\om}$-martingale $\prob$-a.s. Fubini's theorem together with the stationarity of the process as seen from the particle yield
  \begin{align}\label{eq:ergodic:limit}
    \E_0\!\bigg[
    \sum_{0 < s \leq t} f(\tau_{X_{s-}} \om, \tau_{X_s} \om )
    \bigg]
    % &\,=\,
    % \E_0\!\bigg[
    % \int_{(0,t]}
    % \sum_{y \in \mathbb{Z}^{d}} \om(\{X_{s-},y\})\,
    % f(\tau_{X_{s-}} \om, \tau_{y} \om )\,
    % \md{s}
    % \bigg]
    % \\
    \;=\;
    t\,
    \mean\!\bigg[
    \sum_{x \in \mathbb{Z}^{d}} \om(\{0, x\})\, f(\om, \tau_x \om)
    \bigg].
  \end{align}
  Set $F(\om) \ldef \sum_{0 < s \leq 1} f(\tau_{X_{s-}} \om, \tau_{X_s} \om)$ to lighten notation.  
  Repeating the above argument for $|f|$ instead of $f$ we notice that $F \in L^1(\prob_0)$. In particular, an application of Birkhoff's ergodic theorem yields
  \begin{align*}
    \lim_{t \to \infty} \frac{1}{t}
    \sum_{k=0}^{\lfloor t \rfloor - 1} F(\tau_{X_k} \om)
    &\;=\;
    % \lim_{t \to \infty}
    % \bigg(
    % 1 - \frac{t - \lfloor t \rfloor}{t}
    % \bigg)
    % \frac{1}{\lfloor t \rfloor}
    % \sum_{k=0}^{\lfloor t \rfloor - 1} F(\tau_{X_k} \om)
    % \;=\;
    \lim_{t \to \infty} \frac{1}{\lfloor t \rfloor}
    \sum_{k=0}^{\lfloor t \rfloor-1} F(\tau_{X_{k}} \om)
    \\
    % &\underset{n \to \infty}{\;\longrightarrow\;}
    &\;=\;
    \E_0\!\big[F(\om)\big]
    \overset{\eqref{eq:ergodic:limit}}{\;=\;}
    \mean\!\bigg[
    \sum_{x \in \mathbb{Z}^{d}} \om(\{0, x\})\, f(\om, \tau_x \om)
    \bigg]
  \end{align*}
  in  $L^1(\prob_0)$ and $\prob_0$-a.s.
  Next we write
  \begin{align}\label{eq:ergodic:split}
    \frac{1}{t}
    \sum_{0 < s \leq t} f(\tau_{X_{s-}} \om, \tau_{X_s} \om)
    \;=\;
    % \\
    % &\mspace{36mu}=\;
    % \Big(
    % 1 - \frac{t - \lfloor t \rfloor}{t}
    % \Big)
    % \frac{1}{\lfloor t \rfloor}
    \frac{1}{t}
    \sum_{k=0}^{\lfloor t \rfloor-1} F(\tau_{X_{k}} \om)
    \,+\,
    \frac{1}{t}
    \sum_{\lfloor t \rfloor < s \leq t}\mspace{-4mu}
    f(\tau_{X_{s-}} \om, \tau_{X_s} \om).
  \end{align}
  Thus, it remains to show that the second term on the right-hand side of \eqref{eq:ergodic:split} vanishes $\prob_0$-a.s.~and in $L^1(\prob_0)$ as $t \to \infty$.  For this purpose,
  set $G(\om) \ldef \sum_{0 < s \leq 1} |f(\tau_{X_{s-}} \om, \tau_{X_s} \om)|$.  Since, by using again the stationarity of the process as seen from the particle,
  \begin{align*}
    \sum_{k=1}^{\infty}
    \p_0\!\big[j\, G(\tau_{X_k} \om) > k\big]
    \;=\;
    \sum_{k=1}^{\infty}
    \p_0\!\big[j\, G(\om) > k\big]
    \;\leq\;
    j\E_0\!\big[G(\om)\big]
    \;<\;
    \infty,
  \end{align*}
  an application of the first Borell-Cantelli lemma yields that
  \begin{align}\label{eq:ergodic:borell-cantelli}
    \p_0\!\big[
    \limsup_{k \to \infty}\, \Big\{G(\tau_{X_k} \om) > \frac{k}{j}\Big\}
    \bigg]
    \;=\;
    0,
    \qquad \forall\, j \in \bbN.
  \end{align}
  Since
  \begin{align*}
    &\Big\{
    \lim_{t \to \infty}
    \Big|
    \frac{1}{t}\,
    {\textstyle \sum_{\lfloor t \rfloor < s \leq t}}\,
    f(\tau_{X_{s-}} \om, \tau_{X_s} \om)
    \Big|
    \;=\;
    0
    \Big\}^{\!c}\\
    &\;\subset\;
    \Big\{
    \lim_{t \to \infty} \frac{1}{\lfloor t \rfloor}\,
    G(\tau_{X_{\lfloor t \rfloor}} \om) = 0
    \Big\}^{\!c}
    \;\subset\;
    \bigcup_{j=1}^{\infty}\,
    \limsup_{k \to \infty}\, \Big\{ G(\tau_{X_k} \om) > \frac{k}{j} \Big\},
  \end{align*}
  we finally obtain that, in view of \eqref{eq:ergodic:borell-cantelli}, the second term on the right-hand side of \eqref{eq:ergodic:split} vanishes $\prob_0$-a.s. as $t \to \infty$.  The convergence in $L^1(\prob_0)$ is immediate because
  \begin{align*}
    % \lim_{t \to \infty}
    \E_0\!\bigg[\,
    \bigg|
    \frac{1}{t} \sum_{\lfloor t \rfloor < s \leq t}\mspace{-5mu}
    f(\tau_{X_{s-}}\om, \tau_{X_s} \om)
    \bigg|\,
    \bigg]
    \;\leq\;
    % \lim_{t \to \infty}
    \frac{1}{t}
    \mean\!\bigg[
    \sum_{x \in \mathbb{Z}^{d}} \om(\{0,x\})\, |f(\om, \tau_x \om)|
    \bigg]
    \underset{t \to \infty}{\;\longrightarrow\;}
    0.
  \end{align*}
  Thus both the a.s.~and the $L^1$ statement of the lemma follows.
\end{proof}
%%%%%%%%%%%%%%%% 
%%%%%%%%%%%%%%%%%%%%%%%%%%%%%%%%%%%%%%%%%%%%%%%%%%%%%%%%%%%%%%%%%%%%%%%%%%%%%%%%%%%%%%% 
\begin{lemma}\label{lemma:stat-seq-ergodic-thm-implies-uniform-conv}
  Let $(F_n)_{n\in\N}$ be a stationary sequence of random variables so that $\E[|F_1|]<\infty$ and
  \begin{equation}\label{eq:appendix1}
    \frac{1}{n}\sum_{k=1}^n F_k 
    \;\xrightarrow[n\to\infty]{}
     \E[F_1] \,\prob\text{-a.s.}%\ and in $L^1(\prob)$}.
  \end{equation}
  Then, for any $T>0$ we have
  \begin{equation}
    \sup _{0\le t\le T}\big |  \frac 1 n \sum_{k=1}^{\lfloor nt\rfloor} F_k - t \E[F_1] \big|
     \;\xrightarrow[n\to\infty]{}
     0 \,\prob\text{-a.s.}  \end{equation}
If additionally
  \begin{equation}
    \big |  \frac 1 n \sum_{k=1}^{n} |F_k| -  \E[|F_1|] \big|
     \;\xrightarrow[n\to\infty]{}
     0 \,\text{ in }L^1(\prob),  
     \end{equation}
then 
 \begin{equation}
    \sup _{0\le t\le T}\big |  \frac 1 n \sum_{k=1}^{\lfloor nt\rfloor} F_k - t \E[F_1] \big|
     \;\xrightarrow[n\to\infty]{}
     0 \,\text{ in }L^1(\prob).  \end{equation}
\end{lemma}
\begin{proof}
  We first show the convergence holds almost surely. Fix $T>0$ and $\delta>0$. Since \eqref{eq:appendix1} holds a.s., there are random $M,N\in\N$ so that
  \begin{align*}
    \left  |  \frac 1 n \sum_{k=1}^{n} F_k \right|\le M \text{ for all } n\ge 1
  \end{align*}
  and for all $n\ge N$
  \begin{align*}
    \left  |  \frac 1 n \sum_{k=1}^{n} F_k - \E[F_1] \right|<\frac{\delta}{3T}.
  \end{align*}
  Let $\epsilon \coloneqq\frac{\delta}{4(|\E[F]| + M)}>0$.
  For $0<t\le \epsilon$, we have for all $n$
  \begin{align*}
    \left |  \frac {1} {n} \sum_{k=1}^{\lfloor nt\rfloor} F_k -  t\E[F] \right |  =
    t\left |  \frac{\lfloor nt\rfloor}{\lfloor nt\rfloor} \frac {1} {\lfloor nt\rfloor } \sum_{k=1}^{\lfloor nt\rfloor} F_k -  \E[F] \right |
    \le \epsilon (M+|\E[F]|)
    < \delta.
  \end{align*}
  On the other hand, for all $t\in[ \epsilon,T]$ and for all $n$ large enough so that $N\le \epsilon n$ and $\frac{1}{n} |\E[F_1] |< \frac{\delta}{2}$ we have
  \begin{align*}
    \bigg  |  \frac {1} {n } \sum_{k=1}^{\lfloor nt\rfloor} F_k - t\E[F_1] \bigg |
    \le
    \frac{\lfloor nt\rfloor}{n}  \bigg  |   \frac {1} {\lfloor nt\rfloor } \sum_{k=1}^{\lfloor nt\rfloor} F_k -  \E[F_1] \bigg  | 
    +
    \big  |  (t - \frac{\lfloor nt\rfloor}{n} ) \E[F_1] \big  | 
    \le
    t\frac{\delta}{3T} +     \frac{1}{n}  |\E[F_1] | < \delta.
  \end{align*}
  Hence, for all such $n$
  \begin{align*}
    \sup _{0\le t\le T} \bigg |  \frac {1} {n} \sum_{k=1}^{\lfloor nt\rfloor} F_k - t \E[F_1] \bigg |< \delta
  \end{align*}
  and we obtained the $\prob$-a.s.\ convergence. To get the convergence in $L^1(\prob)$, first note that  
  \begin{align}\label{eq:UI-domination}
    \sup _{0\le t\le T} \bigg |  \frac {1} {n} \sum_{k=1}^{\lfloor nt\rfloor} F_k - t \E[F_1] \bigg |
    &\le
    \frac {1} {n} \sum_{k=1}^{\lfloor nT\rfloor}  |F_k|  + T \E[|F_1|] 
%    \\
%  &  \le 
%  \left| \frac {1} {n} \sum_{k=1}^{\lfloor nT\rfloor}  |F_k| - T \E[|F_1|] \right| + 2 T \E[|F_1|]
.
  \end{align}
By assumption, 
$
  \{\frac 1 n \sum_{k=1}^{n} |F_k|\}_{n\ge 1} 
$ converges  in $L^1(\prob)$ as $n\to\infty$. Hence also does $
  \{\frac 1 n \sum_{k=1}^{\lfloor nT\rfloor} |F_k|  + T \E[|F_1|] \}_{n\ge1} 
$, 
and in particular is uniformly integrable.
%  \begin{align*}
%   \mean\left[ \left| \frac {1} {n} \sum_{k=1}^{\lfloor nT\rfloor}  |F_k| - T \E[|F_1|] \right|\right] \xrightarrow[n\to\infty]{}0.  
%  \end{align*}
By (\ref{eq:UI-domination}) we obtain that $\left\{\sup _{0\le t\le T} \bigg |  \frac {1} {n} \sum_{k=1}^{\lfloor nt\rfloor} F_k - t \E[F_1] \bigg |\right\}_{n\ge 1}$ is also uniformly integrable. With the convergence $\prob$-almost surely, Vitali's convergence theorem (cf. \cite[Theorem 5.5.2.]{Durrett2010PTE4}) gives us also the convergence in $L^1(\prob)$, as required. 
\end{proof}
%%%%%%%%%%%%%%%%%%%%%% 
% 
\begin{lemma}\label{lem:ergodic in uniform norm}
  Suppose that the assumptions of Lemma~\ref{lemma:ergodic} holds. Then, for all $T>0$
  \begin{align*}
    \sup_{0 < t \leq T}
    \bigg|
    \frac{1}{n} \sum_{0 < s \leq tn} f(\tau_{X_{s-}} \om, \tau_{X_s} \om )
    - t \mean\!\big[\sum_{x \in \mathbb{Z}^{d}} \om(\{0,x\}) f(\om, \tau_x \om) \big]
    \bigg|
    \;\xrightarrow[n\to\infty]{}
    0 \end{align*}
  in  $L^1(\prob_0)$ and $\prob_0$-a.s.
\end{lemma}
\begin{proof}
  As in the proof of Lemma~\ref{lemma:ergodic}, set $F(\om) \ldef \sum_{0 < s \leq 1} f(\tau_{X_{s-}} \om, \tau_{X_s} \om)$, and write
  \begin{align}\label{eq:ergodic:split2}
    \frac{1}{n}
    \sum_{0 < s \leq nt} f(\tau_{X_{s-}} \om, \tau_{X_s} \om)
    \;=\;
    % \\
    % &\mspace{36mu}=\;
    % \Big(
    % 1 - \frac{t - \lfloor t \rfloor}{t}
    % \Big)
    % \frac{1}{\lfloor t \rfloor}
    \frac{1}{n}
    \sum_{k=0}^{\lfloor nt \rfloor-1} F(\tau_{X_{k}} \om)
    \,+\,
    \frac{1}{n}
    \sum_{\lfloor nt \rfloor < s \leq nt}\mspace{-4mu}
    f(\tau_{X_{s-}} \om, \tau_{X_s} \om).
  \end{align}

  We shall first treat the second term. Let $X_k\coloneqq \sum_{k < s \leq k+1} \mspace{-4mu} |f(\tau_{X_{s-}} \om, \tau_{X_{s}} \om)|.$
  Then $X_k$ is a stationary sequence, $X_0\ge0$ $\prob_0$-a.s., $\E_0[X_0]<\infty$ and moreover
  \begin{align*}
    \sup_{0 < t \leq T}\frac{1}{n} \bigg|
    \sum_{\lfloor nt \rfloor < s \leq nt}\mspace{-4mu}
    f(\tau_{X_{s-}} \om, \tau_{X_s} \om) \bigg|
    \;\le\;
    \frac{1}{n} \max\{X_0,X_1,...,X_{\lfloor nT\rfloor +1}\}.
  \end{align*}
  By Lemma 2.30 of \cite{KLO12} the right hand side converges to zero, both in  $L^1(\prob_0)$ and $\prob_0$-a.s. Hence, using \eqref{eq:ergodic:limit} it is left to show the following limit holds both in $L^1(\prob_0)$ and $\prob_0$-a.s.
  \begin{align*}
    \sup_{0 < t \leq T}
    \bigg|
    \frac{1}{n}\sum_{k=0}^{\lfloor nt \rfloor-1} F(\tau_{X_{k}} \om)
    - t \E_0\!\big[F(\om)\big]
    \bigg|
   \;\xrightarrow[n\to\infty]{} 0,
  \end{align*}
  which is an application of Lemma~\ref{lemma:stat-seq-ergodic-thm-implies-uniform-conv} to the stationary sequence $F_k\coloneqq F(\tau_{X_{k}} \om)$ using the assumption \eqref{eq:assumption-integrability} and Lemma~\ref{lemma:ergodic}.
\end{proof}
\begin{coro}\label{cor:conditional-ergodic-in-suo-norm} Suppose that the assumptions of Lemma~\ref{lemma:ergodic} holds with $f(\om, \tilde \om)$ replaced by $f(\om, \tilde \om)\,\indizero$. Then,  the assertion of Lemma~\ref{lem:ergodic in uniform norm}, modified so that the expectation $\mean$ is replaced by $\EE$, holds also when the convergence is taken to be with respect to $L^1(\PP)$ and $\PP$-a.s.
 \end{coro}
\begin{proof}
   Applying Lemma~\ref{lem:ergodic in uniform norm} to the function 
   \[
   f(\om, \tilde \om)\,\indizero,
   \]
   we obtain, for all $T>0$,
  \begin{align*}
    \sup_{0 < t \leq T}
    \bigg|
    \frac{1}{n} \sum_{0 < s \leq tn} f(\tau_{X_{s-}} \om, \tau_{X_s} \om )\,\indicator_{0\in\mathcal{C}_\infty(\tau_{X_{s-}}\omega)}
    - t \mean\!\big[\sum_{x \in \mathbb{Z}^{d}} \om(\{0,x\}) f(\om, \tau_x \om) \,\indicator_{0\in\mathcal{C}_\infty(\omega)} \big]
    \bigg|
 \xrightarrow[n\to\infty]{}
    0 
  \end{align*}
  in  $L^1(\prob_0)$ and $\prob_0$-a.s. 
  Note that, starting at the origin the walk stays in the initial cluster at all times, that is $\mathcal{C}_0(\tau_{X_t}\om)=\mathcal{C}_0(\omega)$ for all $t\ge 0$. Using the uniqueness of the infinite cluster we have in particular that 
  the events $\{0\in\mathcal{C}_\infty(\omega)\}$ and $\{0\in\mathcal{C}_\infty(\tau_{X_{s-}}\omega)\}$ coincide $\prob_0$-a.s. 
  Dividing by $\prob({0\in\mathcal{C}_\infty(\omega)})$ we therefore obtain
\begin{align*}
    \sup_{0 < t \leq T}
    \bigg|
    \frac{1}{n} \sum_{0 < s \leq tn} f(\tau_{X_{s-}} \om, \tau_{X_s} \om )
    - t \EE\!\big[\sum_{x \in \mathbb{Z}^{d}} \om(\{0,x\}) f(\om, \tau_x \om)  
    \big]
    \bigg|
     \;\xrightarrow[n\to\infty]{} 0
  \end{align*}
  in  $L^1(\PP_0)$ and $\PP_0$-a.s., as required.
  \end{proof}

\begin{lemma}\label{cor:ergodic_p-var}
  Let $\Psi,\Xi\in L^2_{\mathrm{cov}}$ and let $T>0$ and $p>2$.
  For every $(s,t)\in\Delta_{[0,T]}$ we set
  \begin{align*}
    A^n_{s,t}
    \;\coloneqq\;
    \frac{1}{n}   \cQ_{ns,nt}(\Psi,\Xi) - (t-s)  \langle \Psi,\Xi \rangle_{L^2_\mathrm{cov}(\PP)},
  \end{align*}
  where
  \begin{align*}
    \cQ_{s,t}(\Psi,\Xi)
    \;\coloneqq\;
    Q_{s,t}(\Psi(\omega,X_\cdot),\Xi(\omega,X_\cdot)).
  \end{align*}
  Then, 
  \begin{align*}
    \big \| A^n \big\|_{p/2\text{-}\mathrm{var},\, [0,T]}
     \;\xrightarrow[n\to\infty]{}
 0
    \text{ in $L^1(\PP_0)$ and}
    %	\text{ moreover }
    % \big \| A^n \big\|_{\infty\text{-}\mathrm{var},\, [0,T]}
    % \to 0
    \;\; \PP_0\text{-a.s.}
  \end{align*}
\end{lemma}
\begin{proof}
  \textbf{Step 1.} Assume first that $\omega$ is non-periodic $\prob$-a.s., that is $\prob(\tau_x\omega=\omega)=0$ for all $x\in\mathbb{Z}^{d}$. Therefore $(\tau_{X_t}\omega,X_0)$ uniquely defines $(\omega, X_t)$. Indeed, $X_t = \sum_{x\in\mathbb{Z}^{d}}x\indicator_{\tau_x\omega=\tau_{X_t}\omega}$. 
  In particular, the function 
  \begin{align*}
    \eta (\omega,\omega')\ldef \sum_{x\in\mathbb{Z}^{d}}x\indicator_{\tau_x\omega=\omega'}
  \end{align*}
  is well-defined and moreover, the function
  \begin{align*}
    f (\omega,\omega') \coloneqq
    \Psi\big(\omega, \eta (\omega,\omega') \big)\,
    \Xi\big(\omega,\eta (\omega,\omega')\big)\indizero
  \end{align*}
  satisfies Assumption \ref{eq:assumption-integrability}.
  Note also that we have
  \begin{align*}
    \mean\!\bigg[
    \sum_{x \in \mathbb{Z}^{d}} \om(\{0,x\})\, f(\om, \tau_x \om)
    \bigg]
    \;=\; \mean\!\bigg[\sum_{x \in \mathbb{Z}^{d}} \om(\{0,x\}) \Psi(\om,x)\,\Xi(\om,x) \indizero \bigg] 
  \end{align*}
so that 
    \begin{align*}
    \EE\!\bigg[
    \sum_{x \in \mathbb{Z}^{d}} \om(\{0,x\})\, f(\om, \tau_x \om)
    \bigg]
    \;=\;
    \langle \Psi,\Xi \rangle_{L^2_\mathrm{cov}(\PP)}.
  \end{align*}
Moreover,
\begin{align*}
    \sum_{0 < s \leq tn} f(\tau_{X_{s-}} \om, \tau_{X_s} \om )
    \;=\;
    \cQ_{0,tn}(\Psi,\Xi)\indizero.  
  \end{align*}
  Applying Lemma~\ref{lem:ergodic in uniform norm}, we get
  \begin{align*}
    \sup_{0 < t \leq T}
    \bigg|
    \frac{1}{n} \cQ_{0,tn}(\Psi,\Xi)\indizero 
    - t \mean\!\big[\sum_{x \in \mathbb{Z}^{d}} \om(\{0,x\}) \Psi(\om,x)\,\Xi(\om,x) \indizero\big]
    \bigg|
    \;\xrightarrow[n \to \infty]\;
    0  
  \end{align*}
  in  $L^1(\prob_0)$ and $\prob_0$-a.s.
  Dividing by $\prob(0\in\mathcal{C}_\infty(\omega))$ yields 
  \begin{align*}
    \big \| A^n \big\|_{\infty\text{-}\mathrm{var},\, [0,T]}
     \;\xrightarrow[n\to\infty]{}
 0 
   \; \text{ in $L^1(\PP_0)$ and}
    \; \PP_0\text{-a.s.}
  \end{align*}
  Set 
  \begin{align*}
    \cQ^n_{s,t}(\Psi ,\Xi )\ldef \frac{1}{n}   \cQ_{ns,nt}(\Psi ,\Xi ).
  \end{align*}
  Note that $\cQ_{s,t} (\Psi, \Xi )$ is monotone in $t$ whenever $\Psi=\Xi$ and we get in this case
  \begin{align*}
    \Big \| \cQ^n(\Psi ,\Xi ) \Big\|_{1\text{-}\mathrm{var},\, [0,T]}
    \;=\;
    \frac{1}{n}   \cQ_{0,nT} (\Psi,\Xi ) 
    \;=\; \Big \| \cQ^n(\Psi ,\Xi ) \Big\|_{\infty\text{-var},\, [0,T]}.
  \end{align*}
  % 
  % To end the $L^1(\prob_0)$ convergence, it is enough to show tightness of
  %	\begin{align*}
  %   \Big \| Q^n(\Psi ,\Xi )  \Big\|_{p/2\text{-}\mathrm{var},\, [0,T]}.\end{align*}
  Since $Q_{s,t}(\cdot,\cdot)$ is a symmetric form. By polarization, for general $\Psi,\Xi$ it holds
  \begin{align*}
    \Big\|   \cQ^n (\Psi,\Xi)  \Big\|_{1\text{-}\mathrm{var},\, [0,T]}
    \le
    8 \Big \| \cQ^n (\Psi ,\Psi ) \Big\|_{\infty\text{-var},\, [0,T]} 
    + \Big \|
    \cQ^n (\Xi ,\Xi )  \Big\|_{\infty\text{-var},\, [0,T]}.
  \end{align*}
  since the two terms converge $\PP_0\text{-a.s.}$ their supremum over $n$ is bounded by some constant $C(\omega)$ $\PP_0\text{-a.s.}$.
  Therefore
  \begin{align*}
    \Big\|   A^n (\Psi,\Xi)  \Big\|_{1\text{-}\mathrm{var},\, [0,T]}
    &\le
    \Big\|   \cQ^n (\Psi,\Xi)  \Big\|_{1\text{-}\mathrm{var},\, [0,T]} 
    +
    T|\langle \Psi,\Xi \rangle_{L^2_\mathrm{cov}(\PP)}|\prob(0\in\mathcal{C}_\infty(\omega))\\
    & \le C(\omega) +T|\langle \Psi,\Xi \rangle_{L^2_\mathrm{cov}(\PP)}|\prob(0\in\mathcal{C}_\infty(\omega))
    =:\tilde C(\omega).
  \end{align*}
  Since for any $H$ and $p>2$
  \begin{align*}
    \big \| H \big\|_{p/2\text{-}\mathrm{var},\, [0,T]}
    \le
    \big \| H \big\|_{1\text{-}\mathrm{var},\, [0,T]}
    \big \|H \big\|_{\infty\text{-}\mathrm{var},\, [0,T]}^{\frac{p-2}{p}},
  \end{align*}
  we get
  \begin{align*}
    \Big \|  A^n \Big\|_{p/2\text{-}\mathrm{var},\, [0,T]}
    \le
    \tilde C(\omega) \,
    \Big \|  A^n \Big\|_{\infty\text{-}\mathrm{var},\, [0,T]}^{\frac{p-2}{p}}	
     \;\xrightarrow[n\to\infty]{\;\PP_0\text{-a.s.}\;}
    0.
  \end{align*}

  To see the convergence in $L^1(\PP_0)$ it is enough to show the tightness of the $p/2$ variation norms.

  Using the L\'evy system theorem \eqref{eq:Levy:system}, stationarity and Fubini's theorem, we obtain
  \begin{align*}
    \mean_0\! \Big[\frac{1}{n}   \cQ_{0,nT} (\Psi,\Xi )\indizero\Big] 
    &\;=\;
    \mean_0\! \Big[\frac{1}{n}   \int_0^{nT}\sum_{x\in\mathbb{Z}^{d}}\tau_{X_s}\omega(0,x)\Psi(\omega,x) \Xi(\omega,x) \indizero \md s \Big]. 
  \end{align*}  
  Hence
  \begin{align*}
    \EE_0\! \Big[\frac{1}{n}   \cQ_{0,nT} (\Psi,\Xi )\Big] 
    &\;=\;
    \EE_0\! \Big[\frac{1}{n}   \int_0^{nT}\sum_{x\in\mathbb{Z}^{d}}\tau_{X_s}\omega(0,x)\Psi(\omega,x) \Xi(\omega,x)]\\
    &\;=\;
    T \langle\Xi,\Psi\rangle_{L_{\mathrm{cov}}^2(\PP)}.
  \end{align*}
  Therefore
  \begin{align*}
    \EE_0 \Big[\Big\|  \cQ^n(\Psi,\Xi)  \Big\|_{1\text{-}\mathrm{var},\, [0,T]}\Big]
    \le
    8 T (\|\Xi\|_{L_{\mathrm{cov}(\PP)}^2}^2 + \|\Psi\|_{L_{\mathrm{cov}(\PP)}^2}^2).
  \end{align*}
  and the desired tightness with respect to
  $\| \cdot \|_{p/2\text{-}\mathrm{var},\, [0,T]}$	is implied. 

  \textbf{Step 2.} Assume that $\omega$ is $x$-periodic for some $x\ne 0$. This can be treated by enlarging the environment as in \cite[Execise 2.5]{Bi11}. Here are the details.
  Enlarge the space $\Omega$ of environments: Let $\Lambda=\{0,1\}^{E^d}$ and $E^d=\{\{x,y\}: |x-y|=1, x,y\in\mathbb{Z}^d\}$ is the edge set of the Euclidean lattice taken with product Bernoulli(1/2) measure $P$ and a corresponding expectation $E$. Take $\tilde\Omega=\Omega\times\Lambda$, with the product measure $\prob\times P$. Then $\prob\times P$ is stationary and ergodic with respect to translations $\{\tau_x : x \in \mathbb{Z}^{d}\}$ of $\,\mathbb{Z}^{d}$, where  
  $\tau_x(\omega,\lambda) := (\tau_x\omega,\tau_x\lambda)$ for 
  $(\omega,\lambda)\in \tilde\Omega$.
  Then, a version of Lemma~\ref{lemma:inv:environment_process} holds in these settings, that is $t\mapsto \tilde\omega_t:=(\tau_{X_t}\omega,\tau_{X_t}\lambda)$ is a reversible ergodic Markov process on $\tilde\Omega$ with respect to $\prob\times P$. 
  Note that $(\omega,\lambda)$ is non-periodic $\prob\times P$-a.s, that is 
  \[
  \prob\times P ( \tau_x(\omega,\lambda)= (\omega,\lambda))=0\quad \text{ for all } x\in\mathbb{Z}^d\backslash\{0\}.  
  \]
  Therefore Lemmas \ref{lemma:ergodic} and \ref{lem:ergodic in uniform norm} apply when $\Omega$ is replaced by $\tilde\Omega$, and hence also Corollary~\ref{cor:conditional-ergodic-in-suo-norm}.  
  We can now go back to \textbf{Step 1} with the modification $\tilde f$ of $f$ as follows. 
  \begin{align*}
    \tilde f ((\omega,\lambda),(\omega',\lambda')) \coloneqq
    \Psi\big(\omega, \eta (\lambda,\lambda') \big)\,
    \Xi\big(\omega,\eta (\omega,\omega')\big)\indizero.
  \end{align*}
  Therefore we obtain 
  \begin{align*}
    \big \| A^n \big\|_{p/2\text{-}\mathrm{var},\, [0,T]}
   \;\xrightarrow[n\to\infty]{} 
   0
    \text{ in $L^1(\PP_0\times P)\;$  and}
    \; \PP_0\times P\text{-a.s.}
  \end{align*}
   In particular the convergence holds also $\PP_0\text{-a.s.}$. 
   We are left to show the convergence also holds in $L^1(\PP_0)$. Note that for a fixed $\omega$
    \begin{align*}
    \sum_{0 < s \leq tn} \tilde f(\tau_{X_{s-}} (\om,\lambda), \tau_{X_s} (\om,\lambda) )
    \;=\;
    \cQ_{0,tn}(\Psi,\Xi)
    \indizero\, \indicator_{\lambda\text{ is non-periodic}}.  
  \end{align*}
Independence then yields 
  \begin{align*}
    \EE_0 [\big \| A^n \big\|_{p/2\text{-}\mathrm{var},\, [0,T]}]   
    =
    \EE_0\times E \bigg[\big \| A^n \big\|_{p/2\text{-}\mathrm{var},\, [0,T]}\,|\,\lambda\text{ is non-periodic}\bigg]
    \;\xrightarrow[n\to\infty]{} 
    0.
    \end{align*}
In particular, the convergence holds also in $L^1(\PP_0)$ and the proof is completed.
\end{proof}

We close the appendix with a formulation of Slutsky's Theorem \cite{Slutsky1925} that fits our settings. 
\begin{theorem}(Slutsky's Theorem in Skorohod Topology)\label{thm:slutsky}
  Let $\mathbf{X}^n=(X^n,\bbX^n)$ and $\mathbf{Y}^n=(Y^n,\bbY^n)$ be sequences of random elements of $D\ldef D([0,T],\mathbb{R}^d\times \mathbb{R}^{d\times d})$. Suppose that
  $\mathbf{X}^n \Rightarrow \mathbf{X}$ in $D$, where $\mathbf{X}=(X,\bbX)$ has continuous sample paths almost surely 
  and $\mathbf{Y}^n \xrightarrow{\;\mathbb P\;} (0,f)$ in D, where $f\in C([0,T],\mathbb{R}^d)$ is deterministic and continuous.
  Then
  \begin{align*}
    \mathbf{X}^n  + \mathbf{Y}^n  \;\Rightarrow\; \mathbf{X} + (0,f) = (X,\bbX+f) \quad\text{in } D.
  \end{align*}
\end{theorem}

Remark about notation: by 
$\mathbf{X}=(X,\bbX) \in D\ldef D([0,T],\mathbb{R}^d\times \mathbb{R}^{d\times d})$, 
we mean that $\mathbf{X_t}=(X_t,\bbX_t)\in \mathbb{R}^d\times \mathbb{R}^{d\times d}$ for all $t\in[0,T]$ and is a c\`adl\`ag function. By the sum of the functions we mean a point-wise sum, e.g.
\begin{align*}
  (\mathbf{X}^n  + \mathbf{Y}^n)_t = \mathbf{X}^n_t  + \mathbf{Y}^n_t =(X_t+Y_t, \bbX_t +\bbY_t)
  \quad \text{    
    and
  } 
  \quad
  (X,\bbX+f)_t=(X_t,\bbX_t+f(t)).
\end{align*}

We shall now give a sketch of the proof, see also \cite[Theorem B.1]{orenshtein2021rough}.
%%%%%%%%% 
%%%%%%%%%%%%%%% 
\begin{proof}[Sketch of proof] 
  First note $\mathbf{Y}^n\xrightarrow{\;\mathbb P\;} (0,f)$ in probability in the uniform topology by continuity. Then there is a joint convergence $(\mathbf{X}^n,\mathbf{Y}^n)\Rightarrow (\mathbf{X}^n,(0,f))$ in $D\times D$. Then observing the function $(\mathbf{X},\mathbf{Y})\mapsto \mathbf{X}+\mathbf{Y}$ is continuous at the point $(\mathbf{X},\mathbf{Y})$ whenever both $\mathbf{X}$ and $\mathbf{Y}$ are continuous functions since uniform convergence implies convergence in $D\times D$. Then conclude by the continuous mapping theorem using continuity of addition at continuous pairs.
\end{proof}

\bibliographystyle{plain}
\bibliography{literature}

\end{document}